\documentclass[aop]{imsart}

\RequirePackage{amsthm,amsmath,amsfonts,amssymb}
\RequirePackage[numbers]{natbib}
\RequirePackage[colorlinks,citecolor=blue,urlcolor=blue]{hyperref}
\usepackage{graphicx}

\usepackage{enumerate}
\usepackage{float}
\usepackage{centernot}

\startlocaldefs
\numberwithin{equation}{section}
\theoremstyle{plain}
\newtheorem{theorem}{Theorem}[section]
\newtheorem*{definition}{Definition}
\newtheorem{lemma}[theorem]{Lemma}
\theoremstyle{remark}
\newtheorem*{remark}{Remark}
\allowdisplaybreaks
\usepackage{tikz}
 
\endlocaldefs

\newcommand{\abs}[1]{\left| #1 \right|}
\newcommand{\norme}[1]{\left\| #1 \right\|}
\newcommand{\N}{\mathbb{N}}
\newcommand{\Z}{\mathbb{Z}}

\newcommand{\Proba}{\mathbb{P}}
\newcommand{\Ent}[1]{\left\lfloor #1\right\rfloor}
\newcommand{\Ceil}[1]{\left\lceil #1\right\rceil}
\newcommand{\demi}{\frac{1}{2}}
\newcommand{\eqninfty}{\stackrel{n\to\infty}{\sim}}
\newcommand\numberthis{\addtocounter{equation}{1}\tag{\theequation}}
\newcommand{\diam}{\mathrm{diam}}
\newcommand{\quadet}{\quad\text{and}\quad}
\newcommand{\qquadet}{\qquad\text{and}\qquad}
\newcommand{\quadou}{\quad\text{where}\quad}
\newcommand{\qquadou}{\qquad\text{where}\qquad}
\newcommand{\quador}{\quad\text{or}\quad}

\newcommand{\quadavec}{\quad\text{with}\quad}
\newcommand{\qquadavec}{\qquad\text{with}\qquad}

\renewcommand{\thesection}{\arabic{section}}
\newcommand{\guillemets}[1]{``#1''}

\newcommand{\calB}{\mathcal{B}}
\newcommand{\calC}{\mathcal{C}}

\newcommand{\calE}{\mathcal{E}}
\newcommand{\calF}{\mathcal{F}}

\newcommand{\calH}{\mathcal{H}}

\newcommand{\calJ}{\mathcal{J}}

\newcommand{\calL}{\mathcal{L}}
\newcommand{\calM}{\mathcal{M}}
\newcommand{\calN}{\mathcal{N}}
\newcommand{\calO}{\mathcal{O}}
\newcommand{\calP}{\mathcal{P}}

\newcommand{\calR}{\mathcal{R}}
\newcommand{\calS}{\mathcal{S}}
\newcommand{\calT}{\mathcal{T}}

\newcommand{\calY}{\mathcal{Y}}

\newcommand{\bE}{\mathbb{E}}
\newcommand{\s}{\mathfrak{s}}
\newcommand{\PlA}{\Proba}
\newcommand{\Pml}{\calP_{\mu}^{\lambda}}
\newcommand{\Esp}{\mathbb{E}}
\newcommand{\ElA}{\Esp}
\newcommand{\torus}{{\Z_n^d}}
\newcommand{\dist}{x^\star}
\newcommand{\reff}[1]{(\ref{#1})}
\renewcommand{\epsilon}{\varepsilon}
\date{\today}

\begin{document}

\begin{frontmatter}

\title{The Critical Density for Activated Random Walks is always less than 1}
\runtitle{Critical Density for ARW}

\begin{aug}
\author[A]{\fnms{Amine}~\snm{Asselah}\ead[label=e1]{amine.asselah@u-pec.fr}},
\author[B]{\fnms{Nicolas}~\snm{Forien}\ead[label=e2]{nicolas.forien@uniroma1.it}}
\and
\author[C]{\fnms{Alexandre}~\snm{Gaudilli\`ere}\ead[label=e3]{alexandre.gaudilliere@math.cnrs.fr}}
\address[A]{LAMA, UPEC, UPEM, CNRS, F-94010, Cr\'eteil, France and NYU Shanghai\printead[presep={,\ }]{e1}}

\address[B]{Sapienza Universit\`a di Roma, Dipartimento di Matematica, Roma, Italy\printead[presep={,\ }]{e2}}

\address[C]{Aix Marseille Univ, CNRS, I2M, Marseille, France
\printead[presep={,\ }]{e3}}
\end{aug}

\begin{abstract}
{\it Activated Random Walks}, on~$\Z^d$ for any~$d\geqslant 1$, 
is an interacting particle system, where particles can be in
either of two states: active or frozen. Each active particle
performs a continuous-time simple random walk
during an exponential time of parameter~$\lambda$, after which
it stays still in the frozen state, until another
active particle shares its location, and turns it instantaneously 
back into activity. This model is known to have a phase transition,
and we show that the critical density,
controlling the phase transition, is less than one in any
dimension and for any value of the sleep rate~$\lambda$. We provide upper bounds for the critical density
in both the small~$\lambda$ and large~$\lambda$ regimes.
\par\noindent MSC 2020 \emph{subject classifications.} 60K35, 82B26.
\vspace{0.2cm}
\end{abstract}



\end{frontmatter}

\section{Model and results}
\label{sectionIntro}

\subsection{Activated Random Walks}

This paper is a companion to~\cite{FG22}.
We continue our study of a specific reaction-diffusion model 
known as {\it Activated Random Walks} (ARW) invented to study 
{\it self-organized criticality}.
Informally, random walks diffuse on a graph which has a tendency to hinder
the motion of lonely walkers, whereas the vicinity of other diffusions turns
hindered particles into diffusive ones. 
Here, we consider
the Euclidean lattice~$\Z^d$ in any dimension, or a large Euclidean torus.
The initial configuration is
an independent Poisson number of {\it particles} at each site
of~$\Z^d$, with parameter~$\mu<1$. Each particle can be in any
of two states: {\it active} or {\it frozen} (or {\it sleeping}).
Each active particle performs a continuous-time simple random walk with rate~$1$ and is equipped
with an independent exponential clock of parameter~$\lambda$,
at the marks of which the particle changes state, and
stops moving. When a frozen particle shares a site with
another particle, it gets instantly activated.
When the graph is an infinite Euclidean lattice,
and one increases the initial density of active particles,
one expects to see a transition,
at a critical density~$\mu_c(\lambda)$,
from a regime of {\it low density}
where particles are still
to a {\it high density} regime
of configurations made of constantly evolving islands
of sleeping particles at low density
in a sea of diffusing particles at high density.
When we start with a large number of active particles
at the origin we expect a large ball to be eventually covered 
at density~$\mu_c(\lambda)$.
This phenomenon is known as {\it self-organized criticality}, in the
sense that the system alone reaches a critical state.
This notion was introduced in the eighties
by Bak, Tang and Wiesenfeld~\cite{BTW87}
together with a related toy model, {\it the abelian sandpile}.
The ARW model, which is less constrained,
was actually popularized
some 13 years ago by our late friend Vladas Sidoravicius
and we refer to Levine and Liang~\cite{LL21} 
for some comparison between the two models.
The ARW model shares with the abelian sandpile the nice feature
that the order in which the particles are launched is irrelevant,
which is known as the abelian property.

When working with~$\Z^d$, and when active particles are drawn from a 
product Poisson measure of intensity~$\mu$ at each site, 
ARW is known to have
a phase transition between an {\it active phase} and a
{\it frozen phase}. The active phase is characterized by
every vertex being visited infinitely many times, whereas in the frozen phase the origin is visited a finite number of times.
In a seminal work Rolla and Sidoravicius~\cite{RS12}
prove that the system stays active
forever with a probability which is increasing in~$\mu$, and which
satisfies a zero-one law under~$\Pml$, law of the process when the sleep rate is~$\lambda$ and
the initial configuration is drawn from 
the product Poisson measure of intensity~$\mu$. Thus, the
following density~$\mu_c(\lambda)$ is well defined:
$$\mu_c(\lambda)
\ =\ \inf\,\Big\{\,\mu\,:\,\Pml(\text{the origin is visited a finite number of times})=0\,\Big\}\,.$$
In~\cite{RSZ19} Rolla, Sidoravicius and Zindy show that~$\mu_c(\lambda)$ 
is the same number when the initial configuration is drawn from any
translation-invariant ergodic measure with mean~$\mu$.
On~$\Z^d$ with~\smash{$d\geqslant 3$}, Stauffer and Taggi in~\cite{ST18}
show that when~$\lambda$ is small,~$\mu_c(\lambda)<1$, and they provide a general lower bound~$\mu_c(\lambda)\geqslant\lambda/(1+\lambda)$, which is valid in any dimension and for every~\smash{$\lambda>0$}.
In a subsequent work~\cite{Taggi19} Taggi shows 
that~$\mu_c(\lambda)<1$ for all~$\lambda\in (0,\infty)$ on~$\Z^d$ with~\smash{$d\geqslant 3$}, and provides an upper bound on the critical density, showing that~$\mu_c(\lambda)\leqslant C_d\sqrt{\lambda}$ for every~$\lambda>0$, for some positive constant~$C_d$.

Even in dimension one, ARW is far from trivial.
In~$d=1$, Basu, Ganguly and Hoffman introduce in~\cite{BGH18}
a {\it block dynamics} allowing them to replace the complex correlation of
the odometer function
(measuring the number of instructions used at each site) by some
balance equations at the end-points of their blocks (where particles
leave) and at the centers (where particles arrive).
They obtain that~$\mu_c(\lambda)<1$ for small~$\lambda$. 
Following the same approach, Asselah, Rolla and Schapira 
show in~\cite{ARS19}
that~$\mu_c(\lambda)=\calO(\sqrt \lambda)$ for small~$\lambda$,
and Hoffman, Richey and Rolla show in~\cite{HRR20} that,
for any~$\lambda$,~$\mu_c(\lambda)<1$ in dimension 1. 
Finally, let us mention that a lot of works have
considered asymmetric random walks, where~$\mu_c(\lambda)<1$ has
been settled. In dimension~$2$, two independent recent works by 
Forien and Gaudilli\`ere~\cite{FG22} on the one hand, and Y.\ Hu~\cite{Hu22}
on the other hand, have established that~$\mu_c(\lambda)<1$ when~$\lambda$ is small enough.

The family of ARW models
is very rich as we vary~$\lambda$ from~$0$ to~$\infty$,
and as we vary the initial conditions with active and frozen
particles. Let us illustrate this with two examples.
When~$\lambda=0$, the
active particles stay alive forever: if we start with a product
Poisson distribution of sleeping particles and one active
particle at the origin, the model is known as the frog model.
Kesten and Sidoravicius have also studied a model for the propagation of an infection, where the
sleeping particles can move at a slower rate than the active ones.
When~$\lambda=\infty$,
and we send active particles from the origin, then this is the
celebrated model called internal 
diffusion limited aggregation (IDLA). 
This latter model is much older, and in a sense much simpler since
frozen particles remain so forever. It has been thoroughly studied,
and a shape theorem has been obtained in the nineties by
\cite{LBG92} on~$\Z^d$, and on a few other interesting
graphs, as well as for closely related
variants: uniform IDLA~\cite{BDKL20}, 
Hasting-Levitov dynamics~\cite{NST21},
rotor-router~\cite{LP09} and
divisible sandpiles. Besides, fluctuations around the typical shape have
been obtained on~$\Z^d$ independently by Asselah and Gaudilli\`ere
in~\cite{AG13} and by Jerison, Levine and Sheffield in~\cite{JLS12}.
It remains the focus of recent interest~\cite{BG22}.

We show in this paper that when we start with particles which are all active
on the torus~\smash{$\torus:=(\Z/n\Z)^d$}, for any dimension~$d$ and 
$\lambda>0$, there is a density~$\mu<1$, independent of~$n$, 
above which the system remains active during an exponentially large
time (in~$|\torus|$) with overwhelming probability.
This implies that there exists a non-trivial active phase
on the infinite Euclidean lattice
for all sleep rates~$\lambda$.
In other words,
when we start the system on~$\Z^d$ with active particles distributed as a product
measure with more than~$\mu_c(\lambda) < 1$ 
particles per site on average,
then the origin is almost surely visited infinitely many times.

\subsection{Main results}

Our main result is the following.

\begin{theorem}\label{thmAllD}
In any dimension~$d\geqslant 1$, for every sleep rate~$\lambda>0$, the critical density of the Activated Random Walks model on~$\Z^d$ satisfies~$\mu_c(\lambda)<1$.
\end{theorem}
This result is new in~$d=2$ and our proof
encompasses all dimensions (with small changes).
Since it is constructive, it does provide upper bounds on~$\mu_c(\lambda)$:
these bounds are new in~$d=2$, and improve existing bounds for~$d\geqslant 3$.
Note that we cover the regime of large~$\lambda$ as well
as small~$\lambda$.
Theorem~\ref{thmAllD} follows from the upper bounds
we now present (and from the fact that the critical density~$\mu_c$ is non-decreasing in~$\lambda$, see for example~\cite{Rolla20}).

\begin{theorem}\label{thm2D}
In dimension~$d=2$, there exists~$a>0$ such that, for~$\lambda$ small enough,
\begin{equation}
\label{bound2Dlow}
\mu_c(\lambda) \ \leqslant \ \lambda\,|\ln\lambda|^a\,,
\end{equation}
and there exists~$c>0$ such that, for~$\lambda$ large enough,
\begin{equation}
\label{bound2Dhigh}
\mu_c(\lambda) \ \leqslant \ 1-\frac{c}{\lambda(\ln\lambda)^2}\,.
\end{equation}
\end{theorem}

\begin{theorem} \label{thm3D}
In dimension~$d\geqslant 3$, there exists~$c=c(d)>0$ such that, for~$\lambda$ small enough,
\begin{equation}
\label{bound3Dlow}
\mu_c(\lambda)
\ \leqslant \ c\,\lambda\,,
\end{equation}
and there exists~$c=c(d)>0$ such that, for~$\lambda$ large enough,
\begin{equation}
\label{bound3Dhigh}
\mu_c(\lambda)
\ \leqslant \ 1-\frac{c}{\lambda\ln\lambda}\,.
\end{equation}
\end{theorem}
These bounds, in the two regimes of sleep rate, are of
a correct nature (up to some logarithmic factor in~$\lambda$), as we
try to justify in our heuristic discussion below.

\begin{remark}
Our proof method also works to show that~$\mu_c(\lambda)<1$ for every~\smash{$\lambda>0$} in dimension~$1$,
with minor changes, but in this setting it does not yield significantly new bounds on the critical density, 
hence we choose not to detail this case.
We limit ourselves to some comments in section~\ref{subsectionDim1} about how to adapt our proof to the one-dimensional case.
\end{remark}

\subsection{Heuristics}

At a heuristic level, the critical density of {\it frozen particles}
can be thought of as follows. Consider a configuration of
frozen particles drawn from a Poisson product measure
on~$\Z^d$ with density~$\mu$, and launch an active
particle at the origin. The density~$\mu$ equals~$\mu_c(\lambda)$ if
an active particle, in its journey before freezing,
encounters on average exactly one sleeping particle.
In other words, the number of active particles should strike a balance:
one particle wakes up when another one freezes.
Thus, if~$\calR_t$ is the number of distinct visited sites
in a time period~$[0,t]$ by a continuous-time
random walk, and if~$\tau$ is an independent
exponential time of mean~$1/\lambda$, we expect that
\begin{equation}\label{conj-1}
\bE[\calR_\tau]\cdot \mu_c(\lambda)
\ \simeq\ 1\,.
\end{equation}
The symbol~$\simeq$ is here to remind the reader that this is
just heuristics.
Now, there are two regimes: small~$\lambda$ where~$\tau$ is of
order of its mean~$1/\lambda$ which is large, and the regime of large~$\lambda$ where the particle makes a jump with probability~$1/(1+\lambda)$ which is small.
In the case of small~$\lambda$, we rewrite~\reff{conj-1} as~$\bE[\calR_{1/\lambda}]\cdot \mu_c(\lambda)\simeq 1$,
and we only need to recall the large time asymptotics of the range of a random walk (see for example~\cite{DE51}):
\[
\bE[\calR_t]=
\left\{
\begin{array}{ll}
\calO(\sqrt t) & \text{if }d=1 \\
\calO(\frac{t}{\ln t})& \text{if }d=2\\
\calO(t)& \text{if }d\geqslant 3.\\
\end{array} 
\right.\]
This implies the following heuristics for~$\mu_c(\lambda)$ for small~$\lambda$:
$$\mu_c(\lambda)=
\left\{
\begin{array}{ll}
\calO( \sqrt \lambda) & \text{if }d=1 \\
\calO(\lambda|\ln\lambda|)& \text{if }d=2\\
\calO(\lambda)& \text{if }d\geqslant 3.\\
\end{array}
\right.$$
When~$\lambda$ is large, the active particle makes one jump
with probability~$1/(1+\lambda)$ so that
\[
E[\calR_{\tau}]
\ \simeq\ 1+\frac{1}{\lambda}\,,
\]
which, plugged into~\reff{conj-1},
suggests that~$\mu_c(\lambda)\simeq 1-1/\lambda$ when~$\lambda\to\infty$.

Establishing these bounds remains a challenging problem,
as well as establishing some shape theorem, or understanding
ARW at the critical density. The model of ARW presents many 
other interesting questions, and we
refer to Rolla's survey~\cite{Rolla20} for a nice review. One difficulty
is that the time a particle stays in one of its two states
actually depends on the {\it local density} of particles which itself
changes with time: if an active particle travels amidst a region
of {\it high density}, then it most likely remains active as long
as it remains inside this region; instead, if it crosses a 
{\it low density} region, it most likely switches to a frozen state
at the first mark of its exponential clock.

\subsection{Sketch of the proof of Theorem~\ref{thm2D}}
\label{sec-sketch}

Most of our work is devoted to proving Theorem~\ref{thm2D} (the case~$d=2$), which is our main result, while Theorem~\ref{thm3D} (the transient case~$d\geqslant 3$) requires much less technology and simply follows as a by-product of an intermediate Lemma.

We now describe informally the six steps of our strategy, of which three are taken from~\cite{FG22}, and three are new.
The new ideas, namely the {\it dormitories}, the {\it ping-pong rally} and the {\it coloured loops}, are all of a hierarchical nature.

To show that it takes an exponentially large time to stabilize a configuration on the torus, we introduce a hierarchical structure on the set where the particles eventually settle, which we call the {\it hierarchical dormitory}.
With this construction, we first show that some elementary blocks of this hierarchy, called the {\it clusters}, have a stabilization time exponentially large in their size.
We then perform an induction using a toppling strategy  which we call {\it the ping-pong rally}, where neighbouring clusters interact and reactivate each other many times, leading to a stabilization time for their union which is roughly the product of the individual stabilization times.
Thus, we obtain that at each {\it space scale}, the stabilization time of a {\it cluster} is of order an exponential in the volume of the cluster.
The {\it coloured loops} are the last important ingredient in our proof: modifying slightly the dynamics of the ARW model by ignoring some reactivation events, we are able to obtain some independence between the different levels of the hierarchy, which turns out to be crucial in our inductive proof.

\subsubsection{Working on the torus and using the abelian property}
In~\cite{FG22} it is
shown that~$\mu_c<1$ if, when starting with a density~$\mu<1$
of active particles, the time needed to stabilize ARW on the torus~$\torus=(\Z/n\Z)^d$ is exponentially large in~$n$ with high probability.
Thus, we always consider~$\torus$ with~$n$ large enough.

We recall that there is a celebrated graphical representation, known as the site-wise or Diaconis-Fulton representation (see~\cite{DF91} for the original construction or~\cite{Rolla20} for a nice presentation in the context of ARW), where we
pile stacks of independent instructions on top of each site of~$\torus$, and use these instructions one after another to move the particles.

A key property of this representation of the model, known as the {\it abelian property}, is that the final configuration and the number of steps performed (these steps are called {\it topplings}) do not depend on the order with which the instructions are used, allowing us to choose an arbitrary strategy to move particles.
We rely heavily on this property throughout our work (although our graphical representation, described below in section~\ref{subsectionLoopRep} is no longer abelian, its construction itself relies on the abelian property of the ARW model).

\subsubsection{Decomposition over all possible settling sets}
Whereas~\cite{FG22} focuses on the case where the sleep rate~$\lambda$ is small,
in our case we fix~$\lambda$, which can be either small or large,
and we look for a density~$\mu<1$ so that the stabilization time on the torus is exponentially large.
Starting from a fixed initial configuration with~$\mu n^d$ particles, we decompose the probability to stabilize the configuration 
in a given time, summing over all the possible sets 
where the particles can settle.

Then, as in~\cite{FG22}, for a fixed couple~$(\lambda,\,\mu)$ we look for an estimate on the stabilization time which is uniform over all possible settling sets~$A\subset\torus$, and we perform a union bound.
Hence, in our estimate we get a combinatorial factor~\smash{$\binom{n^d}{\mu n^d}$} corresponding to the number of possible settling sets.
This factor can be thought of as an \guillemets{entropic} term, which we have to outweigh by an \guillemets{energy} term corresponding to the probability to stabilize in a given set~$A\subset\torus$ in a short time.
Since this entropic factor~\smash{$\binom{n^d}{\mu n^d}$} gets smaller when~$\mu$ is either close to~$0$ or close to~$1$, we concentrate on these two distinct regimes of the sleep rate, namely~$\lambda\to 0$ and~$\lambda\to \infty$, with a corresponding density~$\mu\to 0$ or~$\mu\to 1$.

The use of a uniform estimate over settling sets is responsible for some logarithm factors appearing in the bounds that we obtain on the critical density~$\mu_c$.
This can be seen in the statement of Lemma~\ref{lemmaConditionActivityTopplings} (where~$\psi(\mu)$ corresponds to the entropy, while~$\kappa$ controls the energy) and in the final estimates for our proofs in sections~\ref{section2D} and~\ref{section3D}, where we tune the parameters~$\lambda$ and~$\mu$ such that the entropy-energy balance is favourable.
Thus, a possible direction to improve our estimates on~$\mu_c$ could be to refine this union bound by ruling out some sets~$A$ on which it is very unlikely that the particles settle.

\subsubsection{Reduction to a model with density~1 on the trace graph}
Once such a set~$A\subset\torus$ is fixed, we look for an upper bound on the probability that the particles settle on~$A$ in a given time.
To this end, as explained in section~3.1 of~\cite{FG22}, we suppress all sleeping instructions on~$\torus\setminus A$.
Indeed, provided that particles eventually settle in~$A$, these sleeping instructions are overridden at some time or another.
Then, using the abelian property of the model, we may first let each particle move until it reaches an empty site of~$A$.
Thus, we may start from the configuration with exactly one active particle on each site of~$A$ (see section~3.2 of~\cite{FG22}).
Thus, we end up with a simplified model on a fixed subset~$A$ that we call the dormitory, which starts fully occupied with active particles.
These active particles cannot settle anywhere but on~$A$.

Let us now describe the order with which the particles move.
At the beginning of each step,
we choose an active particle and read its first unused instruction.
If it is a {\it sleep} instruction, 
the step is over, the particle falls asleep and we choose another particle at the next step.
Otherwise, if it is a {\it jump} instruction, 
we let the particle jump, and follow instructions along its path until it goes
back to its starting point. Indeed, when the particle is not at its starting position,
it is either outside of~$A$, and there is no sleep instruction, or
it is on top of another particle, and the sleep instructions have no effect.
Note that at the end of its {\it loop}, the particle is active, and has
waken up all sleeping particles along the loop. Thus, each step of the dynamics consists either in a sleep event or in drawing a loop (that is to say, the support of an excursion) from
a site with an active particle, 
and updating the set of active particles.

Doing so, after each step (i.e., after a sleep or a loop) we go back to a configuration where there is exactly one particle on each site of~$A$.
Thus, in a way, we reduced the problem to the study of an easier model with density~$1$,
but on the modified graph which is the trace graph on~$A$,
that is to say, the graph whose vertex set is~$A$ and where the transition probability from~$x$ to~$y$
is the probability that~$y$ is the first site of~$A$ encountered by a random walk on the torus
that has just jumped out of~$x$,
as if particles were \guillemets{sliding} on~$\torus\setminus A$ with infinite speed until they reach a site of~$A$.

This idea to reduce the model to the case of density~$1$ is a key idea which is at the core of both~\cite{FG22} and the present work.
Once constrained to this settling set~$A$ where there is just enough space for all the particles to fixate, it is very difficult for the model to reach the stable state where all the particles are sleeping.
Hence, a phenomenon of metastability is expected, the system remaining trapped in a situation where only a fraction of the particles are sleeping, with a huge potential barrier to overcome to bring all particles to a rest.

Thus, it is not surprising that, for every fixed settling set~$A$, the model where particles are forced to settle on~$A$ takes an exponentially large time to reach its stable state, this time being roughly distributed as a geometric random variable.
Many trials are necessary before overcoming the drift.

Therefore, if one can prove that, for a general class of graphs, the model with density~$1$ takes an exponential time to stabilize, there only remains to see if the combinatorial factor~\smash{$\binom{n^d}{\mu n^d}$} corresponding to the choice of the dormitory~$A$ can be outweighed by the estimate on the exponential fixating time of the density-one model on~$A$.

\subsubsection{Hierarchical Dormitories}
Once the settling set~$A$ is fixed, we introduce a deterministic hierarchical structure on~$A$, which is a new ingredient compared to~\cite{FG22}.
This structure consists in a finite decreasing sequence of subsets of~$A$, 
say~$A_0\supset A_1\supset\cdots\supset A_{\calJ}$, 
and a corresponding sequence of partitions~$\calC_0,\,\calC_1,\,\dots,\,\calC_{\calJ}$ (where~$\calC_j$ is a partition of~$A_j$) 
whose elements are called clusters (even though they are not necessarily connected).

The construction of the~$0$-th level of the hierarchy varies depending on which regime we are studying.
The idea is that these clusters are densely connected in some sense, so that the stabilization of any cluster~$C\in\calC_0$ produces a number of loops which is exponentially large in the size of~$C$, so that these clusters can interact with other clusters of~$\calC_0$ which are far apart.

In the regime of small~$\lambda$ (thus with~$\mu\to 0$), that is to say, for the proof of the bound~\reff{bound2Dlow}, the set~$A$ is rather sparse, so we simply take~$A_0=A$ and~$\calC_0$ composed only of singletons.
Indeed, the stabilization of a singleton already emits a number of loops which is geometric with mean~$1/\lambda$, which is large when~$\lambda\to 0$.
In the proof of the bound~\reff{bound2Dhigh} in the large~$\lambda$ regime,
the first partition~$\calC_0$ is composed of clusters of {\it high density}, so high that if at least a fraction of the particles in a cluster are frozen, then loops emanating from well chosen sites have a tendency to wake up many particles.

Then, in both regimes, clusters of~$\calC_1$ are obtained by pairing clusters of~$\calC_0$ when their distance is not too large.
Then, by way of induction at each step~$j$, we construct each partition~$\calC_j$ by merging pairs of clusters in~$\calC_{j-1}$ which are not too far apart.
In so doing, we merge as many pairs as possible, but possibly throw away clusters which are isolated.
Also, at each level some clusters are not merged, but we impose a minimal size for clusters at each level, so that the clusters get bigger and bigger along the hierarchy.
The construction stops when we obtain a partition~$\calC_{\calJ}$ which contains one single set.
We then have to check that we have not thrown away too much, so that this last cluster on top of the hierarchy contains at least a fraction of the initial set~$A$.

The detailed construction of this hierarchy is presented in section~\ref{sectionHierarchy}.

\subsubsection{Ping-pong rallies}
Let us now explain how we control the stabilization time.
In a first step, we prove that, for every cluster~$C\in\calC_0$, the number of topplings to stabilize the configuration on~$C$ is exponentially large in the size of~$C$.
This is done in section~\ref{sectionInit} by using that the number of sleeping particles has a negative drift (when at least a fraction of the cluster is sleeping), implying through a martingale argument that the stabilization time of~$C$ dominates a geometric random variable with exponentially large mean, with an explicit control on the parameter in the exponential (which is important to obtain our explicit bounds on~$\mu_c$).

Then, at each level of the hierarchy, we introduce {\it ping-pong rallies}.
We prove by induction on~$j$ that, for every~$C\in\calC_j$, the number of topplings necessary to stabilize~$C$ is exponentially large in the size of~$C$. 
We insist that at each {\it space scale} we need to control the whole law of the stabilization time, not just the tail.
We now present a mechanism behind the exponential fixation time.

Consider a cluster~$C = D \cup E$ in~$\mathcal{C}_{j + 1}$ 
with~$D$ and~$E$ in~$\mathcal{C}_j$.
Starting with~$D$ and~$E$ fully active
we perform sleeps and loops on each active site in~$D$
up to the full stabilization of~$D$, before doing the same in~$E$ to reach the full stabilization of~$E$.
Now, after these two rounds, some sites in~$D$ may
have been reactivated during the stabilization of~$E$. If~$D$ is
not too far from~$E$,~$D$ has great chances to be actually fully 
reactivated.  We then stabilize~$D$ again, which in turn reactivates~$E$
and so on and so forth up to the complete stabilization of~$C$.
We say that our merging clusters play a ping-pong rally
which ends when one cluster stabilizes
without reactivating {\it all the particles} of its playing partner.

The ping-pong rally is behind the reinforcement of activity.
Indeed, let~$t_D$ and~$t_E$ be the expected values of~$T_D$ and~$T_E$,
the random numbers of loops needed to
stabilize~$D$ and~$E$ respectively.
Let also~$\epsilon_D$ and~$\epsilon_E$
be the probabilities that~$E$ and~$D$
are not completely reactivated 
during the stabilization of~$D$ and~$E$, respectively.
The expected total number of excursions needed to stabilize~$C$
is then larger than or equal to 
$$
    \sum_{k \geqslant 0}\,\bigl[(1 - \epsilon_E) 
(1 - \epsilon_D)\bigr]^k (t_D + t_E)
\ =\ \frac{t_D + t_E}{\epsilon_D+\epsilon_E-\epsilon_D\epsilon_E}
    \ \geqslant\ \frac{t_D + t_E}{ \epsilon_D  + \epsilon_E}\,.
$$
Having reduced our analysis to the density-one,
hence metastable, systems~$D$ and~$E$,
we can expect~$T_D$ and~$T_E$ to be approximate geometric random variables
with success probability~$1 / t_D$ and~$1 / t_E$.
Having also chosen~$D$ and~$E$ close enough for them to merge at scale~$j + 1$,
we can also expect~$\epsilon_D$ and~$\epsilon_E$
to be of order~$1 / t_D$ and~$1 / t_E$ at most.
In metastable situations indeed, 
we can expect the thermalisation times to metastable equilibria
to be small with respect to the mean stabilization times.
In our case the latter,~$t_D$ and~$t_E$,
should be exponentially large in~$|D|$ and~$|E|$,
while the former should be only polynomial in~$|D|$ and~$|E|$.
For the ping-pong rally to stop,
a cluster should then essentially stabilize
within its thermalisation time to metastable equilibrium:
if not, it will produce an exponentially large number of loops
that will continue the ping-pong rally with very large probability.
Since the stabilization time
when starting from a fully active cluster 
will dominate the geometric stabilization time
when starting from metastable equilibrium,
$\epsilon_E$ and~$\epsilon_D$ should be
of order~$1 / t_D$ and~$1 / t_E$
up to logarithmic corrections at most.
This would give a lower bound for the mean number~$t_C$ 
of the needed excursions to stabilize~$C$ of order
$$
    \frac{t_D + t_E }{t_D^{-1} + t_E^{-1}}\ =\ t_D\times t_E\,.
$$
If our induction hypothesis assumes that~$t_D$ and~$t_E$
are exponentially large in the size of~$D$ and~$E$,
we would obtain from these heuristics that~$t_C$ is exponentially large 
in the size of~\smash{$D \cup E = C$}.

\subsubsection{Coloured Loops} 
The technical difficulties lie in the need to control the whole law of the stabilization time at each scale (and not only its expectation or its tail) and in the intricate dependence relation between the length of the ping-pong rally (i.e., how many times the sets~$D$ and~$E$ fully reactivate each other) and the duration of the successive stabilization steps of the rally.
It is not clear how the stabilization time of one set, say~$D$, is correlated with the event that the other set~$E$ is fully reactivated during this stabilization of~$D$.
Intuitively, knowing that the stabilization of~$D$ takes a long time, we have many loops emerging from~$D$ which can reactivate the sites of~$E$, but we also have some information on the shapes of these loops, namely that they tend to visit many sites of~$D$.

To overcome this issue of intertwined dependence,
we introduce in section~\ref{subsectionColouredLoops}
{\it distinguished sites} which bear {\it coloured loops}
which are used to activate clusters at a distinct
level of the hierarchy.
More precisely, each cluster of the hierarchy is equipped with a distinguished site, and each loop emerging from this site is devoted to activating one specific cluster, so that at each level of the hierarchy, a certain proportion of the loops are ignored and shelved apart for further levels.
Doing so, knowing that it takes a long time to stabilize~$D$, we only have an information on the loops which are devoted to reactivation inside of~$D$, while the loops devoted to reactivation of~$E$ are left blank, and are thus distributed as standard excursions.

By allowing only some loops to activate given particles,
we build a dynamics which is faster to stabilize.
Since we only need a lower bound on the stabilization time for the original dynamics,
we will avoid controlling the previously mentioned relaxation time to metastable equilibrium
by working with such a stochastic domination.
It will also turn out that our coloured loop numbers 
are positively correlated with the ping-pong rally lengths.
See section~\ref{subsectionColouredLoops} for the description of these coloured loops,
and section~\ref{sectionInduction} for the inductive step
where we use their crucial independence property.

\subsection{Sketch of the proof of Theorem~\ref{thm3D}}
\label{sketch3D}

In dimensions~$d\geqslant 3$, since the probability that a loop starting from any point~$x\in\torus$ visits any other vertex~$y\in\torus$ is bounded below by a universal positive constant, independent of~$n$ and of the distance between~$x$ and~$y$ (see Lemma~\ref{lemmaHarnack}), we may almost forget the geometry in our strategy.

We still reason with a fixed settling set~$A$ and we perform loops, going back after each step to a configuration with exactly one particle on each site of~$A$.
But, compared with the proof in dimension~$2$, we do not need any hierarchical structure on~$A$, nor to distinguish vertices or to colour the loops.
However, for coherence with the rest of the paper which is devoted to dimension~$2$ and to avoid introducing specific notation for this corollary, we say that we use a trivial hierarchy and we keep one distinguished vertex and coloured loops, but this is simply a matter of notation.

Then, our toppling strategy simply consists in toppling whatever active vertex in~$A$ and performing a sleep or a loop.
The result then follows from the computations of the initialization step in section~\ref{sectionInit}, which show that the system with density one has a metastable behaviour, easily leading to the bounds on the critical density indicated in Theorem~\ref{thm3D}.

\subsection{How to adapt our proof in dimension 1}
\label{subsectionDim1}

As explained above, our proof method also works to show that~$\mu_c<1$ for all~$\lambda>0$ in dimension~$1$, with some adaptations, and it also yields bounds on the critical density, but these bounds are not new.

The regime of small sleep rate~$\lambda$ is already pretty well understood.
We  refer to~\cite{ARS19} which shows that~$c\sqrt{\lambda}\leqslant\mu_c(\lambda)\leqslant C\sqrt{\lambda}$
for some constants~$c,\,C>0$ and~$\lambda$ small enough.

In the regime of large sleep rate~$\lambda$, to obtain a lower bound on~$\mu_c$ using our method, one needs to consider connected sets at the first level of the hierarchy.
Thus, the partition~$\calC_0$ is composed of the connected components of~$A$ which contain at least~$v$ vertices, with~$v$ a well chosen function of~$\lambda$.

Then, in the initialization step, one needs to control the drift in a finer way, using the fact that a connected component in dimension~$1$ is simply a segment.
Thus, one can choose an endpoint as the distinguished vertex and use the toppling procedure which simply consists in toppling the active site which is closest to the distinguished vertex.
Doing so, one can show that there is a drift which leads to the distinguished vertex being awaken many times.
To this end, instead of simply writing that the toppled site has a certain number of sleeping sites in a certain ball around itself, as in section~\ref{sectionInit}, one has to use the fact that there is one sleeping site at distance~$1$, another one at distance~$2$, and so on.
Summing the probabilities to wake up each of these sites, one obtains a series which diverges with the size of the cluster, showing that there is a drift which outweighs the sleep rate when the size~$v$ of the cluster is taken large enough.

After this step, one can conclude using the induction result given by Lemma~\ref{lemmaInduction}.
Thus, our proof also works in dimension~$1$.
But it turns out that, in this case, it does not yield significantly better results than the bounds existing in the literature.
Namely, in~\cite{HRR20} it is shown that~\smash{$\mu_c(\lambda)\leqslant 1-\exp(-c\lambda)$} for some~$c>0$ and~$\lambda$ large enough, and our method yields the same kind of estimate, hence the choice to restrict our exposition to dimensions at least~$2$.

\subsection{Organization of the paper}

After some preliminaries in section~\ref{sectionGeneralConsiderations}, we present the proof of
Theorem~\ref{thm2D} about~$d=2$ in section~\ref{section2D}, followed by the proof of Theorem~\ref{thm3D}, the transient
case~$d\geqslant 3$, in section~\ref{section3D}.
Both of these proofs rely on a certain number of intermediary Lemmas, which are proved in the subsequent sections.

The construction of the dormitories used in the two-dimensional case is presented in section~\ref{sectionHierarchy}, while the induction is performed in sections~\ref{sectionInit} (for the
initialization) and~\ref{sectionInduction} (for the inductive step).
Finally in the appendix, we gather the proofs of some
technical Lemmas.

\section{General considerations}
\label{sectionGeneralConsiderations}

We now present some general definitions and preliminaries.
Let~$d,\,n\geqslant 1$.

\subsection{Notation}
\label{subsectionNotation}

Recall that we write~$\torus=(\Z/n\Z)^d$ for the~$d$-dimensional torus.
Denoting by~\smash{$\pi_n:\Z^d\to\torus$} a standard projection from~$\Z^d$ onto the torus, we define the distance between two points~$x,\,y\in\torus$ as
\begin{equation}
\label{defDistance}
d(x,\,y)
\ =\ \inf\Big\{\,\norme{a-b}_{\infty}\quad:\quad
a,\,b\in\Z^d\,,\ \pi_n(a)=x,\,\ \pi_n(b)=y\,\Big\}\,.
\end{equation}
For every non-empty set~$C\subset\torus$, we define its diameter
$$\diam\,C\ =\ \max_{x,\,y\,\in\, C}\,d(x,\,y)\,.$$
For every~$x\in\torus$ and every~$r\in\N$, we denote by~$B(x,\,r)$ the closed ball in the torus centered on~$x$ with radius~$r$, that is to say,
$$B(x,\,r)\ =\ \big\{y\in\torus\ :\ d(x,\,y)\leqslant r\big\}\,.$$
Note that the volume of~$B(x,\,r)$ is simply given by
\begin{equation}
\label{volumeBall}
\big|B(x,\,r)\big|
\ =\ \begin{cases}
(2r+1)^d &\text{ if } n\geqslant 2r+1\,,\\
n^d &\text{ otherwise.}
\end{cases}
\end{equation}
With a slight abuse of language, a set~$C\subset\torus$ is said to be connected if, for any two points~$x,\,y\in C$, there exists~$k\in\N$ and a sequence~$x_0,\,\ldots,\,x_k\in C$ such that~$x_0=x$,~$x_k=y$ and~$d(x_j,\,x_{j+1})=1$ for every~\smash{$j<k$} (as if there were diagonal edges).
Similarly, if~$r\in\N$, a set~$C\subset\torus$ is said to be~$r$-connected if, for any two points~$x,\,y\in C$, there exists~$k\in\N$ and a sequence~$x_0,\,\ldots,\,x_k\in C$ such that~$x_0=x$,~$x_k=y$ and~$d(x_j,\,x_{j+1})\leqslant r$ every~\smash{$j<k$}.

For every set~$E$, we denote by~$\calP(E)$ the set of all subsets of~$E$.

\subsection{Sufficient condition for activity in terms of the number of topplings}

Lemma~\ref{lemmaConditionActivityTopplings} below gives a sufficient condition on the two parameters~$\lambda$ and~$\mu$ of the model to show that we are in the active phase.
This Lemma follows from~\cite{FG22} and to state it we need to introduce some notation.

The Lemma is formulated in terms of the number of topplings necessary to stabilize a given initial configuration of the model.
The number of topplings refers to the total number of jump and sleep events.
In the continuous-time model, each active particle jumps with rate~$1$ and tries to fall asleep with rate~$\lambda$ (which can either lead to the particle effectively falling asleep if it is alone or to nothing happening otherwise), and both of these events are called topplings.

In the site-wise representation of the model where, for each site, we draw an infinite sequence of toppling instructions (which can consist either of sleep instructions or of jump instructions indicating a neighbouring site to jump on), the number of topplings refers to the number of toppling instructions used.
We refer to~\cite{Rolla20} for a detailed presentation of this representation.

As explained in the sketch given in section~\ref{sec-sketch}, we consider a modification of the ARW model where, for a fixed~$A\subset\torus$, particles cannot fall asleep out of~$A$ (as if the sleeping rate was~$\lambda$ on~$A$ and~$0$ on~$\torus\setminus A$, or if no sleep instructions are drawn out of~$A$).
We write~$\calP^{\lambda,A}_\mu$ for the probability distribution relative to this modification of the ARW model where particles are not allowed to sleep out of~$A$, and we write~$M_A$ for the number of topplings on the sites of~$A$ necessary to stabilize.
The initial configuration is written~$\eta_0:\torus\to\N$, where~$\eta_0(x)=k$ means that we start with~$k$ active particles on the site~$x$.

\begin{lemma}
\label{lemmaConditionActivityTopplings}
Let~$d\geqslant 1$, let~$\lambda>0$ and~$\mu\in(0,1)$, and let us write
\begin{equation}
\label{defPsi}
\psi(\mu)
\ =\ -\mu\ln\mu-(1-\mu)\ln(1-\mu)\,.
\end{equation}
If there exist~$a>0$ and~$b>\psi(\mu)$ such that, 
for~$n\in\N$ large enough, for every~$A\subset\torus$ such that~$|A|=\Ceil{\mu n^d}$, we have
\begin{equation}
\label{conditionMA}
\calP^{\lambda,A}_\mu\Big(\,M_A\,\leqslant\,e^{a n^d}
\ \Big|\ \eta_0=\mathbf{1}_A\,\Big)
\ \leqslant\ e^{-b n^d}\,,
\end{equation}
then~$\mu\geqslant \mu_c(\lambda)$, where~$\mu_c$ is the critical density of the usual ARW model on~$\Z^d$.
\end{lemma}

The proof of this Lemma, which is only a combination of several results of~\cite{FG22}, is briefly presented in the appendix~\ref{sectionProofLemmaCondition} for completeness.

\subsection{Dormitory hierarchy}
\label{subsectionDormitory}

Given a settling set~$A\subset\torus$, we now describe a hierarchical structure that we build on~$A$ and which is the basis of our toppling strategy.
This structure, called the dormitory hierarchy, is deterministically associated to the set~$A$ and also depends on some parameters~$v$ and~$(D_j)_{j\in\N}$ which will be chosen as functions of the sleep rate~$\lambda$.
A dormitory hierarchy is defined as follows:

\begin{definition}
Let~$d,\,n,\,v\geqslant 1$ and let~\smash{$D=(D_j)_{j\in\N}\in(\N\setminus\{0\})^{\N}$}.
For every subset~\smash{$A\subset\torus$}, we call a~$(v,\,D)$-dormitory hierarchy on~$A$ a finite decreasing sequence of subsets~\smash{$A\supset A_0\supset \ldots \supset A_{\calJ}$}, with~${\calJ}\in\N$ and, for every~$j\leqslant {\calJ}$, a partition~$\calC_j$ of~$A_j$ such that:
\begin{enumerate}[(i)]
\item\label{conditionCard}
For~$0\leqslant j\leqslant {\calJ}$, for every~$C\in\calC_j$, we have~$|C|\geqslant 2^{\Ent{j/2}}v$;
\item\label{conditionMerge}
For~$0\leqslant j\leqslant {\calJ}-1$, for every~$C\in\calC_{j+1}\setminus\calC_{j}$, we have~$\diam\,C\leqslant D_j$ and there exist two sets~$C_0,\,C_1\in\calC_{j}$ such that~$C=C_0\cup C_1$;
\item\label{conditionFinal}
The last partition~$\calC_{\calJ}$ contains one single set.
\end{enumerate}
\end{definition}

Given a dormitory hierarchy~$(A_j,\,\calC_j)_{j\leqslant \calJ}$, for every~$j\leqslant {\calJ}$ and every~$x\in A_j$, we define~$C_j(x)$ to be the set~$C\in\calC_j$ such that~$x\in C$.
When~$x\in \torus\setminus A_j$ or~$j>{\calJ}$, we set~$C_j(x)=\varnothing$.
The sets~$C\in\calC_j$ are called {\it clusters} at the level~$j$.
The parameter~$v\in\N$ controls the volume of the clusters at each level of the hierarchy, while the sequence of diameters~$D_j$ ensures that we only merge clusters which are not too far apart.

One might wonder why the condition~\textit{\reff{conditionCard}} is not rather~$|C|\geqslant 2^j v$, which would be possible by merging all clusters in pairs at each level.
This would also work, but it would imply throwing away some clusters which end up alone: for example if~$\calC_0$ contains only three clusters of size exactly~$v$, then to construct~$\calC_1$ one of them would have to be thrown away.
In fact, this leads to throwing away too many clusters and in particular this would weaken the estimate~\reff{bound2Dlow} on the critical density when~$\lambda\to 0$.
Hence the choice of this weaker condition~\textit{\reff{conditionCard}}, which enables us to deal with odd numbers of clusters grouped together: for example groups of three clusters can be merged together in two steps.
This still leads to throwing away some clusters at each level (this is why the sequence~$(A_j)$ is decreasing), but we only discard a cluster if is isolated and not just because it belongs to a group of an odd number of clusters close to one another.
See our construction in section~\ref{sectionHierarchy} for more details.

\subsection{Distinguished vertices and coloured loops}
\label{subsectionColouredLoops}

\subsubsection{Distinguished vertices}
Let~$A\subset\torus$, and let~\smash{$(A_j,\,\calC_j)_{j\leqslant {\calJ}}$} be a dormitory hierarchy on~$A$, as defined in section~\ref{subsectionDormitory}.
We define recursively a distinguished 
vertex in each set of the partitions.
The distinguished point of a set~$C$ is written~$\dist_C$ and the particle sitting in~$\dist_C$ is called
the distinguished particle of the cluster~$C$.
For every~$C\in\calC_0$, we simply set~$\dist_C=\min\,C$, for an arbitrary order on the vertices of the torus.
Then, for~$1\leqslant j\leqslant {\calJ}$, if~$C\in\calC_j\setminus\calC_{j-1}$, the property~\textit{(\ref{conditionMerge})} of the hierarchy tells us that~$C$ is the union of two clusters of~$\calC_{j-1}$.
In this case, we let~$\dist_C$ be the distinguished vertex of the biggest of these two clusters (in terms of number of vertices, and with an arbitrary rule to break ties).

We say that a vertex~$x$ is distinguished at the level~$j\leqslant {\calJ}$ if there exists~$C\in\calC_j$ such that~\smash{$x=\dist_C$}, that is to say, if~\smash{$x=\dist_{C_j(x)}$}.
If~$j>{\calJ}$, we say that no vertex is distinguished at level~$j$.
Note that if~$x$ is distinguished at a certain level~$j$, then it is also distinguished at all levels~$j'$ for~$j'<j$.

\subsubsection{Toppling steps: sleeps and loops}

As explained above, we reason with a fixed subset~$A\subset\torus$ and we study the number of topplings necessary for all the particles to fall asleep, starting from the configuration with one active particle on each site of~$A$, in a modified model where there are no sleep instructions outside of~$A$.

As in~\cite{FG22}, our toppling strategy consists in a certain number of steps such that, after each step, there is still exactly one particle on each site of~$A$.
Thus, the configuration of the model at each step may be encoded by the subset~$R\subset A$ of the sites which contain one active particle, while each site of~$A\setminus R$ contains one sleeping particle.
We say that a set~$C\subset A$ is stable if~$R\cap C=\varnothing$.

At each step, we start by choosing a site~$x\in A$ where an active particle is present. With probability~$\lambda/(1+\lambda)$, this particle falls asleep on~$x$ (we call this step an~$x$-sleep) and we proceed to the next step.

Otherwise (hence with probability~$1/(1+\lambda)$), the particle makes an~$x$-loop, that is to say, it performs a simple random walk on the torus, until it goes back to its starting point~$x$, where it is left, active.
The sleeping particles met along this loop are waken up by the passage of the particle, but only under a certain condition, depending on the \guillemets{colour} of the loop, as explained below.

\subsubsection{Coloured loops}
The coloured loops are a new ingredient compared to~\cite{FG22}: each time a loop starts from a site~$x$, we assign to this loop a random colour~$J$, where~$J+1$ is drawn from a geometric distribution with parameter~$1/2$.
Then, if along its loop starting from~$x$ the particle meets a sleeping particle at a site~$y$, we only wake up the sleeping particle if~$y\in w(x,\,J)$, where the function~$w$ is defined as follows: for every~$x\in A$ and every~$j\in\N$, we set
\begin{equation}
\label{defW}
w(x,\,j)
\ =\ 
\begin{cases}
C_{j+1}(x)\setminus C_{j}(x)
&\text{if }x\text{ is distinguished at level }j\,,\\
\varnothing
&\text{if }x \text{ is distinguished at level }0\text{ but not at level }j\,,\\
C_0(x)
&\text{if }x\text{ is not distinguished at any level.}
\end{cases}
\end{equation}

We now explain the practical meaning of the above definition, which is illustrated on figure~\ref{figureLoops}.

If~$x$ is distinguished at level~$0$ but not at level~$1$, then the~$x$-loops of colour~$0$ (on average half of the~$x$-loops) can only wake up the particles in~$C_1(x)\setminus C_0(x)$, and the rest of the~$x$-loops are ignored, that is to say, they cannot wake up anyone (because~$w(x,\,j)=\varnothing$ for all~$j\geqslant 1$).

If~$x$ is distinguished at levels~$0,\,\ldots,\,j$ but not at level~$j+1$, then the loops of colour~$0$ (which represent on average half of the loops) are devoted to~$C_1(x)\setminus C_0(x)$, while the loops of colour~$1$ (about a quarter of the loops) are devoted to~$C_2(x)\setminus C_1(x)$, and so on, until the loops of colour~$j$ (an average proportion~$1/2^{j+1}$ of the loops) which are devoted to~\smash{$C_{j+1}(x)\setminus C_j(x)$}, whereas the loops of colour strictly more than~$j$ cannot wake up anyone.

As for the loops coming from sites~$x$ which are not distinguished at any level, their colours have no importance and they are only allowed to wake up the sites in the same~$0$-level component~$C_0(x)$.
Note that, on the contrary, the loops coming from a distinguished site~$x$ can never wake up the other sites of~$C_0(x)$.

Note also that, for example, if~$C\in\calC_0\cap\calC_1$ (that is to say, if the cluster~$C$ is not merged with another cluster of~$\calC_0$), then the~$\dist_C$-loops of colour~$0$ are not allowed to wake up anyone, since we have~$w(\dist_C,\,0)=C\setminus C=\varnothing$.

\begin{figure}
\begin{center}
\begin{tikzpicture}
\filldraw[very thick, rounded corners=10pt,fill=blue!10] (-1,-6) rectangle (13,4.5);
\filldraw[very thick, rounded corners=10pt,fill=red!10] (-.5,-.5) rectangle (12.5,3.5);
\filldraw[very thick, rounded corners=10pt,fill=gray!30] (0,0) rectangle (4,3);
\filldraw[very thick, fill=gray!50] (0.6,0.2) rectangle (3.4,1);
\draw[->,very thick] (2,2) -- node[right]{all loops} (2,1.2);
\draw[->,very thick] (2.5,2.1) to[out=-30, in=-90]
(3.3,2.2) to[out=90,in=-10] (2.7,2.6);
\draw (2,2.5) node{$C\in\calC_0$};
\draw (2,0.6) node{$\dist_C={\color{red}\dist_E}={\color{blue}\dist_I}$};

\begin{scope}[xshift=8cm]
\filldraw[very thick, rounded corners=10pt,fill=gray!30] (0,0) rectangle (4,3);
\filldraw[very thick, fill=gray!50] (1.3,0.2) rectangle (2.7,1);
\draw[->,very thick] (2,2) -- node[right]{all loops} (2,1.2);
\draw[->,very thick] (2.5,2.1) to[out=-30, in=-90]
(3.3,2.2) to[out=90,in=-10] (2.7,2.6);
\draw (2,2.5) node{$D\in\calC_0$};
\draw (2,0.6) node{$\dist_D$};
\end{scope}

\draw[->,very thick,red] (3.4,0.3) -- (8,0.3);
\draw[red] (6,0.6) node{loops of colour 0};
\draw[->,very thick,red] (9.3,0.9) -- (4,0.9);
\draw[red] (6,3) node{$E=C\cup D\in\calC_1$};

\begin{scope}[yshift=-5cm]
\filldraw[very thick, rounded corners=10pt,fill=red!10] (-.5,-.5) rectangle (12.5,3.5);
\filldraw[very thick, rounded corners=10pt,fill=gray!30] (0,0) rectangle (4,3);
\filldraw[very thick, fill=gray!50] (1,0.2) rectangle (3,1);
\draw[->,very thick] (2,2) -- node[right]{all loops} (2,1.2);
\draw[->,very thick] (2.5,2.1) to[out=-30, in=-90]
(3.3,2.2) to[out=90,in=-10] (2.7,2.6);
\draw (2,2.5) node{$F\in\calC_0$};
\draw (2,0.6) node{$\dist_F={\color{red}\dist_H}$};

\begin{scope}[xshift=8cm]
\filldraw[very thick, rounded corners=10pt,fill=gray!30] (0,0) rectangle (4,3);
\filldraw[very thick, fill=gray!50] (1.3,0.2) rectangle (2.7,1);
\draw[->,very thick] (2,2) -- node[right]{all loops} (2,1.2);
\draw[->,very thick] (2.5,2.1) to[out=-30, in=-90]
(3.3,2.2) to[out=90,in=-10] (2.7,2.6);
\draw (2,2.5) node{$G\in\calC_0$};
\draw (2,0.6) node{$\dist_G$};
\end{scope}

\draw[->,very thick,red] (3,0.3) -- (8,0.3);
\draw[red] (6,0.6) node{loops of colour 0};
\draw[->,very thick,red] (9.3,0.9) -- (4,0.9);
\draw[red] (6,3) node{$H=F\cup G\in\calC_1$};
\end{scope}

\draw[->,very thick,blue] (1.6,0.2) -- (1.6,-1.5);
\draw[->,very thick,blue] (1.2,-4) -- (1.2,-.5);
\draw[blue] (3,-1) node{loops of colour 1};
\draw[blue] (6,4) node{$I=E\cup H\in\calC_2$};
\draw[->,very thick,green!50!black] (1,1) -- (1,5.3);
\draw[->,very thick,green!50!black] (.6,5.3) -- (.6,4.5);
\draw[green!50!black] (1.1,4.9) node[right]{loops of colour 2};
\end{tikzpicture}
\end{center}
\caption{\label{figureLoops}
Clusters of the dormitory hierarchy are drawn in rounded rectangles.
Each cluster bears a distinguished vertex which is represented by a dark normal rectangle.
The loops from the sites which are not distinguished are only allowed to wake up sites in the same cluster of~$\calC_0$, including the distinguished site.
The loops of colour~$0$ emitted by the distinguished site~$\dist_C$ are only allowed to wake up the sites in~$D$, while the loops of colour~$1$ emitted by~$\dist_E=\dist_C$ are only allowed to wake up sites in~$H$, and the loops of colour~$3$ are devoted to waking up the sites in another cluster of~$\calC_1$ with which~$I$ merges at the next level.
}
\end{figure}
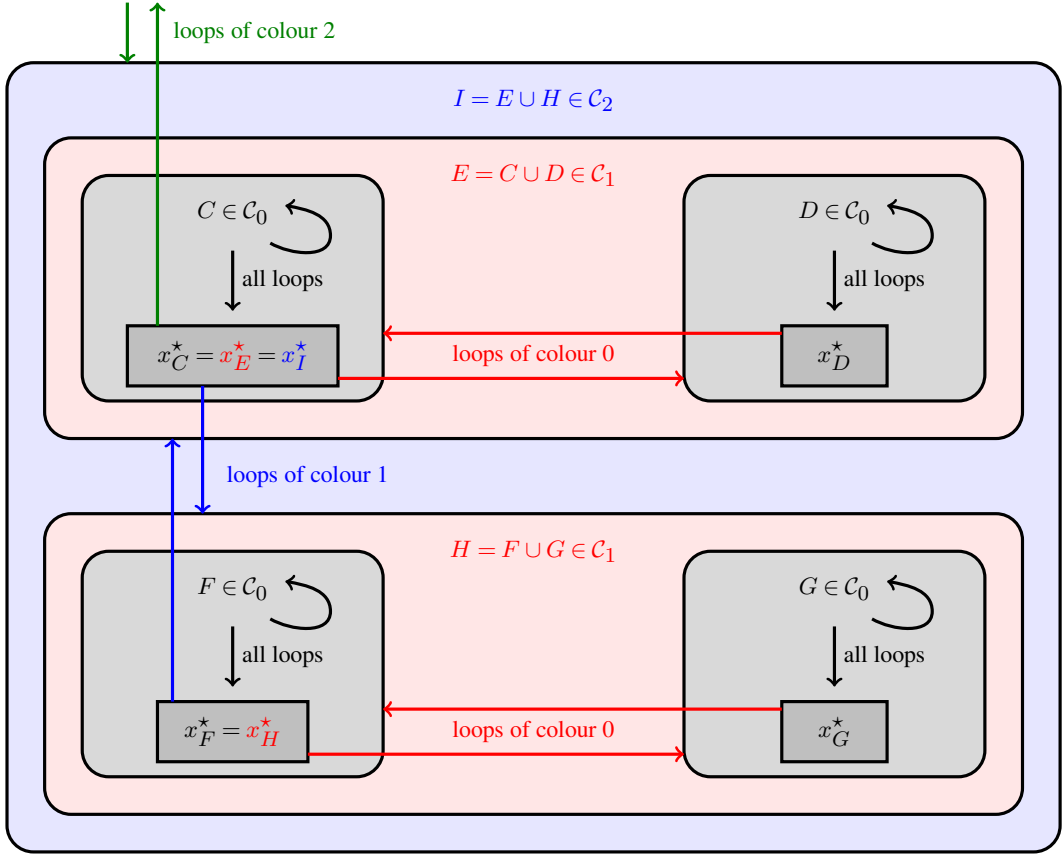

\subsubsection{Comparison with the original ARW model}
If all the particles are sleeping after a certain sequence of steps using this restriction on coloured loops, then the same sequence of steps can be performed in the original model, where no waking up events are ignored (the evolution of the two models can be coupled such that, at any time, the configuration in the modified model is always \guillemets{below} the configuration in the original model after the same number of steps).
This sequence might not be enough to stabilize the configuration in the original model but the number of steps performed provides a lower bound on the number of topplings necessary to stabilize the configuration in the original model (see Lemma~\ref{lemmaConditionActivityLoops} and its proof in the appendix~\ref{sectionProofLemmaCondition} for more details).

\subsection{The loop representation of the modified ARW model}
\label{subsectionLoopRep}

Let us fix~$A\subset\torus$, along with a dormitory hierarchy~$(A_j,\,\calC_j)_{j\leqslant\calJ}$.

\subsubsection{Our probability space}
We now describe a representation of the model which is convenient for our proof method, and which consists in storing an infinite array of loops above every vertex.
The dormitory~$A$ being fixed, we consider independent random variables
$$\left\{\begin{aligned}
&\big(I(x,\,h)\big)_{x\in A,\,h\in\N}\in\{0,1\}^{A\times\N}\,,\\
&\big(J(x,\,\ell)\big)_{x\in A,\,\ell\in\N}\in\N^{A\times\N}\,,\\
&\big(\Gamma(x,\,\ell,\,j)\big)_{x\in A,\,\ell\in\N,\,j\in\N}\in\calP(A)^{A\times\N^2}\,,
\end{aligned}\right.$$
where the variables~$I(x,\,h)$ are Bernoulli with parameter~$\lambda/(1+\lambda)$ while the variables~\smash{$1+J(x,\,\ell)$} are geometric with parameter~$1/2$ and~$\Gamma(x,\,\ell,\,j)$ is distributed as the support of a symmetric random walk on the torus started and killed at~$x$, that is to say, for every~$B\subset\torus$, we have
$$\Proba\big(\Gamma(x,\,\ell,\,j)=B\big)
\ =\ P_x\Big(\,\big\{y\in\torus\ :\ T_y<T_x^+\big\}=B\,\Big)\,,$$
where~$T_y$ denotes the first hitting time of~$y$, while~$T_x^+$ is the first return time to~$x$, and~$P_x$ is the probability measure relative to the symmetric random walk on the torus started at~$x$.
Probabilities and expectations are simply denoted by~$\PlA$ and~$\ElA$, which depend implicitly on the parameter~$\lambda$ and on the set~$A$.

\subsubsection{Update rules}
We now describe the update rules of our model.
Recall that a configuration of our model is a subset~$R\subset A$ indicating which sites are active.
The idea is that, when we perform a toppling step at a site~$x$, to decide whether this step is an~$x$-sleep or an~$x$-loop we look at the instruction~$I(x,\,h)$, where~$h$ counts the total number of~$x$-loops and~$x$-sleeps already performed, so that we do not use twice the same variable~$I(x,\,h)$.

If~$I(x,\,h)=1$, the particle falls asleep (this is what we call an~$x$-sleep).
Otherwise, we perform an~$x$-loop with colour~$j=J(x,\,\ell)$, where~$\ell$ is the number of loops (of any colour) which have already been performed at~$x$.
The effect of this loop is to wake up all the particles in~$\Gamma(x,\,\ell,\,j)\cap w(x,\,j)$, where~$w:A\times\N\to\calP(A)$ is the function defined by~\reff{defW} which indicates the set of sites that~$x$ has the right to wake up during a loop of colour~$j$.

Note that, with this notation, the array~$\Gamma$ contains too many loops because, for every~$x\in A$ and every~$\ell\in\N$, at most one of the loops~$\{\Gamma(x,\,\ell,\,j),\,j\in\N\}$ is used, depending on the colour~$J(x,\,\ell)$ of the~$x$-loop numbered~$\ell$.
But this notation is more convenient to highlight the independence between loops of different colours, in particular in section~\ref{sectionInduction}.

To update the configuration we need to recall the numbers of instructions and loops already used at each vertex.
This is the role of what we call the odometer function~$h:A\rightarrow\N$ and the loop odometer function~$\ell:A\rightarrow\N$.

\subsubsection{Step-toppling operator}
Given a configuration~$R\subset A$, an odometer~$h:A\rightarrow\N$, a loop odometer~$\ell:A\rightarrow\N$ and a site~$x\in R$, writing~$i=I\big(x,\,h(x)\big)$ for the next available instruction at~$x$ and~$j=J\big(x,\,\ell(x)\big)$ which is the colour of the next~$x$-loop, we define the step-toppling operator as
$$\Phi_x(R,\,h,\,\ell)
\ =\ \begin{cases}
\big(R\setminus\{x\},\,h+\delta_x,\,\ell\big)
& \text{ if }i=1\,,\\
\big(R\cup\big(\Gamma(x,\,\ell(x),\,j)\cap w(x,\,j)\big),\,h+\delta_x,\,\ell+\delta_x\big)
& \text{ otherwise.}
\end{cases}$$
This operator gives the configuration obtained after performing a step starting at~$x$, and the resulting odometer and loop odometer after the step.

\subsubsection{Toppling procedures and procedure-toppling operator}
For every cluster~$C\subset \calC_0$, we call a~$C$-toppling procedure any function~$f:\calP(C)\setminus\{\varnothing\}\rightarrow C$ such that, for every configuration~$R\subset C$ with~$R\neq\varnothing$, we have~$f(R)\in R$ and~$\dist_C\in R\Rightarrow f(R)=\dist_C$.
The role of a~$C$-toppling procedure is to indicate the order with which sites must be toppled depending on the actual configuration~$R$, \guillemets{without looking into the future}.
It is fundamental for our proof that the choice of the next toppling depends only on the actual configuration restricted to~$C$.
Here,~$R\subset C$ denotes the set of active sites in~$C$ and, as such, defines the configuration of particles inside~$C$.
The condition~$f(R)\in R$ ensures that we topple an active site, while the condition involving~$\dist_C$ means that we give priority to the distinguished vertex, which is toppled as soon as it is awaken.
The priority given to the distinguished vertex is due to the fact that we want this distinguished vertex to be awaken many times and to emit many loops, which would not be the case if it was the last site to be toppled.

Given a cluster~$C\subset \calC_0$ and a fixed~$C$-toppling procedure~$f$, the procedure-toppling operator simply consists in applying the step-toppling operator~$\Phi_x$ defined above at the site~$x$ indicated by the toppling procedure, and doing nothing if~$C$ is already stable:
$$\Phi_C\,:\,\left\{\begin{aligned}
\calP(A)\times\big(\N^{A}\big)^2
&\ \longrightarrow\ 
\calP(A)\times\big(\N^{A}\big)^2\\
(R,\,h,\,\ell)
&\ \longmapsto\ 
\begin{cases}
(R,\,h,\,\ell)
&\text{ if }R\cap C=\varnothing\,,\\
\Phi_{f(R\cap C)}(R,\,h,\,\ell)
&\text{ otherwise.}
\end{cases}
\end{aligned}\right.$$
The dependency in~$f$ is implicit and is omitted to simplify the notation.
For every~$t\in\N$, this operator iterated~$t$ times is simply written~\smash{$\big(\Phi_{C}\big)^{(t)}$}.

\subsection{Our recursive toppling strategy}
\label{subsectionStrategy}

Let~$A\subset\torus$ and let~$(A_j,\,\calC_j)_{j\leqslant {\calJ}}$ be a fixed dormitory hierarchy on~$A$.
We now explain, for every~$j\leqslant {\calJ}$ and every~$C\in\calC_j$, how we proceed to stabilize the set~$C$.

\subsubsection{Stabilization at the~0-th level}
Assume that, for every cluster~$C\in\calC_0$, we have fixed a~$C$-toppling procedure~$f_C$ (these procedures will be constructed in the proof of Lemma~\ref{lemmaInit2Dhigh}, in section~\ref{sectionInit}).

For every~$C\in\calC_0$, to stabilize the set~$C$ we simply use the toppling procedure~$f_C$ until all the sites of~$C$ are asleep.
Since we may need to stabilize this set~$C$ many times, we consider a general stabilization operator starting from a given initial configuration~$R\subset A$ and a certain offset~$h_0,\,\ell_0$ for the odometers.
Indeed, when we perform many stabilizations of various sets of the hierarchy, each stabilization starts from the configuration and the odometers left by the previous stabilizations.
Namely, for every~$C\in\calC_0$, we define the stabilization operator
$$\mathrm{Stab}_C\,:\,\left\{\begin{aligned}
\calP(A)\times\big(\N^{A}\big)^2
&\ \longrightarrow\ 
\calP(A)\times\big(\N^{A}\big)^2\\
(R_0,\,h_0,\,\ell_0)
&\ \longmapsto\ 
(R_\tau,\,h_\tau,\,\ell_\tau)\,,
\end{aligned}\right.$$
where, for every~$t\geqslant 1$, we write
$$(R_t,\,h_t,\,\ell_t)
\ =\ \big(\Phi_C\big)^{(t)}
(R_0,\,h_0,\,\ell_0)\,,$$
with~$\Phi_C$ referring to the procedure-toppling operator using the toppling procedure~$f_C$, and
$$\tau
\ =\ \inf\big\{t\in\N\ :\ R_t\cap C=\varnothing\big\}\,.$$
If~$\tau=+\infty$, the value of~\smash{$\mathrm{Stab}_C(R_0,\,h_0,\,\ell_0)$} can be defined arbitrarily (we do not care about this case since it occurs with probability~$0$).

\subsubsection{The ping-pong rally}
We now construct recursively the stabilization operators for the successive levels of the hierarchy.
Let~$j\in\{1,\,\ldots,\,{\calJ}\}$ be such that the stabilization operator is well defined for every~$C\in\calC_{j-1}$, and let~$C\in\calC_j\setminus\calC_{j-1}$.
By definition of the dormitory hierarchy, we can write~$C=C_0\cup C_1$ with~$C_0,\,C_1\in\calC_{j-1}$.
Let us assume that~$\dist_C=\dist_{C_0}$ (otherwise we swap the notation between~$C_0$ and~$C_1$).

Then, to stabilize the set~$C$, we start by stabilizing~$C_0$, then~$C_1$.
After this, if some sites of~$C_0$ have been reactivated during the stabilization of~$C_1$, we stabilize~$C_0$ once again.
Then, if some sites of~$C_1$ are still active, we stabilize~$C_1$ again, and so on and so forth, alternating between the two sets until both are fully stabilized.

Formally, the stabilization operator for~$C$ is defined as
\begin{equation}
\label{defStab}
\mathrm{Stab}_C\,:\,\left\{\begin{aligned}
\calP(A)\times\big(\N^{A}\big)^2
&\ \longrightarrow\ 
\calP(A)\times\big(\N^{A}\big)^2\\
(R_0,\,h_0,\,\ell_0)
&\ \longmapsto\ 
(R_\tau,\,h_\tau,\,\ell_\tau)\,,
\end{aligned}\right.
\end{equation}
where, for every~$i\in\N$, we write
$$\left\{\begin{aligned}
&(R_{2i+1},\,h_{2i+1},\,\ell_{2i+1})
\ =\ \mathrm{Stab}_{C_0}(R_{2i},\,h_{2i},\,\ell_{2i})\\
&(R_{2i+2},\,h_{2i+2},\,\ell_{2i+2})
\ =\ \mathrm{Stab}_{C_1}(R_{2i+1},\,h_{2i+1},\,\ell_{2i+1})
\end{aligned}\right.$$
and
$$\tau
\ =\ \inf\big\{i\in\N\ :\ R_i\cap C=\varnothing\big\}\,.$$

\subsubsection{Number of topplings and loops during stabilization}
Let~$j\in\{0,\,\ldots,\,{\calJ}\}$ and let~$C\in\calC_j$.
We denote the total number of sleeps and loops performed during the stabilization of~$C$ by
$$\calH(C)
\ =\ \sum_{x\in C} h_{\text{stab}}(x)
\quadou
(R_{\text{stab}},\,h_{\text{stab}},\,\ell_{\text{stab}})
\ =\ \mathrm{Stab}_C(A,\,0,\,0)
\ =\ \mathrm{Stab}_C(C,\,0,\,0)\,.$$
In our main proof, instead of controlling the total number of sleeps and loops used to stabilize, we concentrate on the number of loops performed by the distinguished vertex~$\dist_C$ during the stabilization of~$C$, which we denote by~$\calL(C)=\ell_{\text{stab}}(\dist_C)$, with~$\ell_{\text{stab}}$ defined as above, whereas the number of sleeps is written~$\calS(C)=h_{\text{stab}}(\dist_C)-\ell_{\text{stab}}(\dist_C)$.
Then, among the loops produced by~$\dist_C$ we are interested in the loops of a specific colour.
Thus, for every~$k\in\N$, we introduce the notation
$$\calL(C,\,k)
\ =\ \abs{\Big\{\,\ell<\calL(C)
\ :\ J(\dist_C,\,\ell)=k\,\Big\}}$$
for the number of loops of colour~$k$ emitted by~$\dist_C$ during the stabilization of~$C$ using our toppling strategy.

\subsection{Sufficient condition for activity in terms of the number of sleeps and loops}
\label{subsectionConditionLoops}

Instead of the more general sufficient condition given by Lemma~\ref{lemmaConditionActivityTopplings}, we use the following more specific condition which is adapted to our setting.
Recall that the function~$\psi$ was defined in~\reff{defPsi}.

\begin{lemma}
\label{lemmaConditionActivityLoops}
Let~$d\geqslant 1$, let~$\lambda>0$ and~$\mu\in(0,1)$.
If there exists~$\kappa>\psi(\mu)$ such that, 
for~$n\in\N$ large enough, for every~$A\subset\torus$ with~$|A|=\Ceil{\mu n^d}$, there exists a dormitory hierarchy~$(A_j,\,\calC_j)_{j\leqslant {\calJ}}$ and a toppling procedure~$f_C$ for every~$C\in\calC_0$ such that, with the recursive toppling strategy defined in section~\ref{subsectionStrategy}, we have the stochastic domination
\begin{equation}
\label{goalStochDom}
\calH(A_{\calJ})
\ \succeq\ \mathrm{Geom}\big(\exp(-\kappa\,n^d)\big)\,,
\end{equation}
then we have~$\mu\geqslant \mu_c(\lambda)$, where~$\mu_c(\lambda)$ is the critical density of the usual ARW model on~$\Z^d$.
\end{lemma}

The proof of this Lemma, which relies on the previous sufficient condition given by Lemma~\ref{lemmaConditionActivityTopplings}, is deferred to the appendix~\ref{sectionProofLemmaCondition}.

Note that in our construction with loops, the number of loops performed during stabilization is not {\it abelian}: it depends on the toppling strategy.
Indeed, once the sleep instructions~$I$, the colours~$J$ and the loops~$\Gamma$ (which altogether enclose all the randomness in our model) are drawn, one can obtain different numbers of loops depending on the order with which the loops are used, and even the number of sleeps and loops or the total odometer summing the number of topplings performed along the loops are not abelian.

However, having fixed our toppling strategy,~$\calH(A_{\calJ})$, the total number of sleeps and loops performed to stabilize, is stochastically dominated by~$M_A$, the number of toppling instructions used on the sites of~$A$ during stabilization in the \guillemets{original} ARW model (where sleep instructions out of~$A$ are ignored but no waking up events are ignored).

\subsection{Correlation between the numbers of loops of each colour}

The following result relates the number of loops of colour~$j$ performed by the distinguished vertex of a cluster~$C\in\calC_j$ with the number of sleeps and loops of colour at most~$j-1$:

\begin{lemma}
\label{lemmaNextColour}
Let~$d,\,n\geqslant 1$, let~$\lambda>0$, let~$A\subset\torus$, let~$(A_j,\,\calC_j)_{j\leqslant {\calJ}}$ be a dormitory hierarchy on~$A$ and for every~$C\in\calC_0$, let~$f_C$ be a toppling procedure on~$C$.
Then, for every~$j\in\{0,\,\ldots,\,{\calJ}\}$ and every cluster~$C\in\calC_j$, we have the equality in distribution
\begin{equation}
\label{sumGeomLoops}
\calL(C,\,j)
\ \stackrel{\text{d}}{=}\ \sum_{i=1}^{\calT}(X_i-1)\,,
\end{equation}
where~$\calT=\calS(C)+\calL(C,\,0)+\cdots+\calL(C,\,j-1)$ and~$(X_i)_{i\geqslant 1}$ are i.i.d.\ geometric random variables with parameter
\begin{equation}
\label{paramNextColour}
\frac{\lambda+1-2^{-j}}
{\lambda+1-2^{-(j+1)}}\,,
\end{equation}
which are independent of~$\calT$.
\end{lemma}

This Lemma can be easily understood if one thinks of a continuous-time variant of our model where particles fall asleep at rate~$p_\s=\lambda$ and perform loops of colour~$j$ with rate~\smash{$p_j=2^{-(j+1)}$}.
Then, the number of loops of colour~$j$ between any two topplings counted in~$\calT$ is a geometric minus one, with parameter
\begin{equation}
\label{computationParamXi}
\frac{p_\s+p_0+\cdots+p_{j-1}}
{p_\s+p_0+\cdots+p_{j-1}+p_j}
\ =\ \frac{\lambda+2^{-1}+\cdots+2^{-j}}
{\lambda+2^{-1}+\cdots+2^{-(j+1)}}
\ =\ \frac{\lambda+1-2^{-j}}
{\lambda+1-2^{-(j+1)}}\,,
\end{equation}
and the independence between~$\calT$ and these geometric variables follows from the fact that the loops with colour~$j$ have no impact on the stabilization of~$C$.
See the proof in the appendix~\ref{sectionProofLemmaNextColour} for more details.

\subsection{A useful property of geometric random variables}

We now state a technical Lemma which is proved in the appendix~\ref{sectionProofLemmaSumGeometrics}.

\begin{lemma}\label{lemmaSumGeometrics}
Let~$N$ be a geometric random variable with parameter~$a\in(0,1)$, and let~$(X_n)_{n\in\N}$ be i.i.d.\ geometric variables with parameter~$b\in(0,1)$, independent of~$N$.
Then, the variable
$$
S\ =\ 1+\sum_{n=1}^N \big(X_n-1\big)\qquad
\text{is geometric with parameter}\qquad
\frac{ab}{1-b+ab}\,.$$
\end{lemma}

\subsection{Hitting probabilities on the torus}
\label{subsectionHitting}

We need to introduce a key function, which measures the chance to wake up a distant site in a loop on the torus~$\torus$, for~$d\geqslant 1$:
\begin{equation}
\label{defUpsilon}
\Upsilon_d\,:\,r\in\N\setminus\{0\}
\ \longmapsto\ 
\inf\,\Big\{\,P_x\big(T_y<T_x^+\big)\,,\ 
n\in\N,\ x,\,y\in\torus\ :\ d(x,\,y)\leqslant r\,\Big\}\,.
\end{equation}

We use the following estimates on this function:

\begin{lemma}
\label{lemmaHarnack}
We have the lower bounds:
\begin{itemize}
\item In dimension~$d=1$, we have~$\Upsilon_1(r)= 1/(2r)$ for every~$r\geqslant 1$;
\item In dimension~$d=2$, there exists~$K>0$ such that~$\Upsilon_2(r)\geqslant K/\ln r$ for every~$r\geqslant 2$;
\item In dimension~$d\geqslant 3$, there exists~$K=K(d)>0$ such that~$\Upsilon_d(r)\geqslant K$ for every~$r\geqslant 1$.
\end{itemize}
\end{lemma}

The proof of this Lemma is deferred to the appendix~\ref{sectionProofLemmaHarnack}.

\section{Proof of Theorem~\ref{thm2D}}
\label{section2D}

This section is devoted to the proof of our main result which gives upper bounds on the critical density in two dimensions for small and large sleep rates.
The arguments rely on several intermediary Lemmas, but we postpone the proofs of these Lemmas to later sections, to allow the reader to grasp the articulation of the proof.

\subsection{Canvas for the inductive proof}
\label{subsectionCanvas}

Recall that we rely on the sufficient condition for activity given by Lemma~\ref{lemmaConditionActivityLoops}.
Thus, for~$\lambda>0$ and~$\mu\in(0,1)$ fixed, we reason with~$n\in\N$ a large integer and~$A\subset\Z_n^2$ fixed with~$|A|=\Ceil{\mu n^d}$, and our goal is to prove the stochastic domination~\reff{goalStochDom}, for a constant~$\kappa>0$ independent of~$n$ and~$A$.

To prove this stochastic domination, we proceed inductively on the {\it levels} of the dormitory hierarchy to show that at each level~$j$, for every~$C\in\calC_j$, the number of loops of colour~$j$ produced when stabilizing~$C$ dominates an explicit geometric random variable.
This is the key technical part of the paper, and we divide the proof into an initialization and an inductive step.

Before performing the induction, we need to construct a dormitory hierarchy~$(A_j,\,\calC_j)_{j\leqslant {\calJ}}$ and to choose a toppling procedure~$f_C$ for each set~$C\in\calC_0$.
For this first step we proceed differently depending on the regime of~$\lambda$ considered, hence we postpone it to the next two sections.

\subsubsection{The induction hypothesis} 
Let~$(\alpha_j)_{j\in\N}$ be a sequence of positive real numbers, to be chosen later (depending on the regime of~$\lambda$ under consideration).
Our induction hypothesis, written~$\calP(j)$, is the following: for every~$j\in\{0,\,\ldots,\,{\calJ}\}$, we define
$$\calP(j)\ :\qquad
\forall\,C\in\calC_j\qquad
1+\calL(C,\,j)
\ \succeq\ \mathrm{Geom}\big(\exp(-\alpha_j |C|)\big)\,.$$
Note that~$\calP(j)$ is an estimate on the whole law of the number of loops: the small and large values need to be controlled on all scales.

We are now ready to state the key inductive Lemma, whose proof occupies section~\ref{sectionInduction}.

\begin{lemma}
\label{lemmaInduction}
Let~$d\geqslant 1$,~$\lambda>0$,~$v\geqslant 1$,~\smash{$(D_j)_{j\in\N}\in(\N\setminus\{0\})^{\N}$}, and let~$(\alpha_j)_{j\in\N}$ be a sequence of positive real numbers such that
\begin{equation}
\label{conditionInduction}
\forall j\in\N\qquad
\frac{4v(1+\lambda) 2^{3j/2}}
{\big(1-e^{-\alpha_j v}\big)\Upsilon_d(D_{j})}
\ \leqslant\ \exp\Big((\alpha_j-\alpha_{j+1})2^{j/2}v\Big)\,.
\end{equation}
For every~$n\geqslant 1$ and every~$A\subset\torus$ equipped with a~$(v,\,D)$-dormitory hierarchy~$(A_j,\,\calC_j)_{j\leqslant {\calJ}}$ and with a collection of toppling procedures~$(f_C)_{C\in\calC_0}$, if the property~$\calP(0)$ holds, then~$\calP(j)$ also holds for every~$j\leqslant {\calJ}$.
\end{lemma}

Note that the above condition implies that the sequence~$(\alpha_j)_{j \in \N}$ is decreasing, meaning that the parameter in the exponential gets smaller at each step.
However, as we will see, we are able to choose these parameters such that~$\alpha_j$ does not tend to~$0$ when~$j\to\infty$.

We now conclude the proof of our bounds on~$\mu_c(\lambda)$ in dimension~$2$ in the two regimes of the parameter~$\lambda$.

\subsection{Low sleep rate: concluding proof of the bound~\reff{bound2Dlow}}
\label{proofBound2Dlow}

We now explain how to combine the ingredients to prove the upper bound~\reff{bound2Dlow} on the critical density when~$\lambda\to 0$.
Recall the definition of dormitory hierarchy which was given in section~\ref{subsectionDormitory}.
In this setting, we rely on the dormitory hierarchy given by the following Lemma:

\begin{lemma}
\label{lemmaDefHierarchy2Dlow}
Let~$d=2$, let~$D_0\geqslant 1$ and let~$D_j=6^jD_0$ for every~$j\geqslant 1$.
For every~$n\geqslant 1$ and every~$A\subset\Z_n^2$ with~$|A|\geqslant 288\,n^2/(D_0)^2$, there exists~${\calJ}\in\N$ and a~$(1,\,D)$-dormitory hierarchy~$(A_j,\,\calC_j)_{j\leqslant {\calJ}}$ on~$A$, with~$|A_{\calJ}|\geqslant |A|-144 n^2/(D_0)^2$ and where~$\calC_0$ contains only singletons.
\end{lemma}

The proof of this Lemma is deferred to section~\ref{sectionHierarchy}.

\begin{proof}[Proof of~\reff{bound2Dlow}]
Let~$d=2$. For every~$\lambda\in(0,1)$, we define, for every~$j\in\N$,
$$D_j\ =\ 6^j\Ceil{\frac{1}{\lambda}}
\quadet
\alpha_j
\ =\ \ln\left(\frac{1+2\lambda}{2\lambda}\right)
-\frac{a}{2}\big(1-2^{-j/4}\big)\ln|\ln\lambda|
\quadavec
a\ =\ \frac{2^{9/4}}{2^{1/4}-1}\,,$$
and we consider
$$\mu\ =\ \lambda\,|\ln\lambda|^{a}
\qquadet
\mu'\ =\ \mu-\frac{144}{(D_0)^2}\,.$$
Note that all these quantities are functions of the sleep rate~$\lambda$, although we omit to write the dependence in the notation.
In what follows, we assume that~$\lambda$ is small enough so that~$\mu<1$.

We start with the construction of the dormitory hierarchy.
When the sleep rate~$\lambda$ tends to~$0$, we have~\smash{$1/(D_0)^2\sim \lambda^2=o\big(\lambda|\ln\lambda|^{a}\big)=o(\mu)$}, whence~$\mu\geqslant 288/(D_0)^2$ provided that~$\lambda$ is chosen small enough.
This allows us to apply Lemma~\ref{lemmaDefHierarchy2Dlow} above to construct, for every~\smash{$n\geqslant 1$} and every~$A\subset\Z_n^2$ with~$|A|=\Ceil{\mu n^d}$, a dormitory hierarchy~$(A_j,\,\calC_j)_{j\leqslant {\calJ}}$ on~$A$ such that~$\calC_0$ contains only singletons and~$|A_{\calJ}|\geqslant |A|- 144 n^2/(D_0)^2\geqslant \mu' n^d$.

Then, we turn to the initialization step.
To prove that~$\calP(0)$ holds, notice that the stabilization of a singleton~$C\in\calC_0$ requires exactly~$\calS(C)=1$ sleep.
Hence, for every~$C\in\calC_0$, it follows from Lemma~\ref{lemmaNextColour} that~$1+\calL(C,\,0)$ is a geometric random variable with parameter
$$\frac{2\lambda}{1+2\lambda}
\ =\ e^{-\alpha_0}
\ =\ e^{-\alpha_0|C|}\,,$$
showing that~$\calP(0)$ holds with our definition of~$\alpha_0$.

We now wish to apply Lemma~\ref{lemmaInduction} (with~$v=1$) to perform the induction step.
To do so, we have to check that the condition~\reff{conditionInduction} is satisfied, that is to say, we want to show that
\begin{equation}
\label{defg}
g(\lambda)
\ :=\ \sup_{j\in\N}\,
\frac{1}{(\alpha_{j}-\alpha_{j+1})\,2^{j/2}}
\,\ln\Bigg[
\frac{4(1+\lambda)2^{3j/2}}
{\big(1-e^{-\alpha_{j}}\big)\Upsilon_2(D_{j})}\Bigg]
\ \leqslant\ 1\,.
\end{equation}
For every~$j\in\N$, we have
$$\alpha_{j}-\alpha_{j+1}
\ =\ \frac{a}{2}\big(2^{-j/4}-2^{-(j+1)/4}\big)\ln|\ln\lambda|
\ =\ \frac{a\big(2^{1/4}-1\big)}{2^{5/4+j/4}}\ln|\ln\lambda|
\ =\ \frac{2\ln|\ln\lambda|}{2^{j/4}}\,,$$
and Lemma~\ref{lemmaHarnack} tells us that there exists~$K>0$ such that, for every~$j\in\N$,
$$-\ln\Upsilon_2(D_j)
\ \leqslant\ \ln\left(\frac{\ln D_j}{K}\right)
\ \leqslant\ \ln\ln\Ceil{\frac{1}{\lambda}}+\frac{j\ln 6}{\ln\Ceil{1/\lambda}}-\ln K\,.$$
Plugging these into~\reff{defg}, we obtain that, when~$\lambda\to 0$,
\begin{align*}
g(\lambda)
&\ =\ \frac{1}{2\ln|\ln\lambda|}
\,\sup_{j\in\N}\,
\frac{2\ln 2
+\ln(1+\lambda)
+3j\ln 2/2
-\ln\big(1-e^{-\alpha_{j}}\big)
-\ln\Upsilon_2(D_{j})
}
{2^{j/4}}\\
&=\ \frac{\ln|\ln\lambda|+O(1)}
{2\ln|\ln\lambda|}
\ =\ \frac{1}{2}+o(1)\,.
\end{align*}
Thus, we have~$g(\lambda)\leqslant 1$ for~$\lambda$ small enough, meaning that the condition~\reff{conditionInduction} necessary to apply Lemma~\ref{lemmaInduction} (the induction step) is satisfied.
Thus, we deduce that~$\calP({\calJ})$ also holds, implying that~\smash{$1+\calL(A_{\calJ},\,{\calJ})$} dominates a geometric variable of parameter
$$e^{-\alpha_{\calJ} |A_{\calJ}|}
\ \leqslant\ e^{-\alpha_\infty\mu' n^d}
\qquadou
\alpha_\infty
\ =\ \inf_{j\in\N}\alpha_j
\ =\ \ln\left(\frac{1+2\lambda}{2\lambda}\right)
-\frac{a}{2}\ln|\ln\lambda|\,,$$
implying the same domination for~$\calH(A_{\calJ})$, since~$\calH(A_{\calJ})\geqslant |A_{\calJ}|+\calL(A_{\calJ})\geqslant 1+\calL(A_{\calJ},\,{\calJ})$.
To deduce that~$\mu\geqslant\mu_c(\lambda)$, we now rely on the sufficient condition given by Lemma~\ref{lemmaConditionActivityLoops}.
Thus, there only remains to check that~$\alpha_\infty\mu'>\psi(\mu)$.
On the one hand, when~$\lambda\to 0$, we have
\begin{align*}
\alpha_\infty\mu'
&\ =\ \left(|\ln\lambda|
-\frac{a}{2}\ln|\ln\lambda|
+O(1)\right)
\big(\lambda|\ln\lambda|^{a}+O(\lambda^2)\big)\\
&\ =\ \lambda|\ln\lambda|^{a+1}
-\frac{a}{2}\lambda|\ln\lambda|^{a}\ln|\ln\lambda|
+O\big(\lambda|\ln\lambda|^{a}\big)
\end{align*}
while, on the other hand, we have
\begin{align*}
\psi(\mu)
&\ =\ \lambda|\ln\lambda|^{a}
\big(|\ln\lambda|-a\ln|\ln\lambda|\big)
+O\big(\lambda|\ln\lambda|^{a}\big)\\
&\ =\ \lambda|\ln\lambda|^{a+1}
-a\lambda|\ln\lambda|^{a}\ln|\ln\lambda|
+O\big(\lambda|\ln\lambda|^{a}\big)\,,
\end{align*}
which implies that~$\alpha_\infty\mu'>\psi(\mu)$ for~$\lambda$ small enough, allowing us to deduce by virtue of Lemma~\ref{lemmaConditionActivityLoops} that~$\mu\geqslant\mu_c(\lambda)$, thereby concluding the proof of the bound~\reff{bound2Dlow}.
\end{proof}

\subsection{High sleep rate: concluding proof of the bound~\reff{bound2Dhigh}}
\label{proofBound2Dhigh}

We now turn to the proof of the upper bound on~$\mu_c(\lambda)$ when the sleep rate~$\lambda$ tends to infinity.
In this regime, we use the dormitory hierarchy given by Lemma~\ref{lemmaDefHierarchy2Dhigh} below, which is proved in section~\ref{sectionHierarchy}.
Recall the definition of~$r$-connectedness, which was given in section~\ref{subsectionNotation}.

\begin{lemma}
\label{lemmaDefHierarchy2Dhigh}
Let~$d=2$, let~$r\geqslant 1$ and~$D_j=6^j\times 96r^3$ for every~$j\in\N$.
For~$n\in\N$ large enough, for every~$A\subset\Z_n^2$ with~$|A|\geqslant n^2/2$, there exists~${\calJ}\in\N$ and a~$\big(r^2,\,D\big)$-dormitory hierarchy~$(A_j,\,\calC_j)_{j\leqslant {\calJ}}$ on~$A$, with~$|A_{\calJ}|\geqslant |A|-n^2/2$ and where every set~$C\in\calC_0$ is~$8r$-connected and satisfies
\begin{equation}
\label{conditionDense}
\forall x\in C\qquad
\big|C\cap B(x,\,4r)\big|
\ \geqslant\ r^2\,.
\end{equation}
\end{lemma}

In fact, one could also consider the simpler hierarchy in which the sets of~$\calC_0$ are simply the connected components of~$A$ which contain at least~$r^2$ vertices, which would also work to prove that~$\mu_c<1$ for all~$\lambda>0$.
However, this would yield a weaker bound on~$\mu_c$ because this can lead to throwing away too many vertices if for example~$A$ contains many connected components with strictly less than~$r^2$ points, hence our choice of a slightly weaker condition on the sets of~$\calC_0$.

The initialization step is performed in the following Lemma, which is proved in section~\ref{sectionInit}.
Let~$K>0$ be the constant given by Lemma~\ref{lemmaHarnack} in dimension~$d=2$.

\begin{lemma}\label{lemmaInit2Dhigh}
Assume that~$d=2$.
There exists~$\lambda_0>1$ such that, for every~$\lambda\geqslant\lambda_0$, defining
\begin{equation}
\label{paramsInit}
r\ =\ 
\Ceil{\frac{8(\ln\lambda)\sqrt{\lambda}}{\sqrt{K}}}
\qquadet
\alpha_0\ =\ 
\frac{K}{\lambda\ln\lambda}\,,
\end{equation}
for every~$D\in(\N\setminus\{0\})^\N$ and~$n\geqslant 1$, if~$A\subset\Z_n^2$ and~$(A_j,\,\calC_j)_{j\leqslant {\calJ}}$ is a~$(r^2,\,D)$-dormitory hierarchy on~$A$ such that every set~$C\in\calC_0$ is~$8r$-connected and \guillemets{dense} in the sense of~\reff{conditionDense}, then, for every~$C\in\calC_0$, there exists a~$C$-toppling procedure~$f$ such that~$1+\calL(C,\,0)$ dominates a geometric variable with parameter~\smash{$\exp\big(-\alpha_0 |C|\big)$}, that is to say, the property~$\calP(0)$ holds.
\end{lemma}

\begin{proof}[Proof of~\reff{bound2Dhigh}]
Let~$\lambda_0>1$ given by Lemma~\ref{lemmaInit2Dhigh}.
We consider the functions~$r$ and~$\alpha_0$ of~$\lambda$ defined in~\reff{paramsInit}, and, for~$j\in\N$, we write
$$\alpha_j
\ =\ \frac{1+2^{-j/4}}{2}\,\alpha_0
\qquadet
D_j
\ =\ 6^j\times 96r^3\,.$$
Let us first check that, with these parameters, the condition~\reff{conditionInduction} required to apply the induction step, namely Lemma~\ref{lemmaInduction} (with~$v=r^2$), is satisfied for~$\lambda$ large enough.
For every~$j\in\N$, we have
$$\alpha_{j}-\alpha_{j+1}
\ =\ \frac{2^{1/4}-1}{2^{5/4}}\,\frac{\alpha_0}{2^{j/4}}\,.$$
Then, using the estimate of Lemma~\ref{lemmaHarnack}, we can write
$$-\ln\Upsilon_2(D_j)
\ \leqslant\ \ln\left(\frac{\ln D_j}{K}\right)
\ =\ \ln\big(j\ln 6+\ln 96+3\ln r\big)-\ln K
\ \leqslant\ \ln\ln r+\ln j+K'\,,$$
for a certain fixed constant~$K'>0$.
Using this and noting that~$\alpha_j\geqslant\alpha_0/2$ for every~$j\in\N$ and that~$2^{5/4}/(2^{1/4}-1)< 16$, we have, when~$\lambda\to\infty$,
\begin{align*}
&\sup_{j\in\N}\,
\frac{1}{(\alpha_{j}-\alpha_{j+1})\,2^{j/2}r^2}
\,\ln\left[
\frac{4r^2(1+\lambda)2^{3j/2}}
{\big(1-e^{-\alpha_{j}r^2}\big)
\Upsilon_2(D_{j})
}\right]
\\
&\leqslant\ \frac{16}{\alpha_0 r^2}
\,\sup_{j\in\N}\,
\frac{
2\ln 2
+2\ln r
+\ln(1+\lambda)
+3j\ln 2/2
-\ln\big(1-e^{-\alpha_{j}r^2}\big)
-\ln\Upsilon_2(D_{j})}
{2^{j/4}}\\
&=\ \frac{1}{4\ln\lambda}\Big[\,
2\ln r
+\ln\lambda
+O\big(e^{-\alpha_0 r^2/2}\big)
+O(\ln\ln r)
+O(1)
\,\Big]
\ =\ \frac{1}{2}+o(1)\,.
\end{align*}
Therefore, the condition~\reff{conditionInduction} is satisfied for~$\lambda$ large enough.
Let~$\lambda_1\geqslant \lambda_0$ such that this condition is satisfied for all~$\lambda\geqslant \lambda_1$, and such that~$4\lambda_1(\ln\lambda_1)^2\geqslant K$.
We now take~$\lambda\geqslant \lambda_1$ and we define
$$\mu\ =\ 1-\frac{K}{8\lambda(\ln\lambda)^2}\,.$$
Let~$n\geqslant 1$, and let~$A\subset\Z_n^2$ with~$|A|=\Ceil{\mu n^d}$ (which implies that~$|A|\geqslant n^2/2$, since~$\mu\geqslant 1/2$).
We now consider a dormitory hierarchy~$(A_j,\,\calC_j)_{j\leqslant {\calJ}}$ given by Lemma~\ref{lemmaDefHierarchy2Dhigh}, so that we have
$$|A_{\calJ}|
\ \geqslant\ 
|A|
-\frac{n^2}{2}
\ \geqslant\ \left(\mu-\demi\right)n^2\,.$$
First, Lemma~\ref{lemmaInit2Dhigh} entails that for every~$C\in\calC_0$, there exists a~$C$-toppling procedure~$f_C$ such that the property~$\calP(0)$ holds with these procedures.
Then, Lemma~\ref{lemmaInduction} ensures that~\smash{$1+\calL(A_{\calJ},\,{\calJ})$} dominates a geometric random variable of parameter
$$\exp\big(-\alpha_{\calJ} |A_{\calJ}|\big)
\ \leqslant\ \exp\big(-\kappa n^2\big)
\qquadavec
\kappa\ =\ \frac{\alpha_0}{2}\left(\mu-\demi\right)\,.$$
We now check that~$\kappa>\psi(\mu)$ for~$\lambda$ large enough, in order to apply~Lemma~\ref{lemmaConditionActivityLoops}.
When~$\lambda\to\infty$, we have
\begin{align*}
\kappa-\psi(\mu)
&\ =\ \frac{\alpha_0}{2}\left(\mu-\demi\right)-\psi(\mu)\\
&\ =\ 
\frac{\alpha_0}{2}
-\frac{\alpha_0 K}{16\lambda(\ln\lambda)^2}
-\frac{\alpha_0}{4}
-\mu|\ln\mu|
-(1-\mu)\big|\ln(1-\mu)\big|\\
&\ =\ 
\frac{K}{4\lambda\ln\lambda}
+O\left(\frac{1}{\lambda^2(\ln\lambda)^3}\right)
-\frac{K}{8\lambda\ln\lambda}
+O\left(\frac{\ln\ln\lambda}{\lambda(\ln\lambda)^2}\right)\\
&\ =\ \frac{K}{8\lambda\ln\lambda}
+o\left(\frac{1}{\lambda\ln\lambda}\right)\,,
\end{align*}
which shows that this quantity is strictly positive when~$\lambda$ is large enough, concluding our proof that~$\mu\geqslant\mu_c(\lambda)$, leading to the claimed upper bound~\reff{bound2Dhigh}, with~$c=K/8$.
\end{proof}

\section{The transient case: proof of Theorem~\ref{thm3D}}
\label{section3D}

We now turn to the simpler case of dimension~$d\geqslant 3$.
Given~$n\geqslant 1$ and~$A\subset\torus$, we consider the trivial hierarchy with only one level, that is to say,~${\calJ}=0$,~$A_0=A$ and~$\calC_0=\{A\}$.
Then, as explained in paragraph~\ref{sketch3D}, compared to the recursive proof in dimension~$2$, we only keep the initialization step.
Indeed, we will see that the results in dimension~$d\geqslant 3$ are easy consequences of the initialization step of the previous proof, more precisely of Lemma~\ref{lemmaMetastability}.

\subsection{Low sleep rate: proof of the bound~\reff{bound3Dlow}}
\label{subsectionProofBound3Dlow}

Let~$d\geqslant 3$, and let~$K>0$ be the constant given by Lemma~\ref{lemmaHarnack} associated with~$d$ (note that we have~$K<1$).
For every~$\lambda<K^8/e$ and~$n\geqslant 1$, we consider
$$\mu\ =\ \frac{e}{K^8}\,\lambda\,,\qquad
v\ =\ \big\lceil\mu n^d\big\rceil\,,\qquad
\alpha
\ =\ |\ln\lambda|-2|\ln K|
\qquadet
\beta
\ =\ 1-\frac{2|\ln K|}{|\ln\lambda|}\,.$$
We show that, provided that~$\lambda$ is small enough, we have~$\mu\geqslant\mu_c(\lambda)$.
To this end, we wish to apply Lemma~\ref{lemmaMetastability} with these parameters and~$r=+\infty$, and for this we have to check the condition~\reff{conditionMetastability}.
For every fixed~$\lambda<K^8/e$, since~$\alpha>0$ and~$\beta\in(0,1)$, we have
$$\lim_{n\to\infty}\,\exp\Big[\alpha\big(1-(1-\beta)v\big)\Big]
\ \leqslant\ \lim_{n\to\infty}\,\exp\Big[\alpha\big(1-(1-\beta)\mu n^d\big)\Big]
\ =\ 0\,,$$
implying that, for~$n$ large enough, we have
$$\lambda\big(e^\alpha-1\big)
\ =\ K^2-\lambda
\ \leqslant\ K^2
\ \leqslant\ K\big(1-e^{\alpha(1-(1-\beta)v)}\big)\,,$$
which is precisely the required condition~\reff{conditionMetastability}.
Thus, we may apply Lemma~\ref{lemmaMetastability} to deduce that, for~$n$ large enough, for every~$A\subset\torus$ with~$|A|=\Ceil{\mu n^d}$, the variable~$\calN_{\calB}$ (the number of visits of the set of configurations with \guillemets{many} active particles, as defined in the statement of Lemma~\ref{lemmaMetastability}) dominates a geometric variable with parameter
$$\exp\Big(-\alpha\big\lfloor\beta|A|\big\rfloor\Big)
\ \leqslant\ \exp\Big(-\alpha\big\lfloor\beta\mu n^d\big\rfloor\Big)\,,$$
and so does the number of topplings~$\calH(A)$, since~$\calH(A)\geqslant\calN_\calB$.
Note now that, when~$\lambda\to 0$,
$$\alpha\beta\mu
\ =\ \left(1-\frac{2|\ln K|}{|\ln\lambda|}\right)^2
\frac{e}{K^8}\,\lambda|\ln\lambda|
\ =\ \frac{e}{K^8}\,\lambda|\ln\lambda|
-\frac{4e|\ln K|}{K^8}\,\lambda
+o(\lambda)\,,$$
while
$$\psi(\mu)
\ =\ \mu|\ln\mu|+(1-\mu)|\ln(1-\mu)|
\ =\ \frac{e}{K^8}\,\lambda|\ln\lambda|
-\frac{8e|\ln K|}{K^8}\,\lambda
+o(\lambda)\,.$$
Therefore, writing
$$\kappa\ =\ \frac{e}{K^8}\,\lambda|\ln\lambda|
-\frac{6e|\ln K|}{K^8}\,\lambda\,,$$
we deduce that, for~$\lambda$ large enough, for~$n$ large enough, we have
$$\exp\Big(-\alpha\big\lfloor\beta\mu n^d\big\rfloor\Big)
\ \leqslant\ \exp\big(-\kappa\,n^d\big)
\qquadet
\kappa\ >\ \psi(\mu)\,,$$
implying that~$\mu\geqslant\mu_c(\lambda)$ by virtue of Lemma~\ref{lemmaConditionActivityLoops}, which shows~\reff{bound3Dlow} with~$c=e/K^8$.\hfill\qed

\subsection{High sleep rate: proof of the bound~\reff{bound3Dhigh}}

Let~$d\geqslant 3$, and let~$K>0$ be the constant given by Lemma~\ref{lemmaHarnack} applied in dimension~$d$.
Let~$\lambda>1$ and let
$$\mu\ =\ 1-\frac{K}{16\lambda\ln\lambda}\,.$$
We assume that~$\lambda$ is large enough so that~$\mu>0$.
Let~$\alpha=K/(2\lambda)$, let~$\beta=1/2$, let~$n\geqslant 1$ and~\smash{$v=\Ceil{\mu n^d}$}.
As in section~\ref{subsectionProofBound3Dlow}, the condition~\reff{conditionMetastability} is satisfied provided that~$\lambda$ and~$n$ are large enough, since~\smash{$\lambda\big(e^{\alpha}-1\big)\sim\lambda\alpha=K/2$} when~$\lambda\to\infty$.
Thus, it follows from Lemma~\ref{lemmaMetastability} that, for every~$A\subset\torus$ with~$|A|=\Ceil{\mu n^d}$, the variable~$\calN_{\calB}$ dominates a geometric variable with parameter
$$\exp\left(-\alpha\Ent{\frac{|A|}{2}}\right)
\ =\ \exp\left(-\frac{K}{2\lambda}\Ent{\frac{\Ceil{\mu n^d}}{2}}\right)
\ \leqslant\ \exp\left(-\frac{K}{8\lambda}\,n^d\right)\,,$$
for~$\lambda$ and~$n$ large enough.
Now, when~$\lambda\to\infty$, we have
$$\psi(\mu)
\ =\ \psi\left(\frac{K}{16\lambda\ln\lambda}\right)
\ \sim\ \frac{K}{16\lambda\ln\lambda}
\ln\left(\frac{16\lambda\ln\lambda}{K}\right)
\ \sim\ \frac{K}{16\lambda}\,.$$
Thus, for~$\lambda$ large enough, we have~$K/(8\lambda)>\psi(\mu)$, which allows us to apply Lemma~\ref{lemmaConditionActivityLoops} to deduce that~$\mu\geqslant\mu_c(\lambda)$, concluding the proof of the upper bound~\reff{bound3Dhigh} with~$c=K/16$.\hfill\qed

\section{Construction of the dormitory hierarchy}
\label{sectionHierarchy}

The goal of this section is to detail the construction the hierarchical structure on the settling set~$A$.
The method differs depending on the regime of sleep rate considered, but the only difference is the definition of~$\calC_0$.
Hence, we start by explaining the recursive construction of the hierarchy once~$\calC_0$ is defined, which is common in the two regimes, before detailing the construction of this first level in the two regimes.

\subsection{Inductive construction}

\begin{lemma}
\label{lemmaDefHierarchy}
Let~$d,\,n,\,v,\,D_0\geqslant 1$, let~$D_j=6^jD_0$ for every~$j\geqslant 1$ and let~$A\subset\torus$.
Assume that~$A_0\subset A$ is such that~$|A_0|\geqslant 8v(6n/D_0)^d$ and that~$\calC_0$ is a partition of~$A_0$ such that, every cluster~$C\in\calC_0$ is~\smash{$\Ent{D_0/(12v)}$}-connected and satisfies~$|C|\geqslant v$.
Then, one can complete~$(A_0,\,\calC_0)$ into a~$(v,\,D)$-dormitory hierarchy~$(A_j,\,\calC_j)_{j\leqslant {\calJ}}$ on~$A$ such that
\begin{equation}
\label{lowerBoundAJ}
|A_{\calJ}|
\ \geqslant\ |A_0|
-4v\left(\frac{6n}{D_0}\right)^d\,.
\end{equation}
\end{lemma}

We proceed by induction, constructing two levels at each step.
The idea is that, once the partition~$\calC_{2j}$ is defined, we construct~$\calC_{2j+1}$ and~$\calC_{2j+2}$ by keeping the clusters~\smash{$C\in\calC_{2j}$} which are of size~$|C|\geqslant 2^{j+1}v$ and merging as many pairs or triples of the remaining clusters as possible, while ensuring that doing so, we do not create any cluster of diameter larger than~$D_{2j}$.
The remaining clusters of~$\calC_{2j}$ are thrown away.
To merge three clusters of~$\calC_{2j}$, we simply merge two of them in~$\calC_{2j+1}$, before merging the resulting cluster with the third one when passing to~$\calC_{2j+2}$.

\begin{proof}
We construct the sequence~$A_j$ and the partitions~$\calC_j$ recursively.
To obtain the lower bound~\reff{lowerBoundAJ}, it is enough to ensure that, at each level~$j$, we have
\begin{equation}
\label{proofCardAj}
\big|A_0\setminus A_j\big|
\ \leqslant\ 
2v(6n)^d\,
\sum_{i=0}^{\Ent{(j-1)/2}} \frac{2^{i}}{(D_{2i})^d}\,.
\end{equation}
Indeed, if the above bound is satisfied, then we have
$$\big|A_0\setminus A_{\calJ}\big|
\ \leqslant\ 
2v\,\left(\frac{6n}{D_0}\right)^d\,
\sum_{i=0}^{+\infty}\left(\frac{2}{6^{2d}}\right)^i
\ =\ 
2v\,\left(\frac{6n}{D_0}\right)^d
\times\frac{1}{1-2/6^{2d}}
\ \leqslant\ 
4v\,\left(\frac{6n}{D_0}\right)^d\,.$$

Now assume that~$j\in\N$ and that the sets~$A_0,\,\ldots,\,A_{2j}$ and the partitions~$\calC_0,\,\ldots,\,\calC_{2j}$ are constructed and satisfy the two requirements~\textit{(\ref{conditionCard})-(\ref{conditionMerge})} and the property~\reff{proofCardAj}.
Note that this latter property, together with the assumption~$|A_0|\geqslant 4v(6n/D_0)^d$, ensure that~$A_{2j}\neq\varnothing$.
If~$|\calC_{2j}|=1$, we let~${\calJ}=2j$ and we stop here, the hierarchy being complete.
Thus, we now assume that~$|\calC_{2j}|\geqslant 2$.
Defining the sets ($\calN$ stands for \guillemets{narrow},~$\calM$ stands for \guillemets{merging},~$\calB$ for \guillemets{big} and~$\calR$ for \guillemets{rubbish})
$$\calN
\ =\ \bigg\{C\in\calC_{2j}\ :\ \diam\,C\,\leqslant\,\frac{D_{2j}}{6}\bigg\}\,,$$
$$\calM\ =\ \bigg\{C\in\calN\ :\ \exists\,C'\in\calN\setminus\{C\}\,,\ \diam\big(C\cup C'\big)\,\leqslant\,\frac{D_{2j}}{2}\bigg\}\,,$$
$$\calB
\ =\ \Big\{C\in\calN\setminus\calM\quad :\quad
|C|\,\geqslant\,2^{j+1} v\Big\}
\qquadet
\calR
\ =\ \calN\setminus\big(\calM\cup\calB\big)\,,$$
we let
$$A_{2j+2}
\ =\ A_{2j+1}
\ =\ \bigcup_{C\,\in\,\calC_{2j}\setminus\calR} C
\ =\ \bigcup_{C\,\in\,\calM\,\cup\,\calB\,\cup\,(\calC_{2j}\setminus\calN)} C\,,$$
meaning that~$\calR$ corresponds to the part that is thrown away between~$A_{2j}$ and~$A_{2j+1}$.

To construct the partitions~$\calC_{2j+1}$ and~$\calC_{2j+2}$, we consider the graph structure on the set~$\calM$ which is obtained by declaring any two sets~$C_1,\,C_2\in\calM$ to be neighbours if and only if~\smash{$\diam(C_1\cup C_2)\leqslant D_{2j}/2$}.
This yields a graph with no isolated point, by definition of~$\calM$.
Now, we have the following Lemma whose simple proof is deferred to the
appendix~\ref{sectionProofLemmaGroup2or3}:

\begin{lemma}
\label{lemmaGroup2or3}
For every finite undirected graph~$G=(V,\,E)$ with no isolated point, there exists a partition of~$V$ into sets of cardinality~$2$ or~$3$ and diameter (for the graph distance on~$G$) at most~$2$.
\end{lemma}

This Lemma yields a partition~$\Pi$ of~$\calM$ into 
sets of cardinality~$2$ or~$3$ and diameter (for the graph distance we just defined) at most~$2$.
The partition~$\calC_{2j+2}$ is then obtained by keeping the sets in~$\calC_{2j}\setminus\calN$ and in~$\calB$ as they are and by merging the sets in~$\calM$ in pairs or triples according to this partition~$\Pi$.
Formally, we let
$$\calC_{2j+2}
\ =\ \big(\calC_{2j}\setminus\calN\big)
\,\cup\,\calB
\,\cup\,\Bigg\{\,\bigcup_{C\in P} C\,,\ \  P\in\Pi\,\Bigg\}\,.$$

To construct the intermediate partition~$\calC_{2j+1}$, we consider a partition~$\Pi'$ obtained from~$\Pi$ by keeping the sets of cardinality~$2$ and splitting each set of cardinality~$3$ into one set of cardinality~$2$ and one singleton (chosen arbitrarily), and we let
$$\calC_{2j+1}
\ =\ \big(\calC_{2j}\setminus\calN\big)
\,\cup\,\calB
\,\cup\,\Bigg\{\,\bigcup_{C\in P} C\,,\ \  P\in\Pi'\,\Bigg\}\,.$$

We now check that the above construction satisfies the required conditions.
First, merging two or three clusters of size at least~$2^{j}v$ yields a cluster of size at least~$2^{j+1} v$.
Besides, if~$C\in\calC_{2j}\setminus\calN$, meaning that~$\diam\,C>D_{2j}/6= D_{2j-1}$, then the condition~\textit{(\ref{conditionMerge})} tells us that~$C$ does not come from a previous merging, meaning that~$C\in\calC_0$, which implies that~$C$ is~\smash{$\Ent{D_0/(12v)}$}-connected, whence
$$|C|
\ \geqslant\ \frac{\diam\,C}{D_0/(12v)}
\ \geqslant\ \frac{2v D_{2j}}{D_0}
\ =\ 2\times 6^{2j}v
\ \geqslant\ 2^{j+1} v\,.$$
Therefore, the condition~\textit{(\ref{conditionCard})} holds at ranks~$2j+1$ and~$2j+2$.

To prove~\textit{(\ref{conditionMerge})}, it remains to check that, for every~$P\in\Pi$, we have
$$\diam\ \bigcup_{C\in P} C
\ =\ \sup_{C_1,\,C_2\in P}
\ \sup_{x\in C_1,\,y\in C_2}
\,d(x,\,y)
\ \leqslant\ D_{2j}\,.$$
Let~$P\in\Pi$, let~$C_1,\,C_2\in P$, let~$x\in C_1$ and~$y\in C_2$.
By definition of~$\Pi$, we know that the graph distance between~$C_1$ and~$C_2$ in~$\calM$ is at most~$2$.
If~$C_1=C_2$ or if the graph distance is~$1$, then we have~$d(x,\,y)\leqslant \diam(C_1\cup C_2)\leqslant D_{2j}/2\leqslant D_{2j}$.
If the graph distance between~$C_1$ and~$C_2$ in~$\calM$ is~$2$, then there exists~$C_3\in\calM$ such that~$\diam(C_1\cup C_3)\leqslant D_{2j}/2$ and~\smash{$\diam(C_2\cup C_3)\leqslant D_{2j}/2$}.
Choosing an arbitrary~$z\in C_3$, we deduce that
$$d(x,\,y)
\ \leqslant\ d(x,\,z)+d(z,\,y)
\ \leqslant\ \diam\big(C_1\cup C_3\big)+\diam\big(C_3\cup C_2\big)
\ \leqslant\ D_{2j}\,,$$
which shows that, in all cases,~\textit{(\ref{conditionMerge})} holds at ranks~$2j+1$ and~$2j+2$.

To prove that~\reff{proofCardAj} remains true for~$2j+1$ and~$2j+2$., we only need to show that
\begin{equation}
\label{throwAwayEachStep}
\big|A_{2j+1}\setminus A_{2j}\big|
\ \leqslant\ 2^{j+1}v\left(\frac{6n}{D_{2j}}\right)^d\,.
\end{equation}
By definition of~$\calR$, we know that for every cluster~$C\in\calR$, we have~$|C|<2^{j+1}v$ and~\smash{$\diam\,C\leqslant D_{2j}/6$}.
Thus, we may write
\begin{equation}
\label{controlSingletons}
\big|A_{2j+1}\setminus A_{2j}\big|
\ =\ \bigg|\bigcup_{C\,\in\,\calR} C\,\bigg|
\ \leqslant\ \abs{\calR} 2^{j+1} v\,.
\end{equation}
If~$C_1,\,C_2\in\calR$ with~$C_1\neq C_2$, we have~$\diam(C_1\cup C_2)>D_{2j}/2$ whence, by the triangle inequality,
$$d(C_1,\,C_2)
\ \geqslant\ \diam(C_1\cup C_2)-\diam\,C_1-\diam\,C_2\\
\ >\ \frac{D_{2j}}{2}-2\,\frac{D_{2j}}{6}\\
\ =\ \frac{D_{2j}}{6}\,.$$
We now distinguish between two cases.

On the one hand, if~$D_{2j}/6\geqslant n$, then~$\calN=\calC_{2j}$ and, since we assumed that~$|\calC_{2j}|\geqslant 2$, we also have~$\calM=\calN=\calC_{2j}$, whence~$\calR=\varnothing$.

On the other hand, if~$D_{2j}/6< n$, choosing a vertex in each set of~$\calR$, we deduce that the closed balls of radius~$\Ceil{D_{2j}/12}$ centered on these points are pairwise disjoint, whence, using our formula~\reff{volumeBall} for the volume of the ball,
$$|\calR|\ \leqslant\ \frac{\abs{\torus}}{\big(2\Ceil{D_{2j}/12}+1\big)^d}
\ =\ \left(\frac{6n}{D_{2j}}\right)^d\,.$$
In both cases, going back to~(\ref{controlSingletons}), we obtain~\reff{throwAwayEachStep}, implying that the property~\reff{proofCardAj} is inherited at rank~$2j+1$, and thus also at rank~$2j+2$, since~$A_{2j+2}=A_{2j+1}$.
The result thus follows by induction.
\end{proof}

\subsection{Dormitory hierarchy for low sleep rate: proof of Lemma~\ref{lemmaDefHierarchy2Dlow}}

With the above recursive construction, we easily obtain Lemma~\ref{lemmaDefHierarchy2Dlow}.

\begin{proof}[Proof of Lemma~\ref{lemmaDefHierarchy2Dlow}]
Let~$D_0\geqslant 1$, let~$n\geqslant 1$ and~$A\subset\Z_n^2$ with~$A\geqslant 288 n^2/(D_0)^2$.
For first level of the hierarchy, we simply take~$A_0=A$ and~\smash{$\calC_0=\big\{\{x\},\,x\in A\big\}$}.
The result then follows by applying Lemma~\ref{lemmaDefHierarchy} with~$v=1$ to construct the rest of the hierarchy.
\end{proof}

\subsection{Dormitory hierarchy for high sleep rate: proof of Lemma~\ref{lemmaDefHierarchy2Dhigh}}

We now turn to the proof of Lemma~\ref{lemmaDefHierarchy2Dhigh}.

\begin{proof}[Proof of Lemma~\ref{lemmaDefHierarchy2Dhigh}]
Let~$r\geqslant 1$, let~$D_j=6^j\times 96r^3$, let~$n\geqslant 2r+1$ and let~$A\subset\Z_n^2$ such that~\smash{$|A|\geqslant n^2/2$}.
The set~$A_0$ can be defined as
$$A_0\ =\ \Big\{\,x\in A\quad :\quad 
\exists\,y\in B(x,\,2r)\,,\quad
\big|A\cap B(y,\,2r)\big|\ \geqslant\ r^2\,\Big\}\,.$$
Then, we consider the relation on this set~$A_0$ obtained by declaring two sites~$x,\,y\in A$ to be connected if~$d(x,\,y)\leqslant 8r$, and we simply define~$\calC_0$ as the collection of the connected components of~$A_0$ for this relation (in a word, the~$8r$-connected components of~$A_0$).

The property~\reff{conditionDense} follows from this construction, implying also that~\textit{(\ref{conditionCard})} holds at rank~$j=0$ (with~$v=r^2$).
We now look for a lower bound on~$|A_0|$.
To this end, we claim that, for every~$x\in\Z_n^2$, we have
\begin{equation}
\label{claimA0}
\big|(A\setminus A_0)\cap B(x,\,r)\big|
\ <\ r^2\,.
\end{equation}
Indeed, if~$x\in\Z_n^2$ does not satisfy the above inequality, then for every~\smash{$y\in(A\setminus A_0)\cap B(x,\,r)$}, we have
$$\big|A\cap B(y,\,2r)\big|
\ \geqslant\ \big|A\cap B(x,\,r)\big|
\ \geqslant\ \big|(A\setminus A_0)\cap B(x,\,r)\big|
\ \geqslant\ r^2\,,$$
which contradicts the fact that~$y\notin A_0$.
Using this claim~\reff{claimA0}, we may write (denoting by~$0$ an arbitrary point of the torus),
\begin{multline*}
\big|A\setminus A_0\big|
\times \big|B(0,\,r)\big|
\ =\ \sum_{y\in\Z_n^2}
\mathbf{1}_{\{y\in A\setminus A_0\}}\,
\big|B(y,\,r)\big|
\ =\ \sum_{y\in\Z_n^2}
\mathbf{1}_{\{y\in A\setminus A_0\}}\,
\sum_{x\in\Z_n^2}
\mathbf{1}_{\{d(x,\,y)\leqslant r\}}\\
\ =\ \sum_{x\in\Z_n^2}\,
\sum_{y\in\Z_n^2}
\mathbf{1}_{\{y\in A\setminus A_0\}}\,
\mathbf{1}_{\{d(x,\,y)\leqslant r\}}
\ =\ \sum_{x\in\Z_n^2}
\big|(A\setminus A_0)\cap B(x,\,r)\big|
\ \leqslant\ r^2\,n^2\,,
\end{multline*}
which implies that
$$\big|A\setminus A_0\big|
\ \leqslant\ 
\frac{r^2\,n^2}{\big|B(0,\,r)\big|}
\ =\ \frac{r^2\,n^2}{(2r+1)^2}
\ \leqslant\ \frac{n^2}{4}\,,$$
using that~$n\geqslant 2r+1$.
Since~$n^2/4<n^2/(64r^4)=4v(6n/D_0)^2$ and~$8r=\Ent{D_0/(12v)}$, we may apply Lemma~\ref{lemmaDefHierarchy} with~$v=r^2$ to obtain the whole hierarchy.
\end{proof}

\section{Initialization: proof of Lemma~\ref{lemmaInit2Dhigh}}
\label{sectionInit}

The goal of this section is to prove that, for every set~$C\in\calC_0$, the number of loops of colour~$0$ produced by the distinguished vertex~$\dist_C$ while~$C$ is stabilized (using an appropriate toppling procedure) is exponentially large in the size of~$C$.

\subsection{Active particles amidst sleeping ones}

The following Lemma tells us that, when at least a fraction of the particles are already sleeping, one can find an active particle surrounded by many sleeping particles so that, if toppled, this particle has good chances to wake up many particles.
Recall the definition of~$r$-connectedness which was given in paragraph~\ref{subsectionNotation}.

\begin{lemma}
\label{lemmaActiveSleepingNearby}
Let~$d,\,n,\,r,\,v\geqslant 1$ and~$\beta\in(0,1)$.
Let~$C\subset\torus$ be a non-empty and~$8r$-connected subset of the torus such that, for every~$x\in C$, we have~$|C\cap B(x,\,4r)|\geqslant v$.
Then, for every~$R\subset C$ such that~$0<|R|\leqslant \beta|C|$, there exists~$x_0\in R$ such that
$$\big|(C\setminus R)\cap B(x_0,\,16r)\big|\ \geqslant\ (1-\beta)\,v\,.$$
\end{lemma}

\begin{proof}
Let~$d,\,n,\,r,\,v\geqslant 1$,~$\beta\in(0,1)$, and~$C\subset\torus$ be as in the statement, and let~$R\subset C$ such that~$0<|R|\leqslant \beta|C|$.
Let us consider the set
$$\calY\ =\ \Big\{\,Y\subset C\ :\ \forall\,y,\,y'\in Y\,,\ y\neq y'\ \Rightarrow\ d(y,\,y')> 8r\,\Big\}\,,$$
and choose~$Y\in\calY$ which is maximal for inclusion (we have~$\calY\neq\varnothing$ since~$\varnothing\in\calY$).
Using our assumption on~$C$, we can write
\begin{equation}
\label{lowerBoundCardinalC}
|C|\ \geqslant\ \bigg|C\cap\bigcup_{y\in Y} B(y,\,4r)\bigg|
\ =\ \sum_{y\in Y}\big|C\cap B(y,\,4r)\big|
\ \geqslant\ v\,|Y|\,,
\end{equation}
where we used that the above union is disjoint.
We now notice that
$$C\ \subset\ \bigcup_{y\in Y}B(y,\,8r)\,,$$
otherwise there would exist a point~$z\in C\setminus \cup_{y\in Y}B(y,\,8r)$ and we would have~$Y\sqcup\{z\}\in\calY$, which would contradict the maximality of~$Y$.
Therefore, we have
$$\big|C\setminus R\big|
\ =\ \bigg|\bigcup_{y\in Y}\big((C\setminus R)\cap B(y,\,8r)\big)\bigg|
\ \leqslant\ \sum_{y\in Y}\big|(C\setminus R)\cap B(y,\,8r)\big|\,.$$
Combining this with~(\ref{lowerBoundCardinalC}) and recalling that~$|C\setminus R|\geqslant (1-\beta)|C|$, we obtain that
$$\sum_{y\in Y}\big|(C\setminus R)
\cap B(y,\,8r)\big|
\ \geqslant\ \big|C\setminus R\big|
\ \geqslant\ (1-\beta)|C|
\ \geqslant\ (1-\beta)\,v\,|Y|\,.$$
Therefore, by the pigeonhole principle, we can find~$y\in Y$ such that
\begin{equation}
\label{denseBowl}
\big|(C\setminus R)\cap B(y,\,8r)\big|
\ \geqslant\ (1-\beta)\,v\,.
\end{equation}
If~$y\in R$, then~$x_0=y$ is a solution of the problem.
Thus, we assume henceforth that~$y\notin R$.
Since~$C$ is~$8r$-connected and~$R\neq\varnothing$, we may consider a path~$y=y_0,\,y_1,\,\ldots,\,y_k$ with~\smash{$k\in\N$},~$y_k\in R$ and, for every~$j<k$,~$y_j\in C\setminus R$ and~$d(y_{j},\,y_{j+1})\leqslant 8r$.
We now consider
$$j_0
\ =\ \min\,\Big\{j\in\{0,\,\ldots,\,k\}\ :\ B(y_j,\,4r)\cap R\neq\varnothing\,\Big\}\,,$$
that is to say, we look at the first point of the path which sees an active site nearby.
The above set is non-empty because at least~$y_k\in R$.
We then choose a point~\smash{$x_0\in B(y_{j_0},\,4r)\cap R$}.
If~$j_0=0$, then we have
$$B(x_0,\,16r)
\ \supset\ B(y_0,\,8r)
\ =\ B(y,\,8r)\,,$$
whence, recalling~\reff{denseBowl},
$$\big|(C\setminus R)\cap B(x_0,\,16r)\big|
\ \geqslant\ \big|(C\setminus R)\cap B(y,\,8r)\big|
\ \geqslant\ (1-\beta)\,v\,.$$
Otherwise, if~$j_0>0$, then, by minimality of~$j_0$, we have~$B(y_{j_0-1},\,4r)\cap R
=\varnothing$.
However, following our assumption on~$C$, we know that~$|C\cap B(y_{j_0-1},\,4r)|\geqslant v$.
Together, these two facts imply that
$$\big|(C\setminus R)\cap B(y_{j_0-1},\,4r)\big|
\ \geqslant\ v\,.$$
Now, since
$$d\big(x_0,\,y_{j_0-1}\big)
\ \leqslant\ d\big(x_0,\,y_{j_0}\big)
+d\big(y_{j_0},\,y_{j_0-1}\big)
\ \leqslant\ 4r+8r
\ =\ 12r\,,$$
we deduce that~$B(x_0,\,16r)
\supset B(y_{j_0-1},\,4r)$, whence
$$\big|(C\setminus R)\cap B(x_0,\,16r)\big|
\ \geqslant\ v
\ \geqslant\ (1-\beta)\,v\,,$$
concluding the proof of the Lemma.
\end{proof}

\subsection{The metastability phenomenon}

By the metastability phenomenon we mean that, starting from everyone active inside a set~$C\in\calC_0$, during the stabilization of~$C$ the configuration escapes and returns an exponentially large number of times in the set of configurations with \guillemets{many} active particles.
This Lemma is used again when~$d\geqslant 3$ (see section~\ref{section3D}), taking~$r=+\infty$ (with the convention that~$\Upsilon_d(\infty)=\lim_{r\to\infty}\Upsilon_d(r)$).
Recall the definition of toppling procedures which was given in section~\ref{subsectionLoopRep}.

\begin{lemma}
\label{lemmaMetastability}
Let~$d,\,v\geqslant 1$, let~$\lambda,\,\alpha>0$,~$\beta\in(0,1)$ and~$r\in\N\cup\{\infty\}\setminus\{0\}$ such that
\begin{equation}
\label{conditionMetastability}
\lambda\big(e^{\alpha}-1\big)
\ \leqslant\ \Upsilon_d(16r)\big(1-e^{\alpha(1-(1-\beta)v)}\big)\,.
\end{equation}
For every~$n\geqslant 1$ and every subset~$A\subset\torus$, if~$(A_j,\,\calC_j)_{j\leqslant {\calJ}}$ is a dormitory hierarchy on~$A$ such that~$|C\cap B(x,\,4r)|\geqslant v$ for every~$C\in\calC_0$ and every~$x\in C$, then, for every~$C\in\calC_0$ there exists a~$C$-toppling procedure~$f$ such that, considering the set of configurations
$$\calB
\ =\ \Big\{\,R\subset A\quad:\quad|R\cap C|\,>\,\beta|C|\,\Big\}\,,$$
the number~$\calN_{\calB}$ of visits of this set~$\calB$ during the stabilization of~$C$ (i.e., one plus the number of returns to~$\calB$ from a configuration out of~$\calB$) dominates a geometric random variable with parameter~\smash{$\exp\big(-\alpha\Ent{\beta|C|}\big)$}.
Moreover, this domination is uniform with respect to the first configuration~$R\notin\calB$ reached by the system, that is to say, writing, for every~$t\in\N$,
$$(R_t,\,h_t,\,\ell_t)
\ =\ \big(\Phi_{C}\big)^{(t)}(A,\,0,\,0)
\ =\ \big(\Phi_{C}\big)^{(t)}(C,\,0,\,0)$$
and defining~$\tau=\inf\{t\in\N\,:\,R_t\notin\calB\}$, then for every~$R\subset A$ such that~$\Proba(R_\tau=R)>0$, conditionally on~$\{R_\tau=R\}$, we have the aforementioned stochastic domination.
\end{lemma}

\begin{proof}
With the notation of the statement, Lemma~\ref{lemmaActiveSleepingNearby} (which remains true if~$r=\infty$) tells us that there exists a function~\smash{$f_0:\calP(C)\setminus\{\varnothing\}\to C$} such that, for every~$R\subset C$ with~$0<|R|\leqslant \beta|C|$ (for the other configurations, the value of~$f_0$ can be chosen arbitrarily), we have~$f_0(R)\in R$ and
\begin{equation}
\label{expCCaround}
\big|\big(C\setminus R\big)\cap B\big(f_0(R),\,16r\big) \big|\ \geqslant\ (1-\beta)\,v\,.
\end{equation}
We now turn this function into a~$C$-toppling procedure~$f$ by adding the rule that that the distinguished vertex is toppled in priority: thus, for every configuration~$R\in\calP(C)\setminus\{\varnothing\}$, we set
$$f(R)
\ =\ \begin{cases}
\dist_C & \text{ if }\dist_C\in R\,,\\
f_0(R) & \text{ otherwise.}
\end{cases}$$
Namely, we topple the distinguished particle if it is awaken, else if at least a fraction~$1-\beta$ is sleeping we topple a particle surrounded by many sleeping particles, and otherwise we topple whatever active particle.
Let us consider the \guillemets{boundary} set:
\begin{equation}
\label{defPartialB}
\partial\calB\ =\ \Big\{\,R\subset A\setminus\{\dist_C\}\quad :\quad 
|R\cap C|=\big\lfloor\beta|C|\big\rfloor
\,\Big\}
\,.
\end{equation}
We claim that~$\Proba(R_\tau\in\partial\calB)=1$, that is to say, when we exit from~$\calB$ we necessarily arrive in~$\partial\calB$.
Indeed, one can only exit from~$\calB$ when a particle falls asleep.
If this site~$x$ which falls asleep when exiting from~$\calB$ is not~$\dist_C$, it implies that~$\dist_C$ was already sleeping, because otherwise~$\dist_C$ would have been toppled instead of~$x$ (recall that the distinguished vertex has priority over all other sites).
Therefore, when we exit from~$\calB$, the distinguished vertex is always sleeping.

Thus, to prove the result, we may consider the stabilization of~$C$ starting from a deterministic initial configuration~$R_0\in\partial\calB$ and, for every~$t\in\N$, we write
$$(R_t,\,h_t,\,\ell_t)
\ =\ \big(\Phi_{C}\big)^{(t)}(R_0,\,0,\,0)\,,$$
overriding the notation of the statement, and we prove that, starting from the initial configuration, one plus the number of visits of~$\calB$ dominates a geometric variable.
The toppling procedure defines a Markov chain on the state space~$\calP(A)$, and we consider the two following stopping times of this Markov chain:
$$T_\calB\ =\ \inf\big\{t\in\N\ :\ R_t\in\calB\big\}
\qquadet
T_{\text{sleep}}\ =\ \inf\big\{t\in\N\ :\ R_t\cap C=\varnothing\big\}\,.$$
Let us now show that the process
$$M_t\ =\ \mathbf{1}_{\{t<T_\calB\}}
\,e^{-\alpha N_t}
\qquadou
N_t
\ =\ \big|R_t\cap C\setminus\{\dist_C\}\big|$$
is a supermartingale with respect to the filtration~$(\calF_t)_{t\in\N}$ generated by~$(R_t)_{t\in\N}$.

First, note that if~$R_t\in\calB$ or~$R_t\cap C=\varnothing$, then we have~$M_{t+1}= M_t$.
Besides, recall that, as soon as the distinguished vertex~$\dist_C$ is active, it is toppled in priority.
Yet, the distinguished vertex is not counted in~$N_t$, and it cannot wake up anyone in~$C$, because~$w(\dist_C,\,j)\cap C=\varnothing$ for all~$j\in\N$, following our definition~\reff{defW} of~$w$ (which indicates which sites can be awaken by the loops of colour~$j$).
Hence, if~$\dist_C\in R_t$ then we necessarily have~$N_{t+1}= N_t$.
What's more, since the distinguished vertex cannot wake up anyone in~$C$, if~$\dist_C\in R_t$ and~$R_t\notin\calB$ then we still have~$R_{t+1}\notin\calB$.
Thus, we proved the implication
\begin{equation}
\label{MtStable}
\big\{R_t\in\calB\big\}
\cup\big\{R_t\cap C=\varnothing\big\}
\cup\big\{\dist_C\in R_t\big\}
\ \subset\ 
\big\{M_{t+1}=M_t\big\}\,.
\end{equation}
We now assume that~$R_t\notin\calB$,~$R_t\cap C\neq\varnothing$ and~$\dist_C\notin R_t$, and we write~$x=f(R_t)=f_0(R_t)$, which is the next site to be toppled.
Recall that~\smash{$I\big(x,\,h_t(x)\big)$} is the Bernoulli variable which decides if~$x$ falls asleep or performs a loop, in which case this loop covers the set~\smash{$\Gamma\big(x,\,\ell_t(x),\,j\big)$} with~$j=J\big(x,\,\ell_t(x)\big)$.
Recalling that, since~$x\neq\dist_C$, we have~$w(x,\,j)=C_0(x)=C$, we can write
\begin{align*}
N_{t+1}-N_t
&\ =\ 
-I(x,\,h_t(x))
+\big(1-I(x,\,h_t(x))\big)
\times
\big|\Gamma\big(x,\,\ell_t(x),\,j\big)\cap\big(C\setminus (R_t\cup\{\dist_C\})\big)\big|\\
&\ \geqslant\ 
-I(x,\,h_t(x))
+\big(1-I(x,\,h_t(x))\big)
\sum_{y\,\in\, B(x,\,16r)\,\cap\, C\setminus (R_t\cup\{\dist_C\})}
\mathbf{1}_{\{y\in\Gamma(x,\,\ell_t(x),\,j)\}}\,.
\end{align*}
Recall now that the bound~\reff{expCCaround} ensures that the above sum contains at least~$\Ceil{(1-\beta)v-1}$ terms.
Besides, each of these terms dominates a Bernoulli variable with parameter~$\Upsilon_d(16r)$, where~$\Upsilon_d$ is the function which was defined in~\reff{defUpsilon}.
Even though these variables may not be independent, we have the following
result whose elementary proof is deferred to the appendix~\ref{sectionProofLemmaSumBernoulli}.

\begin{lemma}
\label{lemmaSumBernoulli}
Let~$n\in\N$ and let~$X_1,\,\ldots,\,X_n$ be Bernoulli random variables with parameter~$p\in[0,1]$ (non necessarily independent).
Then, for every~$c>0$, we have
$$\Esp\Big[e^{-c(X_1+\cdots+X_n)}\Big]
\ \leqslant\ 1-p+pe^{-cn}\,.$$
\end{lemma}

Using this result, we deduce that, on the event~\smash{$\big\{R_t\notin\calB\big\}\cap\big\{R_t\cap C\neq\varnothing\big\}\cap\big\{\dist_C\notin R_t\big\}$}, we have
\begin{align*}
\Esp\big(M_{t+1}\,\big|\,\calF_t\big)
&\ \leqslant\ 
M_t\times
\Esp\big(e^{-\alpha(N_{t+1}-N_t)}\,\big|\,\calF_t\big)\\
&\ \leqslant\ 
M_t\,
\left(
\frac{\lambda}{1+\lambda}e^{\alpha}
+\frac{1-\Upsilon_d(16r)
\big(1-e^{\alpha(1-(1-\beta)v)}\big)}
{1+\lambda}
\right)\\
&\ =\ 
M_t\,
\left(
1+
\frac{\lambda\big(e^{\alpha}-1\big)
-\Upsilon_d(16r)
\big(1-e^{\alpha(1-(1-\beta)v)}\big)}
{1+\lambda}
\right)
\ \leqslant\ M_t\,,
\end{align*}
using our assumption~\reff{conditionMetastability}.
Recalling~\reff{MtStable}, we deduce that~$(M_t)$ is a supermartingale with respect to~$(\calF_t)$.
Hence, using Doob's Theorem and using that~$R_0\in\partial\calB$, we obtain that, for every~$m>0$,
$$\Esp\big[M_{T_\calB\wedge T_\text{sleep}\wedge m}\big]
\ \leqslant\ M_0
\ =\ \exp\big(-\alpha|R_0\cap C|\big)
\ =\ e^{-\alpha\Ent{\beta|C|}}\,,$$
which implies by Fatou's Lemma that
$$\Esp\big[M_{T_\calB\wedge T_\text{sleep}}\big]
\ =\ \Esp\bigg[\liminf\limits_{m\to\infty}M_{T_\calB\wedge T_\text{sleep}\wedge m}\bigg]
\ \leqslant\ \liminf\limits_{m\to\infty}\,
\Esp\big[M_{T_\calB\wedge T_\text{sleep}\wedge m}\big]
\ \leqslant\ e^{-\alpha\Ent{\beta|C|}}\,,$$
whence
$$\Proba\big(T_\text{sleep}<T_\calB\big)
\ =\ \Esp\big[M_{T_\calB\wedge T_\text{sleep}}\big]
\ \leqslant\ e^{-\alpha\Ent{\beta|C|}}\,.$$
This being true for all starting configurations~$R_0\in\partial\calB$, the result follows.
\end{proof}

\subsection{Concluding proof of Lemma~\ref{lemmaInit2Dhigh}}

Let~$\lambda>1$, and let~$r$ and~$\alpha_0$ be defined by~\reff{paramsInit}.
Let~$D\in(\N\setminus\{0\})^\N$, let~$n\geqslant 1$, let~$A\subset\Z_n^2$ and assume that~$(A_j,\,\calC_j)_{j\leqslant {\calJ}}$ is a~$(r^2,\,D)$-dormitory hierarchy on~$A$ such that every set~$C\in\calC_0$ is~$8r$-connected and satisfies~\reff{conditionDense}.

We wish to apply Lemma~\ref{lemmaMetastability} with~$\alpha=3\alpha_0/2$,~$\beta=5/6$ and~$v=r^2$.
To do so, we need to check that the condition~\reff{conditionMetastability} is satisfied, at least for~$\lambda$ large enough.
Replacing~$\alpha_0$ and~$r$ by their expressions given by~\reff{paramsInit}, we have
$$e^{\alpha(1-(1-\beta)v)}
\ =\ \exp\Bigg[\frac{3\alpha_0}{2}\left(1-\frac{r^2}{6}\right)\Bigg]
\ \leqslant\ \exp\Bigg[\frac{3K}{2\lambda\ln\lambda}\left(1-\frac{64\lambda(\ln\lambda)^2}{6K}\right)\Bigg]
\ \stackrel{\lambda\to\infty}{\sim}\ 
\frac{1}{\lambda^{16}}\,,$$
whence~$e^{\alpha(1-(1-\beta)v)}\leqslant 1/8$ for~$\lambda$ large enough.
Combining this with the lower bound on~$\Upsilon_2$ 
given by Lemma~\ref{lemmaHarnack}, we deduce that, for~$\lambda$ large enough,
\begin{equation}
\label{checkingConditionMetastability}
\lambda\big(e^{\alpha}-1\big)
-\Upsilon_2(16r)
\big(1-e^{\alpha(1-(1-\beta)v)}\big)
\ \leqslant\ 
\lambda\left[\exp\left(\frac{3K}{2\lambda\ln\lambda}\right)-1\right]
-\frac{(7/8)K}{\ln r+\ln 16}\,.
\end{equation}
Using now that~$\ln r\sim(\ln\lambda)/2$ when~$\lambda\to\infty$, we may write
$$\lambda\left[\exp\left(\frac{3K}{2\lambda\ln\lambda}\right)-1\right]
-\frac{(7/8)K}{\ln r+\ln 16}
\ \stackrel{\lambda\to\infty}{\sim}\ 
\frac{3K}{2\ln\lambda}-\frac{7K}{4\ln\lambda}
\ =\ -\frac{K}{4\ln\lambda}
\ <\ 0\,.$$
Plugging this into~\reff{checkingConditionMetastability}, we deduce that the condition~\reff{conditionMetastability} necessary to apply Lemma~\ref{lemmaMetastability} is satisfied for~$\lambda$ large enough.
Thus, using this Lemma we deduce that for every~$C\in\calC_0$ there exists a~$C$-toppling procedure~$f$ such that~$\calN_{\calB}$, the number of visits of the set~$\calB$ defined in the statement of Lemma~\ref{lemmaMetastability}, dominates a geometric variable of parameter
$$\exp\left(-\frac{3\alpha_0}{2}\Ent{\frac{5|C|}{6}}\right)\,.$$
Recalling that~$|C|\geqslant r^2$ with~$r\to\infty$ when~$\lambda\to\infty$, we have~$\Ent{5|C|/6}\geqslant 3|C|/4$ provided that that~$\lambda$ is chosen large enough so that~$r^2\geqslant 12$.
Thus, for~$\lambda$ large enough, for every~$C\in\calC_0$, the variable~$\calN_\calB$ dominates a geometric random variable with parameter~\smash{$\exp\big(-9\alpha_0|C|/8\big)$}.

Now it remains to deduce a lower bound on the number of sleeps performed by the distinguished vertex before stabilization.
As in the proof of Lemma~\ref{lemmaMetastability}, we consider the process starting from a fixed initial configuration~$R_0\in\partial\calB$, with~$\partial\calB$ given by~\reff{defPartialB}, so that the number of visits of~$\calB$ from this initial configuration is~$\calN_\calB-1$.
Let us introduce two new random times: first,
$$T_{\star}
\ =\ \inf\big\{t\in\N\ :\ \dist_C\in R_t\big\}\,,$$
which is the first time when the distinguished vertex is awaken (recall that~$\dist_C\notin R_0$ because we start from~\smash{$R_0\in\partial\calB$}), and, for~$t\in\N$,
$$T_{\calB}^t
\ =\ \inf\big\{t'\geqslant t\ :\ R_{t'}\in\calB\big\}\,.$$
We also define a deterministic time~$t_0=\Ent{2|C|\ln|C|/\ln\lambda}$, and we now look for a lower bound on the probability of the event that, starting from~$R_0$, we wake up the distinguished particle before~$t_0$ and, after doing this, we visit~$\calB$ before stabilization, that is to say,
\begin{equation}
\label{defE}
\calE
\ =\ \Big\{T_{\star}\leqslant t_0\Big\}
\cap\Big\{T_{\calB}^{T_{\star}}<T_{\text{sleep}}\Big\}\,.
\end{equation}
To this end, we write
\begin{align*}
\Proba\big(\calE^c\big)
&\ =\ 
\Proba\Big(\big\{T_{\star}> t_0\big\}
\cup\big\{T_{\calB}^{T_{\star}}\geqslant T_{\text{sleep}}\big\}\Big)
\ =\ 
\Proba\Big(\big\{T_{\star}> t_0\big\}
\cup\big\{T_{\calB}^{T_{\star}}=\infty\big\}\Big)\\
&\ \leqslant\ 
\Proba\Big(\big\{T_{\star}> t_0\big\}
\cup\big\{T_{\calB}^{t_0}=\infty\big\}\Big)\\
&\ \leqslant\ 
\Proba\Big(T_{\text{sleep}}\leqslant t_0\Big)
\,+\,\Proba\Big(\big\{T_{\star}> t_0\big\}\cap\big\{t_0<T_{\text{sleep}}\big\}\Big)\\
&\qquad+\Proba\Big(\big\{T_{\calB}^{t_0}=\infty\big\}
\cap\big\{t_0<T_{\text{sleep}}\big\}\Big)\,.
\numberthis\label{threeProbas2D}
\end{align*}
We now control each of these terms.
To deal with the first term, we note that~$T_{\text{sleep}}\geqslant \calN_{\calB}$, which we proved to dominate a geometric variable with parameter~\smash{$\exp\big(-9\alpha_0|C|/8\big)$}, whence
$$\Proba\Big(T_{\text{sleep}}\leqslant t_0\Big)
\ \leqslant\ 
\Proba\Big(\calN_{\calB}\,\leqslant\,t_0\Big)
\ \leqslant\ t_0\,\exp\left(-\frac{9\alpha_0|C|}{8}\right)\,.$$
To deal with the third term, we note that, if~$\calN_\calB> t_0$, that is to say, we have at least~$t_0$ returns to~$\calB$, then at least one of these returns must occur after time~$t_0$, whence
$$\Proba\Big(\big\{T_{\calB}^{t_0}=\infty\big\}
\cap\big\{t_0<T_{\text{sleep}}\big\}\Big)
\ \leqslant\ \Proba\Big(\calN_{\calB}\leqslant t_0\Big)
\ \leqslant\ t_0\,\exp\left(-\frac{9\alpha_0|C|}{8}\right)\,.$$
To deal with the second term, we note that,~$C$ being~$8r$-connected, we have~$\diam\,C\leqslant 8r|C|$.
Thus, as long as the distinguished vertex sleeps, at each step, the probability to wake it up is at least~\smash{$\Upsilon_2\big(8r|C|\big)/(1+\lambda)$}.
Therefore, we have
$$\Proba\Big(\big\{T_{\star}> t_0\big\}\cap\big\{t_0<T_{\text{sleep}}\big\}\Big)
\ \leqslant\ \left(1-\frac{\Upsilon_2\big(8r|C|\big)}{1+\lambda}\right)^{t_0}\,.$$
Plugging these three bounds into~\reff{threeProbas2D} and using the lower bound on~$\Upsilon_2$ given by Lemma~\ref{lemmaHarnack}, we get
\begin{align*}
\Proba\big(\calE^c\big)
&\ \leqslant\ 2t_0\,e^{-9\alpha_0|C|/8}
\,+\,\left(1-\frac{\Upsilon_2\big(8r|C|\big)}{1+\lambda}\right)^{t_0}\\
&\ \leqslant\ \frac{4|C|\ln|C|}{\ln\lambda}\,
e^{-9\alpha_0|C|/8}
\,+\,\exp\left(-\frac{2|C|\ln|C|}{\ln\lambda}
\times\frac{1}{1+\lambda}
\times\frac{K}{\ln 8+\ln r+\ln|C|}\right)\\
&\ \leqslant\ \frac{4|C|^2}{\ln\lambda}\,
e^{-9\alpha_0|C|/8}
\,+\,\exp\left(-\frac{4K|C|\ln r}{(1+\lambda)\ln\lambda(3\ln 2+3\ln r)}\right)\\
&\ =\ \frac{e^{-\alpha_0|C|}}{2\lambda}
\times\Bigg[\frac{8\lambda|C|^2}{\ln\lambda}\,
e^{-\alpha_0|C|/8}
\,+\,2\lambda\exp\left(\alpha_0|C|-\frac{4K|C|\ln r}{3(1+\lambda)\ln\lambda(\ln 2+\ln r)}\right)
\Bigg]\,.
\end{align*}
We now wish to show that, for~$\lambda$ large enough, for every~$|C|\geqslant r^2$, the quantity between the brackets is smaller than~$1$.
Defining
$$M
\ =\ \sup_{x>0}\big(x^2e^{-x/16}\big)\,,$$
we may write
\begin{equation}
\label{probaEcomp}
\frac{2\lambda\,\Proba\big(\calE^c\big)}{e^{-\alpha_0|C|}}
\ \leqslant\ 
\frac{8M\lambda^3\ln\lambda}{K^2}\,
e^{-\alpha_0 r^2/16}
\,+\,2\lambda\exp\left[\bigg(\frac{K}{\lambda\ln\lambda}-\frac{4K\ln r}{3(1+\lambda)\ln\lambda\big(\ln 2+\ln r\big)}\bigg)|C|\,\right]\,.
\end{equation}
On the one hand, we have
\begin{equation}
\label{probaEcomp1}
\frac{8M\lambda^3\ln\lambda}{K^2}\,
e^{-\alpha_0 r^2/16}
\ \geqslant\ \frac{8M\lambda^3\ln\lambda}{K^2}
e^{-4\ln\lambda}
\ =\ \frac{8M\ln\lambda}{K^2\lambda}
\ \stackrel{\lambda\to\infty}{\longrightarrow}\ 0\,.
\end{equation}
On the other hand, we have
$$\frac{K}{\lambda\ln\lambda}
-\frac{4K\ln r}{3(1+\lambda)\ln\lambda\big(\ln 2+\ln r\big)}
\ \stackrel{\lambda\to\infty}{\sim}\ 
\frac{K}{\lambda\ln\lambda}-
\frac{4K}{3\lambda\ln\lambda}
\ =\ -\frac{K}{3\lambda\ln\lambda}\,,$$
whence, for~$\lambda$ large enough,
$$\frac{K}{\lambda\ln\lambda}
-\frac{4K\ln r}{3(1+\lambda)\ln\lambda\big(\ln 2+\ln r\big)}
\ \leqslant\ -\frac{K}{6\lambda\ln\lambda}\,,$$
implying that, for~$|C|\geqslant r^2$,
\begin{align*}
2\lambda
\exp\left[\bigg(\frac{K}{\lambda\ln\lambda}
-\frac{4K\ln r}{3(1+\lambda)\ln\lambda\big(\ln 2+\ln r\big)}\bigg)|C|\,\right]
&\ \leqslant\ 2\lambda
\exp\left(-\frac{K}{6\lambda\ln\lambda}\times\frac{64\lambda(\ln\lambda)^2}{K}\right)\\
&\ =\ \frac{2}{\lambda^{29/3}}
\ \stackrel{\lambda\to\infty}{\longrightarrow}\ 0\,.
\numberthis\label{probaEcomp2}
\end{align*}
Plugging~\reff{probaEcomp1} and~\reff{probaEcomp2} into~\reff{probaEcomp}, we deduce that, for~$\lambda$ large enough, for every~$|C|\geqslant r^2$, we have
$$\Proba\big(\calE^c\big)
\ \leqslant\ \frac{1}{2\lambda}\,\exp\big(-\alpha_0|C|\big)\,.$$
Recalling the definition~\reff{defE} of the event~$\calE$, note that, if~$\calE$ is realized, it means that, starting from~$R_0\in\partial\calB$, the distinguished particle is waken up, after which we visit~$\calB$ and, sooner or later, we reach~$\partial\calB$ again, which implies that the distinguished particle fell asleep.
Since our estimate on~$\calE$ is valid uniformly over all the initial configurations~$R_0\in\partial\calB$, we have a renewal sequence, and we deduce that~$\calS(C)$, which denotes the number of~$\dist_C$-sleeps during the stabilization of~$C$, dominates a geometric random variable with parameter~\smash{$\exp\big(-\alpha_0|C|\big)/(2\lambda)$}.

To conclude our proof, there only remains to translate this result into a lower bound on the number of loops of colour~$0$.
Following Lemma~\ref{lemmaNextColour}, we have the equality in distribution
$$\calL(C,\,0)
\ \stackrel{\text{d}}{=}\ \sum_{i=1}^{\calS(C)}(X_i-1)\,,$$
where the variables~$(X_i)$ are i.i.d.\ geometric variables with parameter~$b=2\lambda/(2\lambda+1)$ and are independent of~$\calS(C)$.
Using now Lemma~\ref{lemmaSumGeometrics}, we deduce that~$1+\calL(C,\,0)$ dominates a geometric variable with parameter
$$\frac{b\exp\big(-\alpha_0|C|\big)/(2\lambda)}
{1-b+b\exp\big(-\alpha_0|C|\big)/(2\lambda)}
\ =\ \frac{\exp\big(-\alpha_0|C|\big)}
{1+\exp\big(-\alpha_0|C|\big)}
\ \leqslant\ \exp\big(-\alpha_0|C|\big)\,,$$
concluding the proof of~$\calP(0)$.
\hfill\qed

\section{Inductive step: proof of Lemma~\ref{lemmaInduction}}
\label{sectionInduction}

Let~$d,\,\lambda,\,v,\,(D_j),\,(\alpha_j),\,n,\,A,\,(A_j,\,\calC_j)$ be as in the statement.
We assume that~$0\leqslant j< {\calJ}$ is such that the property~$\calP(j)$ holds, and we wish to establish the property~$\calP(j+1)$, that is to say, we need to pass from a statement about the loops of colour~$j$ produced by the clusters of~$\calC_j$ to a statement about the loops of colour~$j+1$ produced by the clusters of~$\calC_{j+1}$, with the constant~$\alpha_{j}$ being replaced with~$\alpha_{j+1}$.

We start by translating our aim into a statement about the number of loops of colour~\smash{$j$}.
Namely, we show that, to establish~$\calP(j+1)$, it is enough to prove the stochastic domination:
\begin{equation}
\label{goalInduction}
\forall\,C\in\calC_{j+1}\qquad
1+\calL(C,\,j)
\ \succeq\ \mathrm{Geom}(a)
\quadou
a\ =\ \frac{1}{2^{j+2}(1+\lambda)}\,
\exp\big(-\alpha_{j+1}|C|\big)\,.
\end{equation}
Indeed, it follows from Lemma~\ref{lemmaNextColour} that
$$\forall\,C\in\calC_{j+1}\qquad
1+\calL(C,\,j+1)
\ \stackrel{d}{=}\ 
\sum_{i=1}^{\calS(C)+\calL(C,\,0)+\cdots+\calL(C,\,j)}(X_i-1)
\ \geqslant\ 
\sum_{i=1}^{1+\calL(C,\,j)}
(X_i-1)\,,$$
where the variables~$(X_i)_{i\in\N}$ are i.i.d.\ geometric variables with parameter
$$b\ =\ \frac{2^{j+2}(1+\lambda)-2}{2^{j+2}(1+\lambda)-1}\,.$$
Thus, using Lemma~\ref{lemmaSumGeometrics} about the sum of a geometric number of geometrics minus one, the statement~\reff{goalInduction} would imply that, for every~$C\in\calC_{j+1}$, the variable~$1+\calL(C,\,j+1)$ dominates a geometric variable with parameter
$$\frac{ab}{1-b+ab}
\ \leqslant\ \frac{ab}{1-b}
\ =\ \big(2^{j+2}(1+\lambda)-2\big)\,a
\ \leqslant\ 2^{j+2}(1+\lambda)\,a
\ =\ \exp\big(-\alpha_{j+1}|C|\big)\,,$$
which is precisely~$\calP(j+1)$.
Thus, there only remains to prove~\reff{goalInduction}.

We now turn to the study of the number of loops of colour~$j$ produced by a cluster~$C\in\calC_{j+1}$.
We start by distinguishing between two cases.
First, if~$C\in\calC_{j}$ then, since~$|C|\geqslant 2^{\Ent{(j+1)/2}}v\geqslant 2^{j/2}v$, the claim~\reff{goalInduction} follows from our assumption~\reff{conditionInduction} on the sequence~$(\alpha_j)_{j\in\N}$, which entails that
$$\exp\big(-\alpha_{j}|C|\big)
\ \leqslant\ \exp\Big[-(\alpha_{j}-\alpha_{j+1})2^{j/2}v-\alpha_{j+1}|C|\Big]
\ \leqslant\ \frac{\exp\big(-\alpha_{j+1}|C|\big)}{2^{j+2}(1+\lambda)}
\ =\ a\,.$$

Hence, we now assume that~$C\in\calC_{j+1}\setminus\calC_{j}$.
In this case, following property~\textit{(\ref{conditionMerge})} of the hierarchy, we have~$\diam\,C\leqslant D_{j}$ and we can write~$C=C_0\cup C_1$ with~$C_0,\,C_1\in\calC_{j}$.
To simplify the notation, we simply write~$\dist_0=\dist_{C_0}$ and~$\dist_1=\dist_{C_1}$.
Upon exchanging the names~$C_0$ and~$C_1$, we may assume that~$|C_0|\geqslant|C_1|$ and~$\dist_C=\dist_0$ (recall that, as explained in section~\ref{subsectionDormitory}, the distinguished vertex of the union of two clusters is the distinguished vertex of the largest of the two clusters).
For~$k\in\{0,1\}$, let us define~\smash{$q_k=\exp\big(-\alpha_{j} |C_k|\big)$}, so that our induction hypothesis~$\calP(j)$ tells us that~$1+\calL(C_0,\,j)$ and~$1+\calL(C_1,\,j)$ respectively dominate geometric random variables of parameter~$q_0$ and~$q_1$.

\subsection{Some notation}
Recall the ping-pong mechanism described in section~\ref{subsectionStrategy}.
To shorten notation, we write
$$\varepsilon\,:\,i\in \N\ \longmapsto\ 
i \text{ mod. } 2
\ =\ 
\begin{cases}
0 & \text{ if } i \text{ is even,}\\
1 & \text{ otherwise.}
\end{cases}$$
Thus, the stabilization of~$C=C_0\cup C_1$ is composed of a series of stabilizations where, during the~$i$-th stabilization (the numbering starting at~$0$), we stabilize the set~$C_{\varepsilon(i)}$.
Let us write~\smash{$R_0=C$} (since the configuration out of~$C$ has no impact on the stabilization of~$C$, we could as well take~$R_0=A$) and~$h_0=\ell_0=0$, and let us define recursively, for every~$i\in\N$,
\begin{equation}
\label{defProcessInduction}
(R_{i+1},\,h_{i+1},\,\ell_{i+1})
\ =\ \mathrm{Stab}_{C_{\varepsilon(i)}}(R_i,\,h_i,\,\ell_i)\,,
\end{equation}
where~$\mathrm{Stab}$ is the stabilization operator defined by~\reff{defStab}.
Note that if~$i\in\N$ is such that~\smash{$R_{i}\cap C=\varnothing$} (that is to say, both sets are stabilized), then the above definition implies that~\smash{$(R_j,\,h_j,\,\ell_j)=(R_i,\,h_i,\,\ell_i)$} for all~$j>i$.
For every~$i\in\N$, we denote by~$\calL_i$ the the number of loops of colour~$j$ produced by the distinguished vertex~$\dist_{\varepsilon(i)}$ during the~$i$-th stabilization.
Formally, for every~$i\in\N$, we have
$$\calL_i
\ =\ \abs{\Big\{\,
\ell\in\big\{\ell_{i}\big(\dist_{\varepsilon(i)}\big),\,\ldots,\,\ell_{i+1}\big(\dist_{\varepsilon(i)}\big)-1\big\}
\ :\ 
J\big(\dist_{\varepsilon(i)},\,\ell\big)
\,=\,j
\,\Big\}}\,.$$
Note that the variables~$(\calL_i)_{i\in\N}$ are not independent because, for example, the second stabilization of~$C_0$ (which produces~$\calL_2$ loops at~$\dist_0$) depends on the configuration which results from the previous stabilizations, in particular it depends on the sites of~$C_0$ which have been reactivated during the stabilization of~$C_1$.

\subsection{Number of good stabilizations}
We now define the random number
\begin{equation}
\label{defN}
\calN
\ =\ \inf\Big\{\,i\geqslant 1\ :\ 
C_0\centernot\subset R_{2i}
\quador
C_1\centernot\subset R_{2i+1}
\,\Big\}\,.
\end{equation}

When~$1\leqslant i<2\calN-1$, after the~$i$-th stabilization (during which the set~$C_{\varepsilon(i)}$ is stabilized), the other set~$C_{\varepsilon(i+1)}$ is fully active.
This means that the loops of colour~$j$ produced by the distinguished vertex~\smash{$\dist_{\varepsilon(i)}$} during the~$i$-th stabilization have entirely covered the other set (recall that only these loops are allowed to wake up the sites in the other set, as explained in section~\ref{subsectionColouredLoops}).

Thus, each of the first~$\calN$ stabilizations of~$C_0$ and~$C_1$ starts with the corresponding set fully active.
Note that we do not require anything on the first stabilization of~$C_0$ because, in any case, since~$C_1$ is already fully active at the beginning of the procedure, it is still active after the first stabilization of~$C_0$, so that this first stabilization of~$C_0$ does not need to wake up~$C_1$ for the mechanism of reciprocal activation to be able to go on.

\subsection{The induction relation}
The key inequality which yields the induction property is
\begin{equation}
\label{key-induction}
\calL(C,\,j)
\ =\ \sum_{i=0}^{+\infty}\calL_{2i}
\ \geqslant\ \sum_{i=0}^{\calN-1}
\calL_{2i}\,.
\end{equation}
Note that we only count the loops produced during the even steps because we are only interested in the number of loops emerging from~$\dist_0$.

\subsection{Our aim: a geometric sum of i.i.d.\ geometric variables}
We will show later that~$\calN$ is a geometric variable, and we want to show that the sum in~\reff{key-induction} dominates a geometric variable.
If the variables~$1+\calL_{2i}$ were i.i.d.\ geometric variables independent of~$\calN$, we could conclude using Lemma~\ref{lemmaSumGeometrics} that~$1+\calL(C,\,j)$ dominates a geometric variable, whose parameter has an explicit expression in the two parameters.

A first issue is that the variables~$1+\calL_{2i}$ are not i.i.d.\ geometric variables, and this family not even dominates a family of i.i.d.\ geometrics.
Even though it follows from the induction hypothesis that~$1+\calL_0$ dominates a geometric with parameter~$q_0$, this is not the case of the subsequent variables, because for~$i\geqslant 2\calN$ some stabilizations may start from a set which is not fully active.
In fact with probability~$1$ we even have~$\calL_i=0$ for~$i$ large enough.
To solve this problem, we show below that we can replace this sequence~$\calL$ with another sequence~$\calL'$ which does dominate a sequence of independent geometrics.

\subsection{An infinite \guillemets{Sisyphus} sequence}
It is convenient to consider another infinite sequence~\smash{$\calL'=(\calL'_i)_{i\in\N}$} which is defined similarly, with the only difference that, after each stabilization, we proceed with the next stabilization assuming that all sites of the other set have been awaken.
Formally, we start with the same initial parameters~$R'_0=R_0=C$,~$h'_0=h_0=0$ and~$\ell'_0=\ell_0=0$ and we define recursively, for~$i\in\N$,
$$\big(R'_{i+1},\,h'_{i+1},\,\ell'_{i+1}\big)
\ =\ \mathrm{Stab}_{C_{\varepsilon(i)}}\big(C_{\varepsilon(i)},\,h'_i,\,\ell'_i\big)\,.$$
Note that the only difference with respect to the previous definition~\reff{defProcessInduction} is that as first argument of the stabilization operator~$\mathrm{Stab}$ we have~$C_{\varepsilon(i)}$ instead of~$R_i$, which means that we start the~$i$-th stabilization with the set~$C_{\varepsilon(i)}$ fully active, regardless of what happened during the previous stabilizations.
We then also define~$\calL'_i$ to be the number of loops of colour~$j$ produced by the distinguished vertex of~$C_{\varepsilon(i)}$ during the~$i$-th stabilization, that is to say,
$$\calL'_i
\ =\ \abs{\Big\{\,
\ell\in\big\{\ell'_{i}\big(\dist_{\varepsilon(i)}\big),\,\ldots,\,\ell'_{i+1}\big(\dist_{\varepsilon(i)}\big)-1\big\}
\ :\ 
J\big(\dist_{\varepsilon(i)},\,\ell\big)
\,=\,j
\,\Big\}}\,,$$
and we define~$\calN'$ exactly as~$\calN$ was defined in~\reff{defN} but replacing~$R$ with~$R'$.
One can then notice that, by definition of~$\calN$, for every~$i<2\calN$, we have~$\calL'_i=\calL_i$ and~$R'_{i+1}=R_{i+1}$.

We claim that we also have~$\calN'=\calN$.
On the one hand, if~$C_0\centernot\subset R_{2\calN}= R'_{2\calN}$ then by definition we have~$\calN'=\calN$.
Otherwise, if~$C_0\subset  R_{2\calN}= R'_{2\calN}$, then we also have~$\calL'_{2\calN}=\calL_{2\calN}$ whence~\smash{$R'_{2\calN+1}= R_{2\calN+1}\centernot\supset C_1$}, which also implies that~$\calN'=\calN$.
Therefore, in both cases we have~$\calN'=\calN$.

We may now rewrite the induction relation~\reff{key-induction} as
\begin{equation}
\label{key-induction-bis}
\calL(C,\,j)
\ \geqslant\ \sum_{i=0}^{\calN-1}
\calL_{2i}
\ =\ \sum_{i=0}^{\calN'-1}
\calL'_{2i}\,.
\end{equation}

\subsection{The dependency we want to get rid of}
With~\reff{key-induction-bis}, the problem is now much simplified, because the variables~$(\calL'_{i})$ are independent, and for every~$i\in\N$, we have that~$1+\calL'_{2i}$ is distributed as~$1+\calL(C_0,\,j)$ which, by the induction hypothesis, dominates a geometric random variable with parameter~$q_0$.

However, this sequence~$\calL'$ is not independent of~$\calN'$.
Yet, this dependency can be controlled, and we show below that~$\calL'$ and~$\calN'$ are positively correlated.

\subsection{The key independence property thanks to the coloured loops}
An important property of this new sequence~$\calL'$ is that it is independent of the loops of colour~$j$ produced by the two distinguished vertices:
\begin{equation}
\label{indep}
(\calL'_i)_{i\in\N}
\quad\perp\quad
\big(\Gamma(\dist_k,\,\ell,\,j)\big)_{k\in\{0,1\},\,\ell\in\N}
\end{equation}
because, among the loops emitted by~$\dist_0$ and~$\dist_1$, only those of colour~$0,\,\ldots,\,j-1$ can affect the internal stabilizations of~$C_0$ and~$C_1$, while the loops of colour~$j$ emitted by~$\dist_0$ are only used to wake up sites of~$C_1$ during the stabilizations of~$C_0$, and vice versa.
Thus, since this new sequence~$\calL'$ is constructed precisely by ignoring which sites of the other cluster are reactivated during each stabilization, we have the above independence property.

This means that, conditioned on~$\calL'$, we know the \textit{number} of loops of colour~$j$ that can reactivate the other set, but we have no information about the \textit{shapes} of these loops of colour~$j$, so that these loops still behave as i.i.d.\ excursions.
This property would not have hold if we had not coloured the loops.

\subsection{Why distinguished vertices ?}

Along with the above independence property~\reff{indep} which follows from our colouring of the loops, it is also important that the loops devoted to reactivating~$C_1$ (resp.~$C_0$) all have the same starting point~$\dist_0$ (resp.~$\dist_1$).
Indeed, it is crucial to have in each set a distinguished vertex which is the only one able to wake up the other set.

Otherwise, if we had coloured all the loops but with all the vertices allowed to wake up the other set (that is to say, if the loops emanating from any point~$x\in C_0$ could reactivate the sites in~$C_1$ if and only if they are of colour~$j$), then it would be harder to control the correlation between~$\calL'_{2i}$ and the probability that~$C_1$ is fully reactivated by the loops emerging from~$C_0$ during the~$i$-th stabilization of~$C_0$, because we would not know how the starting points of these loops are distributed.

Hence our choice to select one distinguished vertex in each set and to devote the job of reactivating the other set only to the loops emerging from this vertex.
Doing so, conditioned on~$\calL'$, we know how many loops we have at our disposal to wake up the other set, and we know exactly where these loops start, while we have no information about the shapes of these loops. Using these ingredients, we now show that there is a positive correlation between~$\calL'$ and~$\calN'$.

\subsection{Positive correlation}
We now decompose~\reff{key-induction-bis} over the different possible values of~$\calN'$, writing
\begin{multline*}
\PlA\big(
\calL(C,\,j)
\,\geqslant\,m\big)
\ \geqslant\ 
\sum_{k=1}^{+\infty}\,
\PlA\Bigg(\Bigg\{\sum_{i=0}^{k-1}\calL'_{2i}\,\geqslant\,m\Bigg\}\cap\big\{\calN'=k\big\}\Bigg)\\
\ =\ 
\sum_{k=1}^{+\infty}\,
\PlA\Bigg(
\Bigg\{\sum_{i=0}^{k-1}\calL'_{2i}\,\geqslant\,m\Bigg\}
\cap\big\{\calN'\geqslant k\big\}
\cap\big\{C_0\centernot\subset  R'_{2k}\ \text{or}\  C_1\centernot\subset  R'_{2k+1}\big\}
\Bigg)\,.
\end{multline*}
In the above three events, the two first only depend on~$\calL'_0,\,\ldots,\,\calL'_{2k-2}$ and~$R'_2,\,\ldots,\, R'_{2k-1}$, and are therefore independent of the third event, leading to
\begin{equation}
\label{positiveCorrelation}
\PlA\big(
\calL(C,\,j)
\,\geqslant\,m\big)
\ \geqslant\ 
\sum_{k=1}^{+\infty}\,
\PlA\Bigg(
\Bigg\{\sum_{i=0}^{k-1}\calL'_{2i}\,\geqslant\,m\Bigg\}
\cap\big\{\calN'\geqslant k\big\}
\Bigg)
\,\PlA\big(C_0\centernot\subset  R'_{2k}\ \text{or}\  C_1\centernot\subset  R'_{2k+1}\big)\,.
\end{equation}
To show that the two first events are positively correlated, we write
\begin{align*}
\PlA\big(\,\calN'\geqslant k\ \big|\ \calL'\,\big)
&\ =\ \PlA\Bigg(\,
\bigcap_{1\,\leqslant\,i\,<\,k}
\big\{C_0\subset  R'_{2i}\big\}
\cap
\big\{C_1\subset  R'_{2i+1}\big\}
\ \Bigg|\ \calL'
\,\Bigg)\\
&\ =\ 
\PlA\Bigg(\,
\bigcap_{1\,\leqslant\,i\,\leqslant\,2k-2}
\big\{C_{\varepsilon(i+1)}\subset U_i\big\}
\ \Bigg|\ \calL'
\,\Bigg)\,,\numberthis\label{probaIntersection}
\end{align*}
where the set~$U_i$ is given by
$$U_i
\ =\ \bigcup_{
\substack{
\ell'_i(\dist_{\varepsilon(i)})\,\leqslant\,\ell\,<\,\ell'_{i+1}(\dist_{\varepsilon(i)})\\
J(\dist_{\varepsilon(i)},\,\ell)\,=\,j}}
\Gamma\big(\dist_{\varepsilon(i)},\,\ell,\,j\big)\,.$$
Now note that, following the key independence property~\reff{indep}, the loops involved in the above union are independent of~$\calL'$, while the number of loops involved in the union is by definition~$\calL'_i$.
Thus, conditioned on~$\calL'$, the events in the intersection in~\reff{probaIntersection} are independent and we simply have
\begin{align*}
\PlA\big(\,\calN'\geqslant k\ \big|\ \calL'\,\big)
&\ =\ \prod_{1\,\leqslant\,i\,\leqslant\,2k-2}
\PlA\Bigg(\,
C_{\varepsilon(i+1)}\subset
\bigcup_{0\,\leqslant\,\ell\,<\,\calL'_i}
\Gamma\big(\dist_{\varepsilon(i)},\,\ell,\,j\big)
\ \Bigg|\ \calL'
\,\Bigg)\\
&\ =\ \prod_{1\,\leqslant\,i\,\leqslant\,2k-2}
\psi_{\varepsilon(i)}\big(\calL'_i\big)\,,\numberthis\label{productPsi}
\end{align*}
where the function~$\psi_0$ is defined by
$$\psi_0\,:\,x\in\N\ \longmapsto\ \PlA\Bigg(\,
C_1\subset
\bigcup_{0\,\leqslant\,\ell\,<\,x}
\Gamma\big(\dist_0,\,\ell,\,j\big)
\,\Bigg)$$
and~$\psi_1$ is defined similarly, replacing~$C_1$ with~$C_0$ and~$\dist_0$ with~$\dist_1$.
Clearly, these functions are increasing, implying that the two events~$\{\calN'\geqslant k\}$ and~$\{\sum_{i=0}^{k-1}\calL'_{2i}\geqslant m\}$ are positively correlated.
Thus, coming back to~\reff{positiveCorrelation}, we deduce that
\begin{align*}
\PlA\big(
\calL(C,\,j)
\,\geqslant\,m\big)
&\ \geqslant\ 
\sum_{k=1}^{+\infty}\,
\PlA\Bigg(
\sum_{i=0}^{k-1}\calL'_{2i}\,\geqslant\,m
\Bigg)
\,\PlA\big(\calN'\geqslant k\big)
\,\PlA\big(C_0\centernot\subset  R'_{2k}\ \text{or}\  C_1\centernot\subset  R'_{2k+1}\big)\\
&\ =\ 
\sum_{k=1}^{+\infty}\,
\PlA\Bigg(\sum_{i=0}^{k-1}\calL'_{2i}\,\geqslant\,m\Bigg)
\,\PlA\big(\calN'= k\big)
\ =\ \PlA\Bigg(\sum_{i=0}^{\calN''-1}\calL'_{2i}\,\geqslant\,m\Bigg)\,,
\end{align*}
where~$\calN''$ is a copy of~$\calN'$ which is independent of~$\calL'$ (with a slight abuse of notation, we keep the notation~$\PlA$ for the new probability measure).
Now notice that if follows from our computation~\reff{productPsi} that~$\calN'$ (and hence also~$\calN''$) is a geometric random variable with parameter~$1-p_0p_1$, where
\begin{equation}
\label{defpk}
p_0
\ =\ \ElA\big[\psi_0(\calL'_0)\big]
\qquadet
p_1
\ =\ \ElA\big[\psi_1(\calL'_1)\big]\,.
\end{equation}
At this point, recalling that, for every~$i\in\N$,~$1+\calL'_{i}$ is distributed as~\smash{$1+\calL(C_{\varepsilon(i)},\,j)$}, which by the induction hypothesis dominates a geometric variable of parameter~$q_{\varepsilon(i)}$, and using Lemma~\ref{lemmaSumGeometrics} about the sum of a geometric number of geometric variables (minus one), we deduce that~$1+\calL(C,\,j)$ dominates a geometric variable with parameter
\begin{equation}
\label{defQprime}
q'
\ =\ \frac{(1-p_0p_1)q_0}{1-p_0p_1q_0}
\ \leqslant\ \frac{(1-p_0p_1)q_0}{1-q_0}\,.
\end{equation}

\subsection{Bound on the parameter of the geometric variable}

We now estimate~$p_0$ and~$p_1$ introduced in~\reff{defpk}.
For every~$x\in\N$, we have
\begin{align*}
1-\psi_0(x)
&\ =\ \PlA\Bigg(\,
C_1\centernot\subset
\bigcup_{0\,\leqslant\,\ell\,<\,x}
\Gamma\big(\dist_0,\,\ell,\,j\big)
\,\Bigg)
\ \leqslant\ 
\sum_{y\in C_1} \PlA\Bigg(\,y\notin
\bigcup_{0\,\leqslant\,\ell\,<\,x}
\Gamma\big(\dist_0,\,\ell,\,j\big)
\,\Bigg)\\
&\ =\ \sum_{y\in C_1}
\PlA\Big(\,y\notin\Gamma(\dist_0,\,0,\,j)\,
\Big)^x
\ \leqslant\ |C_1|\big(1-\Upsilon_d(D_{j})\big)^x\,,
\end{align*}
using the decreasing function~$\Upsilon_d$ defined by~\reff{defUpsilon}, 
and the fact that~$\diam\,C\leqslant D_{j}$.
Note that, in the one-dimensional case, the above estimate could be easily improved because, in dimension~$1$, if a loop visits both ends of a cluster, then it wakes up everyone in this cluster.

In order to estimate~$p_0$, we replace~$x$ with~$\calL'_0$ and recall once again that~$1+\calL'_0$ dominates a geometric random variable with parameter~$q_0$, yielding
\begin{align*}
1-p_0
&\ =\ \Esp\big[1-\psi_0(\calL'_0)\big]
\ \leqslant\ 
\sum_{m=0}^{+\infty}
q_0(1-q_0)^{m}\big(1-\psi_0(m)\big)\\
&\ \leqslant\ 
|C_1| q_0
\,\sum_{m=0}^{+\infty}
(1-q_0)^{m}\big(1-\Upsilon_d(D_{j})\big)^m
\ =\ 
\frac{|C_1| q_0}{1-(1-q_0)\big(1-\Upsilon_d(D_{j})\big)}\\
&\ =\ \frac{|C_1| q_0}{\Upsilon_d(D_{j})+q_0\big(1-\Upsilon_d(D_{j})\big)}
\ \leqslant\ \frac{|C_1| q_0}{\Upsilon_d(D_{j})}\,.
\end{align*}
A similar bound holds for~$p_1$, replacing~$|C_1|q_0$ with~$|C_0|q_1$.
Using these bounds, we deduce that the parameter~$q'$ defined in~\reff{defQprime} is bounded by
$$q'\ \leqslant\ \frac{\big[(1-p_0)+p_0(1-p_1)\big]q_0}{1-q_0}
\ \leqslant\ \frac{\big(|C_1| q_0+|C_0|q_1\big)q_0}{(1-q_0)\Upsilon_d(D_{j})}
\ \leqslant\ \frac{|C|\,q_1q_0}{(1-q_0)\Upsilon_d(D_{j})}\,,$$
where in the last inequality we used that~$q_0\leqslant q_1$.
Noting now that
$$q_0q_1\ =\ \exp\big(-\alpha_{j}|C_0|-\alpha_{j}|C_1|\big)
\ =\ \exp\big(-\alpha_{j}|C|\big)\,,$$
this becomes
$$q'
\ \leqslant\ 
\frac{|C|\,\exp\big(-\alpha_{j}|C|\big)}
{(1-q_0)\Upsilon_d(D_{j})}
\ \leqslant\ \frac{|C|\,\exp\big(-(\alpha_{j}-\alpha_{j+1})|C|\big)}
{\big(1-e^{-\alpha_{j} v}\big)\Upsilon_d(D_{j})}
\,\exp\big(-\alpha_{j+1}|C|\big)\,.$$
Writing~$u=\alpha_{j}-\alpha_{j+1}$ and noting that the function~$z:x\mapsto x\,e^{-ux}$ is decreasing on~$[1/u,\,\infty)$ and that the property~\textit{(\ref{conditionCard})} of the hierarchy together with our assumption~\reff{conditionInduction} on~$(\alpha_j)$ ensure that
$$|C|
\ \geqslant\ 2^{\Ent{(j+1)/2}} v
\ \geqslant\ 2^{j/2} v
\ \geqslant\ \frac{1}{u}\,,$$
we deduce that~\smash{$z\big(|C|\big)\leqslant z\big(2^{j/2} v\big)$}, whence
$$q'
\ \leqslant\ 
\frac{2^{j/2}v
\,\exp\big(
-(\alpha_{j}-\alpha_{j+1})2^{j/2}v\big)}
{\big(1-e^{-\alpha_{j} v}\big)\Upsilon_d(D_{j})}
\,\exp\big(-\alpha_{j+1}|C|\big)\,.$$
Using our assumption~\reff{conditionInduction}, 
we deduce that~\smash{$q'\leqslant\exp\big(
-\alpha_{j+1}|C|\big)/\big(2^{j+2}(1+\lambda)\big)$}, which proves our goal estimate~\reff{goalInduction}, concluding the inductive step.
\hfill\qed

\begin{appendix}
\section*{Proofs of the Technical Lemmas}

\subsection{Sufficient conditions for activity: proof of Lemmas~\ref{lemmaConditionActivityTopplings} and~\ref{lemmaConditionActivityLoops}}
\label{sectionProofLemmaCondition}

We start with the proof of the sufficient condition formulated in terms of the number of topplings, which is a mere combination of ingredients of~\cite{FG22}:

\begin{proof}[Proof of Lemma~\ref{lemmaConditionActivityTopplings}]
Let~$d\geqslant 1$,~$\lambda>0$,~$\mu\in(0,1)$ and let~$a>0$,~$b>\psi(\mu)$ and~$n_0\geqslant 1$ be such that~\reff{conditionMA} holds for every~$n\geqslant n_0$ and every~$A\subset\torus$ with~$|A|=\Ceil{\mu n^d}$.
Let~\smash{$n\geqslant n_0$} and let us fix an initial configuration~$\eta:\torus\to\N$ such that~$|\eta|=\sum_{x\in\torus}\eta(x)=\Ceil{\mu n^d}$.
Writing~$M$ for the number of topplings necessary to stabilize, recalling that~$M_A$ denotes the number of topplings on the sites of~$A$, and writing~$A_{\text{final}}$ for the random set where the particles eventually settle, we can write, as in section~3.8 of~\cite{FG22}:
\begin{align*}
\calP^\lambda_\mu\Big(M\,\leqslant\,e^{a n^d}
\ \Big|\ \eta_0=\eta\Big)
&\ =\ \sum_{|A|=\lceil\mu n^d\rceil}
\calP^\lambda_\mu\Big(
\big\{M\,\leqslant\,e^{a n^d}\big\}
\cap
\big\{A_{\text{final}}=A\big\}
\ \Big|\ \eta_0=\eta
\Big)\\
&\ \leqslant\ \sum_{|A|=\lceil\mu n^d\rceil}
\calP^{\lambda,A}_\mu\Big(
M_A\,\leqslant\,e^{a n^d}
\ \Big|\ \eta_0=\eta
\Big)\\
&\ \leqslant\ \sum_{|A|=\lceil\mu n^d\rceil}
\calP^{\lambda,A}_\mu\Big(
M_{A}\,\leqslant\,e^{a n^d}
\ \Big|\ \eta_0=\mathbf{1}_A
\Big)
\ \leqslant\ \binom{n^d}{\lceil\mu n^d\rceil}
\,e^{-bn^d}
\,,
\end{align*}
where we used Lemma~5 of~\cite{FG22} in the first inequality,~Lemma~6 of~\cite{FG22} in the second inequality, and our assumption~\reff{conditionMA} in the last inequality.
Then, using the fact that
$$\binom{n^d}{\lceil\mu n^d\rceil}
\ =\ O\Big(e^{\psi(\mu)n^d}\Big)\,,$$
which follows from Stirling's formula, we deduce that
$$\sup_{\eta\in\N^{\torus}\,:\,|\eta|=\lceil\mu n^d\rceil}
\calP^\lambda_\mu\Big(M\,\leqslant\,e^{a n^d}
\ \Big|\ \eta_0=\eta\Big)
\ =\ O\Big(e^{-(b-\psi(\mu))n^d}\Big)\,,$$
implying that, taking~\smash{$0<c<\mathrm{min}\big(a,\,b-\psi(\mu)\big)$}, for~$n$ large enough, we have
$$\sup_{\eta\in\N^{\torus}\,:\,|\eta|=\lceil\mu n^d\rceil}
\calP^\lambda_\mu\Big(M\,\leqslant\,e^{c n^d}
\ \Big|\ \eta_0=\eta\Big)
\ <\ e^{-cn^d}\,,$$
which, by the monotonicity of the number of topplings with respect to the initial configuration, implies that
$$\calP^\lambda_\mu\Big(M\,\leqslant\,e^{c n^d}
\ \Big|\ |\eta_0|\geqslant\mu n^d\Big)
\ <\ e^{-cn^d}\,.$$
We then proceed as in sections~3.9 and~3.10 of~\cite{FG22}, dealing with the Poisson initial distribution and with the exponential toppling clocks, to deduce that, for every~$\mu'>\mu$, there exists~$c'>0$ such that the fixation time~$\calT_n$ of the continuous-time ARW model on~$\torus$ satisfies
$$\forall n\geqslant 1\qquad
\calP_{\mu'}^\lambda\Big(\calT_n\,\leqslant\,e^{c'n^d}\Big)
\ \leqslant\ e^{-c'n^d}\,,$$
which, by Theorem~4 of~\cite{FG22}, implies that~$\mu'\geqslant\mu_c(\lambda)$.
This being true for every~$\mu'>\mu$, we deduce that~$\mu\geqslant\mu_c(\lambda)$.
\end{proof}

We now translate this sufficient condition in the context of our loop representation, where some waking up events are ignored.

\begin{proof}[Proof of Lemma~\ref{lemmaConditionActivityLoops}]
Let~$d\geqslant 1$,~$\lambda>0$,~$\mu\in(0,1)$ and~$\kappa$ as in the statement.
Let us recall that our representation described in section~\ref{subsectionLoopRep}, which consists in storing a pile of loops above each vertex, differs from the classical site-wise representation of the ARW model in that the number of topplings necessary to stabilize is not \textit{abelian}: it depends on the order with which the sites are toppled, as explained in section~\ref{subsectionConditionLoops}.
Yet, it follows from the construction of our toppling strategy that the two representations can be coupled in the following way.

Let~$n$ large enough, let~$A\subset\torus$ with~\smash{$|A|=\Ceil{\mu n^d}$}, and let~$(A_j,\,\calC_j)_{j\leqslant {\calJ}}$ be a dormitory hierarchy on~$A$ and~$(f_C)_{C\in\calC_0}$ a collection of toppling procedures such that the assumption~\reff{goalStochDom} holds.

Draw a stack of toppling instructions~\smash{$\tau=(\tau(x,\,h))_{x\in \torus,\,h\in\N}$} distributed according to~$\calP^{\lambda,A}$, as defined in~\cite{FG22}, that is to say, for every~$x\in A$ and~$h\in\N$, the instruction~$\tau(x,\,h)$ is a sleep instruction with probability~$\lambda/(1+\lambda)$ and otherwise it is a jump instruction to one of the neighbours of~$x$ chosen uniformly, while if~$x\in \torus\setminus A$ and~$h\in\N$, the instruction~$\tau(x,\,h)$ cannot be a sleep and is a jump instruction to a uniform neighbour of~$x$, the instructions being independent.
Independently, for every~$x\in A$ and every~$\ell\in\N$, we draw an independent variable~$1+J(x,\,\ell)$ which is geometric with parameter~$1/2$.

Then, follow our toppling strategy to stabilize the set~$A_{\calJ}$ as described in section~\ref{subsectionStrategy}, but reading the instructions from~$\tau$ and the colours from~$J$, and storing the instructions used in a stack of sleep instructions~$I$ and a stack of loops~$\Gamma$, so that~$I(x,\,h)\in\{0,1\}$ indicates whether the~$h$-th step at~$x$ is a sleep or a loop while~$\Gamma(x,\,\ell,\,j)\subset\torus$ is the support of the~$\ell$-th loop produced at~$x$, where~$j=J(x,\,\ell)$ is the colour of this loop.
At the end of the stabilization of~$A_{\calJ}$, we complete the obtained stack~$I$ with~i.i.d.\ Bernoulli variables with parameter~\smash{$\lambda/(1+\lambda)$} and the stack~$\Gamma$ with independent supports of excursions.
Since the choice of the next site to topple using a toppling procedure (as defined in section~\ref{subsectionLoopRep}) only depends on the present configuration of the model (and not for example on the instructions which have not been used yet), the stacks~$(I,\,J,\,\Gamma)$ that we obtain have the distribution described in section~\ref{subsectionLoopRep}.
Thus, we obtain a coupling between, on the one hand, the site-wise representation of the ARW model, constructed with the stack of toppling instructions (jumps or sleeps) and with no sleeps out of~$A$, that we call henceforth the~$A$-ARW model, and our model constructed with the stacks of sleeps, of colours and of loops, in which we ignore some reactivation events depending on the colours of the loops.

Along this construction, one may execute the toppling instructions in the two models in parallel.
Recall that the two models do not only differ by the notation but also in the fact that, in our model with loops, we ignore some reactivations (while both models share the constraint that particles cannot sleep out of~$A$).
Thus, with the coupling described above, the two models do not follow the same evolution.
However, one can show that, at any time of the procedure, if a site of~$A$ is active in the loop model, then it is also active in the~$A$-ARW model.
Therefore, our toppling strategy with coloured loops yields a legal toppling sequence (as defined in~\cite{Rolla20}) in the~$A$-ARW model, meaning that, at the end of the stabilization of~$A_{\calJ}$, it might be that some sites of~$A_{\calJ}$ are not sleeping in the~$A$-ARW model, but at least we can guarantee that we only toppled sites which were active in both models.
In a certain sense, our model stabilizes faster because we ignore some reactivation events.

Thus, we have the stochastic domination~$M_A\succeq\calH(A_{\calJ})$.
Therefore, our assumption~\reff{goalStochDom} implies that~$M_A$ also dominates a geometric variable with parameter~$\exp(-\kappa n^d)$.
Now, to conclude by applying Lemma~\ref{lemmaConditionActivityTopplings}, there only remains to see that, if~$X$ is a geometric variable with parameter~$\exp(-\kappa n^d)$, then, choosing~$b$ such that~$\psi(\mu)<b<\kappa$ and~$a>0$ such that~\smash{$a<\kappa-b$}, we have
\begin{multline*}
\Proba\left(X\leqslant e^{an^d}\right)
\ =\ \Proba\left(X\leqslant\big\lfloor e^{an^d}\big\rfloor\right)
\ =\ 1-\big(1-e^{-\kappa n^d}\big)^{\lfloor e^{an^d}\rfloor}\\
\ =\ 1-\exp\left(\big\lfloor e^{an^d}\big\rfloor\ln\big(1-e^{-\kappa n^d}\big)\right)
\ \leqslant\ -\big\lfloor e^{an^d}\big\rfloor\ln\big(1-e^{-\kappa n^d}\big)
\ \eqninfty\ e^{-(\kappa-a)n^d}\,,
\end{multline*}
implying that, for~$n$ large enough, we have
$$\calP^\lambda_A\Big(\,M_A\,\leqslant\,e^{an^d}\,\Big)
\ \leqslant\ \PlA\Big(\,\calH(A_{\calJ})\,\leqslant\,e^{an^d}\,\Big)
\ \leqslant\ e^{-bn^d}\,,$$
which allows us to deduce that~$\mu\geqslant\mu_c(\lambda)$ by virtue of Lemma~\ref{lemmaConditionActivityTopplings}.
\end{proof}

\subsection{Correlation between loop colours: proof of Lemma~\ref{lemmaNextColour}}
\label{sectionProofLemmaNextColour}

\begin{proof}[Proof of Lemma~\ref{lemmaNextColour}]
Let~$d,\,n\geqslant 1$ and let~$A\subset\torus$ equipped with~$(A_j,\,\calC_j)_{j\leqslant \calJ}$ a dormitory hierarchy on~$A$ and toppling procedures~$(f_C)_{C\in\calC_0}$.

Let us describe an alternative construction of our model with exponential clocks behind the Bernoulli and geometric variables~$I(x,\,h)_{x,\,h}$ and~$J(x,\,\ell)_{x,\,\ell}$ defined in section~\ref{subsectionLoopRep}.
For every~$x\in A$ and every~$h\in\N$, we consider an exponential variable~$T_\s(x,\,h)$ with parameter~$p_\s=\lambda$ and, for every~$j\in\N$, an exponential variable~$T_j(x,\,h)$ with parameter~$p_j=2^{-(j+1)}$, all these exponential clocks being independent.
Then, for every~$x\in A$ and~$h\in\N$, writing
$$T_{\mathrm{min}}(x,\,h)
\ =\ \mathrm{min}\Big(T_\s(x,\,h),\ 
\inf_{j\in\N}\,T_j(x,\,h)\Big)\,,$$
we define, for~$x\in A$,~$h\in\N$ and~$\ell\in\N$,
$$I(x,\,h)
\ =\ \mathbf{1}_{\{T_{\mathrm{min}}(x,\,h)\,=\,T_\s(x,\,h)\}}
\qquadet
J(x,\,\ell)
\ =\ \mathbf{1}_{\{T_{\mathrm{min}}(x,\,h(\ell))\,=\,T_j(x,\,h(\ell))\}}\,,$$
where~$h(\ell)$ is the index of the~$\ell$-th zero in the sequence~$I(x,\,h)_{h\in\N}$, that is to say:
$$\left\{\begin{aligned}
&h(0)
\ =\ \min\{h\geqslant 0\,:\,I(x,\,h)=0\}\,,\\
&\forall\,\ell\geqslant 1\,,\ 
h(\ell)
\ =\ \mathrm{min}\big\{h>h(\ell-1)\,:\,I(x,\,h)=0\big\}\,.
\end{aligned}\right.$$

With~$I$ and~$J$ defined in this way, the~$h$-th toppling of the site~$x$ is a sleep if the clock~$T_\s(x,\,h)$ rings before all the clocks~\smash{$T_j(x,\,h),\,j\in\N$}, which happens with probability~$\lambda/(1+\lambda)$, and it is a loop of colour~$j$ if the first of these clocks to ring is~$T_j(x,\,h)$, which happens with probability~\smash{$1/(2^{j+1}(1+\lambda))$}.
Thus, the variables~$I(x,\,h)_{x,\,h}$ defined above are i.i.d.\ Bernoulli variables with parameter~$\lambda/(1+\lambda)$ and the variables~$J(x,\,\ell)_{x,\,\ell}$ are i.i.d.\ geometric variables with parameter~$1/2$, independent of~$I(x,\,h)_{x,\,h}$, that is to say, they are distributed as in the construction of our model described in section~\ref{subsectionLoopRep}.
Thus, to prove the result we may assume that the model is constructed with these exponential clocks.

Now let~$j\leqslant\calJ$ and~$C\in\calC_j$, and recall that the loops of colour~$j$ which are performed by the distinguished vertex~$\dist_C$ are not allowed to wake up anyone in~$C$: when~$\dist_C$ produces a loop of colour~$j$, the configuration inside~$C$ is not affected. 
Thus, since our toppling strategy to stabilize~$C$ only looks at the configuration inside~$C$ (recalling the definition of the toppling procedures in section~\ref{subsectionLoopRep}), this implies that, after a~$\dist_C$-loop of colour~$j$, this distinguished vertex~$\dist_C$ is toppled again, as if this loop had not occurred.

Now, call the~$\dist_C$-sleeps and the~$\dist_C$-loops of colour~$0,\,\ldots,\,j-1$ the {\it useful} topplings of~$\dist_C$,
and call {\it useless} topplings of~$\dist_C$ the~$\dist_C$-loops of colour at least~$j$, and let us write~\smash{$\calT=\calS(C)+\calL(C,\,0)+\cdots+\calL(C,\,j-1)$} for the total number of useful topplings performed by~$\dist_C$ during the stabilization of~$C$.
As the useless topplings do not impact the stabilization of~$C$, the sequence of useful topplings of~$\dist_C$ performed during the stabilization of~$C$ is independent of the useless topplings which are performed between any two useful topplings.
Thus, the variable~$\calT=\calS(C)+\calL(C,\,0)+\cdots+\calL(C,\,j-1)$ is independent of the sequence of exponential clocks~\smash{$\big(T_j(\dist_C,\,h)\big)_{h\in\N}$}.

Let us now define inductively~$h_0=0$ and, for every~$i\geqslant 1$,
$$h_{i}
\ =\ \min\Big\{h>h_{i-1}\ :\ 
T_{\mathrm{min}}(\dist_C,\,h)\,\neq\,T_k(\dist_C,\,h)
\text{ for every }k\geqslant j\Big\}\,,$$
so that for every~$i\in\{1,\,\ldots,\,\calT\}$,~$h_i$ is the number of useless topplings of~$\dist_C$ which are produced between the~$(i-1)$-th and the~$i$-th useful toppling of~$\dist_C$.
Let us write~$X_i-1$ for the number of loops of colour~$j$ among these useless topplings.
Namely, for every~$i\geqslant 1$, we write
$$X_i
\ =\ 1+\abs{\Big\{h\in\{h_{i-1}+1,\,\ldots,\,h_i-1\}\ :\ 
T_{\mathrm{min}}(\dist_C,\,h)\,=\,T_j(\dist_C,\,h)\Big\}}\,.$$
With these variables, we have the relation~\reff{sumGeomLoops}, the variables~$(X_i)_{i\geqslant 1}$ are i.i.d., and~$X_1-1$ is distributed as the number of times that a Poisson point process with intensity~$p_j$ fires before an exponential clock of parameter~$p_\s+p_0+\cdots+p_{j-1}$ rings.
Thus,~$X_1$ is a geometric variable with parameter given by the computation~\reff{computationParamXi}.
Also, it follows from the above considerations that~$\calT$ is independent of~$(X_i)_{i\geqslant 1}$, since the loops of colour~$j$ are irrelevant for the stabilization of~$C$.
Therefore, the variables~$X_i$ are independent of~$\calT$, and the result follows.
\end{proof}

\subsection{An elementary property of geometric sums: proof of Lemma~\ref{lemmaSumGeometrics}}
\label{sectionProofLemmaSumGeometrics}

\begin{proof}[Proof of Lemma~\ref{lemmaSumGeometrics}]
Consider the following heads and tails experiment made with two coins,
a big one and a small one, which give head with probability $a$ and $b$, respectively.

Toss the small coin up to getting a head, then toss the big coin.
If the big coin gives head, then stop there.
Otherwise, if the big coin gives tail, restart the experiment: toss the small coin again until getting head, then toss the big coin, stop if it gives head, and so on.

Let~$N$ be the total number of times that the big coin is tossed during the experiment.
The law of~$N$ is geometric with parameter~$a$.
Hence, the law of the total number of small tails is a sum of~$N$ independent geometrics minus one, each with parameter~$b$.

The experiment may be seen as a sequence of independent patterns, the three possible patterns being: a small tail ; a small head followed by a big tail ; a small head followed by a big head.
These three patterns occur with probability~$1-b$,~$b(1-a)$ and~$ba$, respectively.
Then, the total number of small tails is the number of occurrences of the first pattern before obtaining the third pattern (so that occurrences of the second pattern may be ignored).
Hence, this number is a geometric minus one, with parameter~$ba/(1-b+ba)$.
\end{proof}

\subsection{Hitting probabilities on the torus: proof of Lemma~\ref{lemmaHarnack}}
\label{sectionProofLemmaHarnack}

\begin{proof}
The case~$d=1$ is simply the gambler's ruin estimate, see for example Proposition~5.1.1 of~\cite{LL10}.
Assume now that~$d\geqslant 2$ (in fact the proof below also works in dimension~$1$, but it does not give the explicit constant~$1/2$).

The key ingredient of our proof is Harnack's principle, as stated in Theorem~6.3.9 of~\cite{LL10}.
Let us consider the open set~$U$ and the compact set~$K$ given by
$$U\ =\ \big(-1,\,1\big)^d
\setminus \left[-\frac{1}{2},\,\frac{1}{2}\right]^d
\quadet
K\ =\ \left[-\frac{5}{6},\,\frac{5}{6}\right]^d
\setminus \left(-\frac{2}{3},\,\frac{2}{3}\right)^d\,.$$
Harnack's principle tells us that there exists a constant~$C>0$ and an integer~$r_0\geqslant 1$ such that, for every~$r\geqslant r_0$, writing~$U_r=r\,U\cap\Z^d$ and~$K_r=r\,K\cap\Z^d$, for every function~\smash{$f:\overline{U_r}\to[0,\,\infty)$} which is harmonic on~$U_r$, we have~$f(x)\geqslant C f(y)$ for every~$x,\,y\in K_r$ (where~$\overline{U_r}$ denotes the set of vertices of~$\Z^d$ which are in~$U_r$ or have at least one neighbour in~$U_r$).

For every~$r\geqslant 1$, let us write
$$u_r\ =\ P_0\big(T_{\partial\Lambda_r}<T_0^+\big)\,,$$
which is the probability that a simple symmetric random walk on~$\Z^d$ started at the origin exits from the box~$\Lambda_r=\{-r,\,\ldots,\,r\}$ before returning to the origin.

We now let~$n\geqslant 1$ and~$x,\,y\in\torus$ with~$x\neq y$, and let~$r=d(x,\,y)$.
Recall that we work with the \guillemets{infinite-norm distance} on the torus, defined by~\reff{defDistance}.
By definition of this distance, writing~\smash{$\pi_n:\Z^d\to\torus$} for a standard projection, we may take~$a,\,b\in\Z^d$ with~$\norme{a-b}_\infty=r$ such that~$\pi_n(a)=x$ and~$\pi_n(b)=y$.
Note also that we always have~$r\leqslant n/2$.

Upon decreasing the constant~$K$ which appears in the result of the Lemma, we may assume that~$r\geqslant 6\vee r_0$ and that~$a-b$ has all its coordinates even, so that~$m=(a+b)/2\in\Z^d$.

We then consider the function
$$f\,:\,z\in\overline{U_r}\ \longmapsto\ 
P_{\pi_n(m+z)}\big(T_y<T_x\big)\,,$$
where the random walk considered is on the torus~$\torus$.
This function~$f$ is harmonic on~$U_r$ because
$$(m+U_r)
\,\cap\,\big[(a+n\Z^d)\cup(b+n\Z^d)\big]
\ =\ \varnothing\,.$$
Therefore, Harnack's principle ensures that~$\inf f(K_r)\geqslant C\,\sup f(K_r)$ (note that~\smash{$K_r\neq\varnothing$} because we assumed that~$r\geqslant 6$).
Yet, since the two points~$x$ and~$y$ play symmetric roles and since~$K_r$ is symmetric, we have~$f(-z)=1-f(z)$ for every~$z\in K_r$, whence~\smash{$\sup f(K_r)\geqslant 1/2$}.
Thus, we deduce that~$\inf f(K_r)\geqslant C/2$.

The result then follows by writing
$$P_x\big(T_y<T_x^+\big)
\ \geqslant\ P_x\big(T_{\pi_n(m+K_r)}<T_x^+\big)\times
\inf f(K_r)
\ \geqslant\ \frac{C}{2}\,P_a\big(T_{m+K_r}<T_a^+\big)
\ \geqslant\ \frac{C u_{2r}}{2}$$
and by using the classical lower bounds on~$u_r$ (see for example section~4.6 of~\cite{LL10}).
\end{proof}

\subsection{Pairings in graphs: proof of Lemma~\ref{lemmaGroup2or3}}
\label{sectionProofLemmaGroup2or3}

Lemma~\ref{lemmaGroup2or3} is a corollary of the following Lemma:

\begin{lemma}
For every finite connected undirected graph~$G=(V,\,E)$ with~$|V|\geqslant 2$, there exists a partition of~$V$ into sets of cardinality~$2$ or~$3$ and diameter (for the graph distance on~$G$) at most~$2$.
\end{lemma}

\begin{proof}
We proceed by induction on~$|V|$.
The result is obvious if~$|V|=2$ or~$|V|=3$.
We now let~$n\geqslant 4$ and we assume that the result is true for all the connected undirected graphs containing between~$2$ and~$n-1$ vertices.
Let~$G=(V,\,E)$ be a finite connected undirected graph with~$|V|=n$.

The graph~$G$ being connected, we can consider a rooted covering tree of~$G$.
Let~$x\in V$ be a leaf of this covering tree with maximal distance to the root.

If this leaf has a \guillemets{sister}~$y$ (i.e., another leaf which has the same parent), then the induced subgraph on~$V\setminus\{x,\,y\}$ is still connected.
Applying the induction hypothesis to this subgraph yields a partition~$\Pi$ of~$V\setminus\{x,\,y\}$, and then~\smash{$\Pi\cup\big\{\{x,\,y\}\big\}$} is a solution of the problem, since the graph distance on~$G$ between~$x$ and~$y$ is at most~$2$ (because it equals~$2$ on the covering tree).

Otherwise, if the leaf~$x$ has no sister, then we consider its parent~$z$, and, noticing that the induced subgraph on~$V\setminus\{x,\,z\}$ remains connected, we can apply the induction hypothesis to this subgraph.
\end{proof}

\subsection{Sums of dependent Bernoulli variables: proof of Lemma~\ref{lemmaSumBernoulli}}
\label{sectionProofLemmaSumBernoulli}

\begin{proof}
Let~$p\in[0,1]$ and~$c>0$.
We proceed by induction on~$n$.
The result is obvious if~$n=0$.
Assume that the result is valid for~$n\in\N$, let~$X_1,\,\ldots,\,X_{n+1}$ be Bernoulli random variables with parameter~$p$, and let us write~$Y_n=e^{-c(X_1+\cdots+X_n)}$.
Conditioning on the last variable, we have
\begin{align*}
\Esp\Big[&e^{-c(X_1+\cdots+X_{n+1})}\Big]
\ =\ pe^{-c}\,
\Esp\big[Y_n\,\big|\,X_{n+1}=1\big]
+(1-p)\,
\Esp\big[Y_n\,\big|\,X_{n+1}=0\big]\\
&\ =\ p\,
\Esp\big[Y_n\,\big|\,X_{n+1}=1\big]
+(1-p)\,
\Esp\big[Y_n\,\big|\,X_{n+1}=0\big]
-p\big(1-e^{-c}\big)\,\Esp\big[Y_n\,\big|\,X_{n+1}=1\big]\\
&\ =\ \Esp\big[Y_n\big]
-p\big(1-e^{-c}\big)\Esp\big[Y_n\,\big|\,X_{n+1}=1\big]\\
&\ \leqslant\ \Esp\big[Y_n\big]
-p\big(1-e^{-c}\big)e^{-cn}
\end{align*}
Plugging the induction hypothesis in the above formula then yields the desired result.
\end{proof}

\end{appendix}
\begin{acks}[Acknowledgments]
AG and NF thank Vittoria Silvestri for useful discussions
at the Ypatia 2022 conference. They thank the IRL Lysm for support
at this occasion.
\end{acks}

\begin{funding}
NF thanks the German Research Foundation (project number 444084038,
priority program SPP 2265) for financial support.
AA's research was supported by public grants overseen by the
French National Research Agency, ANR SWiWS (ANR-17-CE40-0032-02).
\end{funding}

\bibliographystyle{imsart-number}
\bibliography{arw.bib}
\end{document}